\documentclass[12pt]{amsart}
\usepackage{amsmath, amssymb, amsthm, mathrsfs, mathtools, textcomp, colonequals,
  comment,soul}
\usepackage{epic}
\usepackage[shortlabels]{enumitem}
\usepackage{xypic}
\usepackage[backref]{hyperref}
\usepackage[alphabetic,backrefs,lite]{amsrefs}
\usepackage[usenames,dvipsnames]{color}
\usepackage{fullpage}

\hypersetup{
  colorlinks   = true, 
  urlcolor     = blue, 
  linkcolor    = blue, 
  citecolor   = blue 
}

\numberwithin{equation}{section}
\allowdisplaybreaks[1]
\newtheorem{theorem}[equation]{Theorem}

\newtheorem{corollary}[equation]{Corollary}
\newtheorem{prop}[equation]{Proposition}
\newtheorem{lemma}[equation]{Lemma}

\theoremstyle{definition}

\newtheorem{remark}[equation]{Remark}
\newtheorem{notation}[equation]{Notation}
\newtheorem{defn}[equation]{Definition}

\setenumerate{itemsep=2pt}

\newcommand{\Q}{\mathbb Q}

\newcommand{\Z}{\mathbb Z}

\newcommand{\n}{\mathfrak{n}}
\newcommand{\N}{\mathbb N}

\renewcommand{\O}{\mathcal{O}}
\newcommand{\A}{\mathbb A}

\renewcommand{\P}{\mathbb P}

\renewcommand{\phi}{\varphi}

\newcommand{\Spec}{\mathrm{Spec} \ }

\renewcommand{\c}[1]{\overline{#1}}

\renewcommand{\l}[2]{\tx{length}_{\O_K}(\mathcal{O}_{#1}/\mathcal{O}_{#2})}

\newcommand{\m}{\mathfrak{m}}
\newcommand{\W}{\mathcal{W}}
\newcommand{\X}{\mathcal{X}}
\newcommand{\Y}{\mathcal{Y}}
\newcommand{\tx}[1]{\text{#1}}
\let\L\relax

\DeclareMathOperator{\Gal}{Gal}

\DeclareMathOperator{\divi}{div}
\DeclareMathOperator{\disc}{disc}

\DeclareMathOperator{\Art}{Art}
\DeclareMathOperator{\Aut}{Aut}

\DeclareMathOperator{\lcm}{lcm}

\DeclareMathOperator{\rad}{rad}
\DeclareMathOperator{\red}{red}

\DeclareMathOperator{\blf}{bf}
\DeclareMathOperator{\mbf}{mbf}
\DeclareMathOperator{\mult}{mult}
\DeclareMathOperator{\gr}{gr}
\DeclareMathOperator{\cf}{cf}
\DeclareMathOperator{\cb}{cb}
\DeclareMathOperator{\mcb}{mcb}
\DeclareMathOperator{\cdc}{cdc}
\DeclareMathOperator{\cedc}{cedc}
\DeclareMathOperator{\HJL}{HJL}

\DeclareMathOperator{\ord}{ord}

\DeclareMathOperator{\L}{L}

\title[Conductor-discriminant inequality]{Conductor-discriminant
  inequality for tamely ramified cyclic covers II}
  
\author{Connor Stewart}
\address{Graduate Center, City University of New York}
\curraddr{365 5th Avenue, New York, NY 10016, USA}
\email{connor.stewart95@login.cuny.edu}

\subjclass[2010]{Primary: 11G20, 14H25, 14J17; Secondary: 13F30, 14B05}
\keywords{Artin conductor, minimal discriminant, superelliptic curve}
\thanks{The author was supported by NSF CAREER grant DMS-2047638.}

\date{\today}

\begin{document}

\begin{abstract}
Let $K$ be a Henselian discretely valued field with excellent ring of integers $\O_K$ and algebraically closed residue field $k$. Let $X\to\P^1_K$ be a cyclic cover of degree $n$ prime to the characteristic of $k$. In joint work with Obus and Srinivasan (\cite{Part1}), we define an integer called the \textit{conductor-discriminant contribution} $\tx{cdc}(y)$ associated to a multiplicity $2$ point $y$ of the branch divisor of the normalization in $K(X)$ of a regular $\O_K$-model $\Y$ of $\P^1_{K}$; modulo several key results about $\cdc(y)$, we prove a \textit{conductor-discriminant inequality} for $X$, extending previous work of Ogg, Saito, Liu, Srinivasan, and Obus--Srinivasan. In this companion paper, we supply the necessary technical results for $\cdc(y)$. In particular, we show $\cdc(y)$ is non-negative except under highly restrictive conditions on $n$ and the structure of the branch divisor at $y$. Moreover, if $\cdc(y)$ is negative, we show the spectrum of the complete local ring of any point lying over $y$ under the normalization of $\Y$ in $K(X)$ is a rational double point. Along the way, we show the non-negativity of a related quantity, the \textit{conductor exponent-discriminant contribution} $\cedc(y)$, which is used in \cite{Part1} to give a new proof of a result of Kohls.
\end{abstract}

\maketitle

\tableofcontents

\section{Introduction}

Let $K$ be a Henselian discretely valued field with valuation $\nu_K$, excellent ring of integers $\O_K$, and algebraically closed residue field $k$. Let $n\in\N$ prime to $\tx{char}(k)$. Let $X$ be a smooth, geometrically connected, projective curve over $K$ of genus $g\geq1$ admitting a $\Z/n\Z$-cover to $\P^1_K$.

In joint work with Obus and Srinivasan (\cite{Part1}), we define an integer called the \textit{conductor-discriminant contribution} $\tx{cdc}(y)$ associated to a multiplicity $2$ point $y$ of the branch divisor of the normalization in $K(X)$ of a regular $\O_K$-model of $\P^1_{K}$ (Definition \ref{dfn: cdc of multiplicity $2$ point}), and assuming several key results about $\cdc(y)$, we prove a \textit{conductor-discriminant inequality} for $X$ (\eqref{eqn: CD inequality}).

In this companion paper to \cite{Part1}, we supply the necessary technical results for $\cdc(y)$, most importantly Theorem \ref{thm: contractions}. In the course of proving Theorem \ref{thm: contractions}, we obtain a characterization of the points $y$ with $\cdc(y)<0$ (Theorem \ref{thm: cdc}) and show the non-negativity of a related quantity $\cedc(y)$ (Definition \ref{dfn: cdc of multiplicity $2$ point} and Theorem \ref{thm: cedc}), which is used in \cite{Part1} to show a \textit{conductor exponent-discriminant inequality} for $X$ (\eqref{eqn: CED inequality}). 

This completes the arguments of \cite{Part1}, thus extending work of Ogg, Saito, Liu, Srinivasan, and Obus--Srinivasan on the conductor-discriminant inequality and giving a new proof of the conductor exponent-discriminant inequality, which was first established by Kohls (see Subsection \ref{subsection 1.2} for further discussion).

\subsection{Artin Conductor and Minimal Discriminant}\hfill

Fix a prime $l\neq\tx{char}(k)$, and let $\delta(X/K)$ and $\phi(X/K)$ be respectively the Swan conductor and conductor exponent of the $l$-adic representation $\Gal(\c{K}/K)\to\Aut_{\Q_l}(H^1_{\tx{et}}(X_{\c{K}},\Q_l))$, which are independent of the choice $l$. If $\X$ is a regular model of $X$ over $\O_K$, i.e.\;a regular integral proper flat relative curve over $\O_K$ with generic fiber $\X_K=X$, and $\chi$ is the $l$-adic Euler characteristic, the \textit{Artin conductor}\footnote{The definition of $\Art(\X)$ is motivated by the following situation (see \cite[Proposition 1.1]{Bloch}): Let $F$ be a number field, $\omega$ a finite place of $F$, $\tilde{X}$ a curve over $F$, and $\tilde{\X}$ a regular model of $\tilde{X}$ over the ring of integers $\O_F$. Suppose $K$ is the completion of the maximal unramified extension of $F_\omega$, and $X$ and $\X$ are respectively the base changes of $\tilde{X}$ to $K$ and $\tilde{\X}$ to $\O_K$. Then $\Art(\X)$ plays the same role in the functional equation of the $\zeta$-function of $\tilde{\X}$ that $\phi(X/K)$ plays in the functional equation of the $\zeta$-function of $\tilde{X}$.} of $\X$ is the quantity
\[\Art(\X)=\chi(X_{\c{K}})-\chi(\X_k)-\delta(X/K)\]
We denote the Artin conductor of the minimal regular $\X_{\min}$ of $X$ over $\O_K$ by $\Art(X/K)$.

The Artin conductor is a non-positive integer that measures the degeneracy of the model $\X$. In particular, we have $\Art(\X)=0$ if and only if $\X$ is smooth over $\O_K$ or if $g=1$ and the reduced special fiber $\X_k^\tx{red}$ is regular, and by \cite[Proposition 1]{Liu_conductor}, we have
    \begin{equation}\label{eq: Liu equation}
        -\Art(\X)=N-1+\phi(X/K)
    \end{equation}
where $N$ is the number of irreducible components of $\X_k$.

On the other hand, since $X$ admits a $\Z/n\Z$-cover to $\P^1_K$ with $n$ prime to $\tx{char}(k)$, it follows that $X$ is birational to the affine plane curve $\{y^n=f(x)\}\subseteq\A^2_K$ for some polynomial $f\in\O_K[x]$, called an \textit{integral Weierstrass polynomial} for $X$. By a change of coordinates on $\P^1_K$, we can assume $n\mid\deg(f)$, in which case we define
\[\Delta_{f,K}=\nu_K(\disc(\rad(f)))\] where $\rad(f)$ is the product of the distinct irreducible factors of $f$ in $\O_K[x]$, determined up to scaling by $\O_K^\times$, and $\disc$ is the usual polynomial discriminant (cf.\,\cite[Definition 1.1]{Part1}).



The \textit{minimal discriminant of $X$} is then the quantity
\[\Delta_{X/K}=\min_f\Delta_{f,K}\]
where $f$ runs over all integral Weierstrass polynomials for $X$ with $n\mid\deg(f)$. It is a measure of degeneracy of $\X_{\min}$ in the sense that $\Delta_{X/K}=0$ if and only if $\X_{\min}$ is smooth over $\O_K$ (see \cite[Lemma 2.1(ii) and Proposition 2.4]{Part1}).

\subsection{Conductor-Discriminant Inequality}\label{subsection 1.2}\hfill

If $f$ is an integral Weierstrass polynomial for $X$ with $n\mid\deg(f)$, we say $X$ satisfies the \textit{conductor-discriminant inequality for $f$} if 
\[-\Art(X/K)\leq(n-1)\Delta_{f,K}\]

We say $X$ satisfies the \textit{conductor-discriminant inequality} if it satisfies the conductor-discriminant inequality for all integral Weierstrass polynomials $f$ with $n\mid\deg(f)$, or equivalently,
\begin{equation}\label{eqn: CD inequality}
    -\Art(X/K)\leq(n-1)\Delta_{X/K}
\end{equation}

On one hand, the conductor-discriminant inequality is a comparison of measures of degeneracy of the minimal regular model $\X_{\min}$. On the other hand, by \eqref{eq: Liu equation}, it immediately yields upper bounds on the conductor $\phi(X/K)$ and on the number of irreducible components of  $(\X_{\min})_k$, which are of interest in their own right.\footnote{The main application of the former is to produce lists of curves of fixed genus and bounded conductor in databases such as \cite{LMFDB}; for the latter, see, e.g., \cite[Theorem A.5(2)]{mccallum-poonen}.}

In the case $n=2$ and $g=1$, the conductor-discriminant inequality was shown to be an \textit{equality}, now called the \textit{Ogg-Saito formula}, in \cite{Ogg} and \cite{Saito}. The case $n=2$ and $g=2$ was established in \cite{Liu_conductor}, and the full case for $n=2$ was shown in \cite{OSHyper} after partial results in \cite{Padma_rational} and \cite{Padma_tame}. The results of \cite{Ogg}, \cite{Saito}, and \cite{Liu_conductor} hold, moreover, without the assumption $\tx{char}(k)\nmid n$.

If $f$ is an integral Weierstrass polynomial for $X$ with $n\mid\deg(f)$, there is also the weaker \textit{conductor exponent-discriminant inequality for $f$}, namely
\[\phi(X/K)\leq (n-1)\Delta_{f,K}\]
which, by \eqref{eq: Liu equation}, holds whenever $X$ satisfies the conductor-discriminant inequality for $f$. The conductor exponent-discriminant inequality was shown to hold for all $f$ and all $n$ prime to $\tx{char}(k)$ by Kohls in his Ph.D.\;thesis, who in fact proved a stronger upper bound for $\phi(X/K)$ when $f$ is divisible in $\O_K[x]$ by a uniformizer of $\nu_K$ (see \cite[Theorem 5.1.11]{Kohls}).

We say $X$ satisfies the \textit{conductor exponent-discriminant inequality} if it satisfies the conductor exponent-discriminant inequality for all integral Weierstrass polynomials $f$ with $n\mid\deg(f)$, or equivalently,
\begin{equation}\label{eqn: CED inequality}
\phi(X/K)\leq(n-1)\Delta_{X/K}
\end{equation}

In \cite{Part1}, making use of Theorems \ref{thm: cedc} and \ref{thm: contractions}, Propositions \ref{prop: local resolution is finite} and \ref{prop: stronger positivity for type II}, and Lemma \ref{lemma: multiplicity in div_0(f)} of this paper, we establish the conductor-discriminant inequality, and give a new proof of the conductor exponent-discriminant inequality, for all $n$ prime to $\tx{char}(k)$.

\subsection{Conductor-Discriminant Contribution of Multiplicity $2$ Points}\hfill

Let $f$ be an integral Weierstrass polynomial for $X$ with $n\mid\deg(f)$. Let $\Y$ be a regular model of $\P^1_K$, let $\X\to\Y$ be the normalization of $\Y$ in $K(X)$, let $D$ be the branch divisor of $\X\to\Y$, and let $y$ be a multiplicity $2$ point of $D$. In Section \ref{section 2}, we define a modification of $\Y_m\to\Y$, called the \textit{local $n$-resolution of $D$ at $y$}, such that $\Y_m\to\Y$ is an isomorphism over $\Y\setminus\{y\}$, and if $D_m$ is the branch divisor of the normalization $\X_m\to\Y_m$ of $\Y_m$ in $K(X)$, then $D_m$ has simple normal crossings at all points of the preimage of $y$ on $\Y_m$.

In terms of this modification, we associate integers $\cdc(y)$ and $\cedc(y)$ to $y$, called respectively the \textit{conductor-discriminant contribution} and \textit{conductor exponent-discriminant contribution} of $y$, which are connected to the conductor-discriminant and conductor exponent-discriminant inequalities for $f$ by the following result of \cite{Part1}:

\begin{theorem}\label{thm: reduction to cdc} $($\cite[Proposition 10.9]{Part1}$)$
    There exists a regular model $\Y$ of $\P^1_K$ with the following property:
    \begin{itemize}
        \item Let $\Y_0=\Y$, let $\X_0\to\Y_0$ be the normalization of $\Y_0$ in $K(X)$, let $D_0$ be the branch divisor of $\X_0\to\Y_0$, and let $\{y_1,...,y_w\}$ be the set of multiplicity $2$ points of $D_0$.
        \item For $1\leq i\leq w$, let $\Y_i\to\Y_{i-1}$ be the local $n$-resolution of the branch divisor $D_{i-1}$ of the normalization $\X_{i-1}\to\Y_{i-1}$ of $\Y_{i-1}$ in $K(X)$ at the unique point of $D_{i-1}$ lying over $y_i$.
        \item Let $\X_w'\to\X_w$ be the minimal resolution of $\X_w$.
    \end{itemize}
    Then there is a subset $A\subseteq\{y_1,...,y_w\}$ such that
    \[(n-1)\Delta_{f,K}+\Art(\X_w')\geq\sum_{y\in A}\cdc(y)\]
    and
    \[(n-1)\Delta_{f,K}-\phi(X/K)\geq\sum_{y\in A}\cedc(y)\]  
\end{theorem}  

Keep the notation of Theorem \ref{thm: reduction to cdc}. Since $\Art(\X_w')\leq\Art(\X_{\min})$ by \eqref{eq: Liu equation}, it follows that Theorem \ref{thm: reduction to cdc} would imply the conductor-discriminant inequality for $f$ if we could show $\cdc(y)\geq0$ for all $y\in A$. Similarly,  Theorem \ref{thm: reduction to cdc} would imply the conductor exponent-discriminant inequality for $f$ if we could show $\cedc(y)\geq0$ for all $y\in A$.

The statement for $\cedc(y)$ holds by Theorem \ref{thm: cedc} of this paper. The situation for $\cdc(y)$ is more complicated: $\cdc(y)$ can be negative in general; however, the condition $\cdc(y)<0$ is so restrictive that it implies that the spectrum of the complete local ring of any point of the preimage of $y$ under the normalization of $\Y$ in $K(X)$ is a rational double point admitting a $D_4,E_6,$ or $E_8$ resolution (see Theorem \ref{thm: cdc}). This leads to Theorem \ref{thm: contractions}, which combined with Theorem \ref{thm: reduction to cdc} and \eqref{eq: Liu equation}, is enough to get the conductor-discriminant inequality for $f$ (see \cite[Proposition 10.19 and Corollary 10.20]{Part1}).

\subsection{Main Results}\hfill

The main results of this paper are the following three theorems, the first two of which are used in \cite{Part1} to deduce conductor exponent-discriminant and conductor-discriminant inequalities from Theorem \ref{thm: reduction to cdc}. We note that the proof of Theorem \ref{thm: cedc} does not make use of Theorems \ref{thm: contractions} and \ref{thm: cdc}, and the proof of the conductor exponent-discriminant inequality obtained by combining Theorems \ref{thm: reduction to cdc} and \ref{thm: cedc} is much simpler than that obtained by combining the conductor-discriminant inequality and \eqref{eq: Liu equation}.

\begin{notation}\label{notation: intro notation}
    Let $\Y$ be a regular model of $\P^1_K$, let $\X\to\Y$ be the normalization of $\Y$ in $K(X)$, let $D$ be the branch divisor of $\X\to\Y$, and let $y$ be a multiplicity $2$ point of $D$. Let $x\in\X$ lie over $y$.
\end{notation}

\begin{theorem}\label{thm: cedc}
    Keep Notation \ref{notation: intro notation}. Then $\cedc(y)\geq0$.
\end{theorem}
\begin{proof}
    This follows from Corollaries \ref{cor: cedc positivity for type II} and \ref{cor: cedc positivity for type I}.
\end{proof}

\begin{theorem}\label{thm: contractions}
    Keep Notation \ref{notation: intro notation}, and suppose $\cdc(y)<0$. Let $\Y_m\to\Y$ be the local $n$-resolution of $D$ at $y$. Let $\X\to\Y$ and $\X_m\to\Y_m$ be the normalizations of $\Y$ and $\Y_m$ in $K(X)$, and let $\X'\to\X$ and $\X'_m\to\X_m$ be the minimal resolutions of $\X$ and $\X_m$.
    Then the canonical morphism $\X_m'\to\X'$ contracts at least $-\cdc(y)$ irreducible components of the special fiber $(\X_m')_s$ that are contained in the preimage of $y$ on $\X_m'$.
\end{theorem}
\begin{proof}
    This is Proposition \ref{prop: contractions}.
\end{proof}

Finally, in the course of proving Theorem \ref{thm: contractions}, we obtain the following explicit description of the case $\cdc(y)<0$:

\begin{theorem}\label{thm: cdc}
    Keep Notation \ref{notation: intro notation}. Then $\cdc(y)<0$ if and only if one of the following holds:
    \begin{enumerate}[\upshape (i)]
        \item The point $y$ is contained in a unique irreducible component $\Gamma$ of $D$; the ramification index $e$ of $\Gamma$ with respect to the cover $\X\to\Y$ is $3,4,$ or $5$; and the strict transform of $\Gamma$ under the blow-up of $\Y$ at $y$ is regular.
        \item The point $y$ is contained in a unique irreducible component $\Gamma$ of $D$; the ramification index of $\Gamma$ with respect to $\X\to\Y$ is $3$; the strict transform $\Gamma_1$ of $\Gamma$ under the blow-up $\Y_1\to\Y$ of $\Y$ at $y$ is singular, but the strict transform of $\Gamma_1$ under the blow-up $\Y_2\to\Y_1$ of $\Y_1$ at the singular locus of $\Gamma_1$ is regular; and $n=3$.
        \item The point $y$ is contained two irreducible components of $D$, which meet with intersection multiplicity $2$ at $y$; both components have ramification index $3$ with respect to $\X\to\Y$; the exceptional divisor of the blow-up $\Y_1\to\Y$ of $\Y$ at $y$ has ramification index $3$ with respect to the normalization $\X_1\to\Y_1$ of $\Y_1$ in $K(X)$; and $n=3$.
    \end{enumerate}
    Moreover, in each of these cases, the scheme $\Spec\widehat{\O}_{\X,x}$ is a rational double point, admitting a $D_4$ minimal resolution if \textup{(i)} holds and $e=3$; an $E_6$ minimal resolution if \textup{(iii)} holds or \textup{(i)} holds and $e=4$; or an $E_8$ minimal resolution if \textup{(ii)} holds or \textup{(i)} holds and $e=5$.
\end{theorem}
\begin{proof}
    This follows from Propositions \ref{prop: full positivity for type II}, \ref{prop: full positivity for type I}, and \ref{prop: rational double points}.
\end{proof}
    
\subsection{Outline of Paper}\hfill

In Section \ref{section 2}, we first recall several standard results on models of curves over $\O_K$ that are used throughout the paper. Then for a multiplicity $2$ point $y$ of the branch divisor $D$ of the normalization $\X\to\Y$ of a regular model $\Y$ of $\P^1_K$ in $K(X)$, we define the local $n$-resolution of $D$ at $y$, the conductor-discriminant contribution $\cdc(y)$, and the conductor exponent-discriminant contribution $\cedc(y)$ (Definitions \ref{dfn: cdc of sequence}, \ref{dfn: resolution of local branch divisor}, and \ref{dfn: cdc of multiplicity $2$ point}), following \cite{Part1}. 

In Section \ref{section 3}, we give an explicit description of the local $n$-resolution of $D$ at $y$ (Propositions \ref{prop: description of Y_j}, \ref{prop: full local resolution for type II}, \ref{prop: full local resolution for type I}, and \ref{prop: full local resolution for type I continued}), including all data necessary for computing $\cdc(y)$ and $\cedc(y)$.

In Section \ref{section 4}, we use the results of Section \ref{section 3} to compute various quantities appearing in the definitions of $\cdc(y)$ and $\cedc(y)$. We then show $\cedc(y)\geq0$ always (Corollaries \ref{cor: cedc positivity for type II} and \ref{cor: cedc positivity for type I}) and $\cdc(y)\geq0$ except under highly restrictive conditions on $n$ and the structure of $D$ at $y$ (Propositions \ref{prop: full positivity for type II} and \ref{prop: full positivity for type I}). These results yield Theorems \ref{thm: cedc} and part of Theorem \ref{thm: cdc}.

In Section \ref{section 5}, we apply results of \cite[Section 24]{Lip69} to study the cases where $\cdc(y)<0$. We show that if $\cdc(y)<0$, then the spectrum of the complete local ring of any point of the preimage of $y$ on $\X$ is a rational double point admitting a $D_4,E_6,$ or $E_8$ minimal resolution (Proposition \ref{prop: rational double points}), completing Theorem \ref{thm: cdc}.

Finally, in Section \ref{section 6}, we use the results of Section \ref{section 5} to deduce Theorem \ref{thm: contractions}.

The reader who is interested only in the conductor exponent-discriminant inequality can stop reading after Corollary \ref{cor: cedc positivity for type I}, skipping Proposition \ref{prop: full positivity for type II}.

\addtocontents{toc}{\protect\setcounter{tocdepth}{0}}

\section*{Notation and Conventions}
Throughout, we let $K$ denote a Henselian discretely valued field with valuation $\nu_K$, excellent ring of integers $\O_K$, and algebraically closed residue field $k$. 
We fix $n\in\Z$ and $f\in\O_K[x]$ satisfying $n\geq 2$, $\gcd(n,\tx{char}(k))=1$, and $n\mid \deg(f)$, and we let $X$ be the smooth projective curve over $K$ with affine equation $y^n=f(x)$.

If $Y$ is any curve over $K$, then by a \textit{model} (resp.\:\textit{regular model}) of $Y$ over $\O_K$, we mean a normal (resp.\:regular) integral proper flat relative curve $\Y$ over $\O_K$ with $\Y_K\cong\Y$. By abuse of notation, we will sometimes use the letter $k$ as an index in a sequence. To avoid ambiguity, in Sections \ref{section 2} through \ref{section 6}, if $\Y$ is a model of a curve over $K$, the special fiber $\Y\times_{\Spec\O_K}\Spec k$ is always denoted by $\Y_s$ and never by $\Y_k$.

Suppose $\Y$ is a regular model of $\P^1_{K}$, and let $\X\to\Y$ be the normalization of $\Y$ in $K(X)$. If $\Gamma$ is an irreducible divisor on $\Y$, the \textit{ramification index} of $\Gamma$ will always mean the ramification index with respect to the map $\X\to\Y$; likewise, the \textit{branch divisor} on $\Y$ will always mean the reduced branch divisor of $\X\to\Y$. The \textit{blow-up of $\Y$ at a closed point} $y\in\Y$ will always mean the blow-up at the reduced closed subscheme supported at $\{y\}$.

Suppose now $E$ is a divisor on $\Y$. We say $E$ has \textit{simple normal crossings} if the irreducible components of $E$ are regular, at most two irreducible components intersect at any point, and all such intersections are transverse. If $E$ is effective and reduced, we let $\widetilde{E}$ denote the normalization of $E$, where $E$ is considered as a reduced scheme. If $E=\sum a_i\Gamma_i$ where the $\Gamma_i$ are irreducible divisors, we let $(E\tx{ mod }n)^\tx{red}$ denote the divisor $\sum b_i\Gamma_i$ on $\Y$, where $b_i=0$ if $n\mid a_i$ and $b_i=1$ otherwise. If $F$ is another divisor on $\Y$ and $y\in\Y$ is a closed point, we let $i_y(E,F)$ denote the local intersection number of $E$ and $F$ at $y$. If $F$ is a divisor on $\Y$ relatively prime to $E$, we let $(E,F)$ denote the global intersection number of $E$ and $F$, that is, the sum $\sum_y i_y(E,F)$ where $y$ runs through the finite set $\tx{Supp}(E)\cap\tx{Supp}(F)$.

Finally, if $R$ is a local ring, we let $\widehat{R}$ denote the completion of $R$ at its maximal ideal.


\section*{Acknowledgments}
The contents of this paper are to appear in my Ph.D.\:thesis. I am tremendously grateful to Andrew Obus, my Ph.D. advisor and co-author of \cite{Part1}, for much help, patience, and encouragement throughout the writing of this paper. I am also very grateful to my co-author Padmavathi Srinivasan of \cite{Part1} for allowing me to collaborate on the conductor-discriminant inequality and for many helpful conversations.

\section*{AI Statement}
This work was conducted without the use of AI. In addition, the main results and their proofs were known to the author by January 2025 and discussed in public talks in 2025 and early 2026, prior to the recent surge of AI-assisted proofs.

\addtocontents{toc}{\protect\setcounter{tocdepth}{1}}

\section{Conductor-Discriminant Contribution}\label{section 2}
In this section, we define integers called the \textit{conductor-discriminant} and \textit{conductor exponent-discriminant contributions} $\cdc(y)$ and $\cedc(y)$ attached to a multiplicity $2$ point $y$ of the branch divisor of the normalization $\X\to\Y$ of a regular model $\Y$ of $\P^1_K$ in $K(X)$. The main focus of later sections will be to characterize the points $y$ for which $\cdc(y)<0$ or $\cedc(y)<0$.

\subsection{Preliminaries on Curves over $\O_K$}\hfill

We start by collecting several standard results about curves over $\O_K$ that will be used throughout the paper, most commonly in the setting of a regular model of $\P^1_K$.

\begin{lemma}\label{lemma: unique intersection points new}
    Let $\Y$ be a normal integral proper flat relative curve over $\O_K$, and let $\Gamma$ be an irreducible horizontal divisor on $\Y$, considered as a reduced scheme. Then $\Gamma$ is the spectrum of an analytically irreducible local domain. In particular, $\Gamma$ contains a unique closed point.
\end{lemma}
\begin{proof}
    By \cite[Proposition 8.3.4]{Liu_book}, the divisor $\Gamma$ is finite over $\O_K$, hence affine. Since $\Gamma$ is integral, it follows from Statement 1 of \cite[Tag 04GH]{stacks-project} that $A\colonequals\O_{\Gamma}(\Gamma)$ is a Henselian local domain finite over $\O_K$. Since $\O_K$ is excellent, so is the ring $A$. Then the completion of $A$ is reduced by Statement 3 of \cite[Tag 07NZ (3)]{stacks-project} and has a unique minimal prime by \cite[Tag 0C2E]{stacks-project}, hence $A$ is analytically irreducible.
\end{proof}

\begin{lemma}\label{lemma: intersection number decreases with blow-up}
    Let $\Y$ be a regular integral proper flat relative curve over $\O_K$, and let $y\in\Y$ be a closed point. Let $\Gamma_1,\Gamma_2$ be distinct irreducible divisors on $\Y$, and for $i\in\{1,2\}$, let $\mu_i=\mult_y(\Gamma_i)$. Let $\pi:\Y'\to\Y$ be the blow-up of $\Y$ at $y$, let $E$ be the exceptional divisor, and let $\Gamma'_1,\Gamma'_2$ be the strict transforms of $\Gamma_1,\Gamma_2$. Then
    \begin{enumerate}[\upshape (i)]
        \item $(\Gamma'_1,E)=\mu_1$
        \item $(\Gamma'_1,\Gamma'_2)=(\Gamma_1,\Gamma_2)-\mu_1\mu_2$
        \item If $\Gamma_1$ is horizontal, then the unique closed point $y'$ of $\Gamma_1'$ satisfies $\mult_{y'}(\Gamma_1')\leq \mu_1$.
    \end{enumerate}
\end{lemma}
\begin{proof}
    For (i), by Propositions 9.2.5, 9.2.23, and 9.2.12(a) of \cite{Liu_book} for the three equalities respectively, we have
    \[(\Gamma_1',E)-\mu_1=(\Gamma_1'+\mu_1E,E)=(\pi^*\Gamma_1,E)=0\]
    
    For (ii), if $C$ and $D$ are relatively prime reduced effective divisors on $\Y$, and if $C'$ is the strict transform of $C$ under the blow-up of $\Y$ at a closed point of multiplicity $\mu$ in $C$, then by \cite[Lemma 3.1]{OSHyper}, we have
    \[(C,D)=\l{\widetilde{C+D}}{C+D}-\l{\widetilde{C}}{C}-\l{\widetilde{D}}{D}\]
    and by \cite[Lemma 3.3]{OSHyper}, we have
    \[\l{\widetilde{C}}{C}-\l{\widetilde{C'}}{C'}=\l{C'}{C}=\frac{\mu(\mu-1)}{2}\]
    
    Applying these results to $\Gamma_1$ and $\Gamma_2$ and to $\Gamma_1'$ and $\Gamma_2'$, and noting that the multiplicity of $y$ in the divisor $\Gamma_1+\Gamma_2$ is $\mu_1+\mu_2$, we obtain
    \begin{align*}
       (\Gamma_1,\Gamma_2)-(\Gamma'_1,\Gamma'_2)&=\frac{(\mu_1+\mu_2)(\mu_1+\mu_2-1)}{2}-\frac{\mu_1(\mu_1-1)}{2}-\frac{\mu_2(\mu_2-1)}{2}=\mu_1\mu_2
    \end{align*}

    Finally, for (iii), suppose $\Gamma_1$ is horizontal, so that $\Gamma_1'$ is horizontal as well. Then $\Gamma_1'$ contains a unique closed point $y'$ by Lemma \ref{lemma: unique intersection points new}, and since $(\Gamma_1',E)=\mu_1\geq 1$, we see $y'\in\tx{Supp}(E)$. Since $E\cong\P^1_k$ by \cite[Proposition 9.2.5]{Liu_book}, we have $\mult_{y'}(E)=1$. If $\Y''\to\Y'$ is the blow-up of $\Y'$ at $y'$, and if $\Gamma_1''$ and $E'$ are strict transforms of $\Gamma_1'$ and $E$ under $\Y''\to\Y'$ respectively, then by (i) for the first equality and (ii) applied to $\Gamma_1'$ and $E$ for the second equality, we have
    \[\mu_1-\mult_{y'}(\Gamma_1')=(\Gamma_1',E)-\mult_{y'}(\Gamma_1')=(\Gamma_1'',E')\geq0\qedhere\]
\end{proof}

\begin{lemma}\label{lemma: unique intersection points}
    Let $\Y$ be a regular model of $\P^1_K$.
    \begin{enumerate}[\upshape (i)]
        \item Each irreducible component of the special fiber $\Y_s$ is isomorphic to $\P^1_k$.
        \item The special fiber $\Y_s$ has simple normal crossings and its dual graph is a tree.
        \item If $\Gamma$ and $E$ are relatively prime divisors on $\Y$ with $\Gamma$ irreducible and $E$ connected, then $\Gamma$ and $E$ intersect in at most one closed point.
    \end{enumerate}
\end{lemma}
\begin{proof}
    Statement (i) follows from \cite[Lemma 7.1]{ObusWewers}, statement (ii) is \cite[Lemma 7.7]{Part 1}, and statement (iii) follows from (ii) and Lemma \ref{lemma: unique intersection points new}.
\end{proof}

\subsection{Conductor-Discriminant Contribution of a Sequence of Blow-ups}

\begin{defn}\label{dfn: Hirzebruch-Jung length of a pair}\cite[Definition 7.2]{Part1}
    Let $1\leq \nu\in\Z$ and let $\rho\in(\Z/\nu\Z)^\times$. Let $r_0=\nu,$ let $r_1$ be the least residue of $\rho$ mod $\nu$ (i.e.\,the unique representative of $\rho$ mod $\nu$ in $\{0,...,\nu-1\}\subseteq\Z$), and while $r_i\neq 0$, let $r_{i+1}$ be the least residue of $-r_{i-1}$ mod $r_i$.
    The largest index $i$ for which $r_i$ is defined and non-zero is called the \textit{Hirzebruch-Jung length of the pair $(\nu,\rho)$} and is denoted $\L(\nu,\rho)$. If $\rho$ is represented by an integer $r\in\Z$, we also write $\L(\nu,r)$ for $\L(\nu,\rho)$.
\end{defn}

\begin{remark}\label{rmk: L bound}
    Keep the notation of Definition \ref{dfn: Hirzebruch-Jung length of a pair}. Since the $r_i$ form a decreasing sequence, we see $\L(\nu,\rho)\leq r_1\leq \nu-1$.
\end{remark}

Now let $\Y$ be a regular model of $\P^1_K$, let $\X\to\Y$ be the normalization of $\Y$ in  $K(X)$, and let $D=(\tx{div}_0(f)\tx{ mod }n)^\tx{red}$ be the branch divisor of $\X\to\Y$. Recall that a \textit{simple normal crossings point} of a divisor $E$ on $\Y$ is a point $y\in\tx{Supp}(E)$ that is regular in $E$ or is contained in exactly two irreducible components of $E$, which meet transversely at $y$.
    
\begin{defn}\label{dfn: Hirzebruch-Jung length of a point}   
    Let $y$ be a multiplicity $2$ simple normal crossings point of $D$. Let $\nu,\nu'$ denote the orders of $f$ along the two irreducible components of $D$ containing $y$, and let $e=\frac{n}{\gcd(n,\nu)}$ and $e'=\frac{n}{\gcd(n,\nu')}$ be the ramification indices of these components.
    
    The unordered pairs $(\nu,\nu')$ and $(e,e')$ are called respectively the \textit{order} and \textit{ramification pairs} of $y$. If $\omega,\omega'\in\Z$ satisfy $\omega\equiv \nu$ and $\omega'\equiv\nu'$ mod $n$, then abusing notation, we also say $y$ has order pair $(\omega,\omega')$ (see Remark \ref{rmk: Hirzebruch-Jung length}). In this case, we have $e=\frac{n}{\gcd(n,\omega)}$ and $e'=\frac{n}{\gcd(n,\omega')}$.
\end{defn}

\begin{defn}\cite[Definition 7.9]{Part1} 
    Keep the notation of Definition \ref{dfn: Hirzebruch-Jung length of a point}. Since $\gcd{(e,e')}$ divides $n$, the expression $-(\frac{\nu}{\gcd(\nu,n)})^{-1}\frac{\nu'}{\gcd(\nu',n)}$ is a well-defined element of $(\Z/\gcd(e,e')\Z)^\times$. The non-negative integer
    \[\HJL(y)=\L(\gcd(e,e'),-(\frac{\nu}{\gcd(\nu,n)})^{-1}\frac{\nu'}{\gcd(\nu',n)})\] where $\L$ is as in Definition \ref{dfn: Hirzebruch-Jung length of a pair}, is called the \textit{Hirzebruch-Jung length of the point $y$}.
\end{defn}

\begin{remark}\label{rmk: Hirzebruch-Jung length}
    Keep the notation of Definition \ref{dfn: Hirzebruch-Jung length of a point}. The formula defining $\HJL(y)$ depends only on the classes of $\nu$ and $\nu'$ mod $n$. Furthermore, by \cite[Remark 7.10]{Part 1}, this formula counts the number of irreducible components in the exceptional divisor of the minimal resolution of $\X$ in $x$ for any point $x\in\X$ contained in the preimage of $y$ under $\X\to\Y$. In particular, the definition of $\HJL(y)$ is independent of the choice of labeling of $\nu$ and $\nu'$.
\end{remark}

\begin{notation}\label{notation: cdc set-up}
    Let $F$ be an effective divisor on $\Y$,
    let $\Y_m\to\Y_{m-1}\to...\to\Y_0\equalscolon \Y$ be a finite sequence of blow-ups at reduced closed points, and let $\X_m\to\Y_m$ be the normalization of $\Y_m$ in $K(X)$.
    \begin{itemize}
        \item For $0\leq k\leq m$, let $F_k^*$ be the total transform of $F$ under the map $\Y_k\to\Y_0$, and let $D_k=(F_k^*\tx{ mod }n)^\tx{red}$.
        \item For $1\leq k\leq m$, let $y_k\in\Y_{k-1}$ be the center of the blow-up $\Y_k\to\Y_{k-1}$, and let $\nu_k=\mult_{y_k}(F^*_{k-1})$ and $\mu_k=\mult_{y_k}(D_{k-1})$. 
        \item Let $z_1,...,z_l$ be the multiplicity $2$ simple normal crossings points of the branch divisor of $\X_m\to\Y_m$ that are contained in $D_m$, and for $1\leq i\leq l$, let $(e_i,e_i')$ be the ramification pair of $z_i$.
    \end{itemize} 
\end{notation}

\begin{remark}\label{rmk: mult 2 pt discrepancy}
    Keep Notation \ref{notation: cdc set-up}. If $F=\divi_0(f)$, so that $D_m$ is equal to the branch divisor of $\X_m\to\Y_m$, or if $F=F_y$ in the notation of Subsection \ref{subsection 2.3}, so that Lemma \ref{lemma: multiplicity in div_0(f)}(ii) applies, the points $z_1,...,z_l$ are simply the multiplicity $2$ simple normal crossings points of $D_m$. These are the only cases of interest to us.
\end{remark}

\begin{defn}\label{dfn: cdc of sequence}
    \cite[Definition 9.4]{Part1}
    Keep Notation \ref{notation: cdc set-up}.\,The \textit{conductor-discriminant contribution} ($\cdc$) of the sequence of blow-ups $\Y_m\to...\to\Y_0$ and divisor $F$ is the quantity
    \begin{align*}
        \cdc(\Y_m\to...\to\Y_0,F)=&\sum_{k\geq 1:n\nmid\nu_k}(n-2\gcd(\nu_k,n)+(n-1)\mu_k(\mu_k-3))\\
        +&\sum_{k\geq 1:n\mid\nu_k}(-n+(n-1)\mu_k(\mu_k-1))\\+&\sum_{i=1}^l((\frac{n}{e_i}-1)+(\frac{n}{e_i'}-1))+\sum_{i=1}^l(n-\gcd(\frac{n}{e_i},\frac{n}{e_i'})(\HJL(z_i)+1))
    \end{align*}
    
    Similarly, the \textit{conductor exponent-discriminant contribution} ($\cedc$) of $\Y_m\to...\to\Y_0$ and $F$ is the quantity
    \begin{align*}
        \tx{cedc}(\Y_m\to...\to\Y_0,F)=&\sum_{k\geq1:n\nmid\nu_k}(1+n-2\gcd(\nu_k,n)+(n-1)\mu_k(\mu_k-3))\\
        &+\sum_{k\geq 1:n\mid\nu_k}(1-n+(n-1)\mu_k(\mu_k-1))
        \\&+\sum_{i=1}^l((\frac{n}{e_i}-1)+(\frac{n}{e_i'}-1))+\sum_{i=1}^l(n-\gcd(\frac{n}{e_i},\frac{n}{e_i'}))
    \end{align*}

    We assign names to several of the sums appearing in the definitions of $\cdc$ and $\cedc$. The \textit{blow-up factor} ($\blf$), \textit{crossing factor} ($\cf$), \textit{crossing bonus} ($\cb$), \textit{modified blow-up factor} ($\mbf$), and \textit{modified crossing bonus} ($\mcb$) of $\Y_m\to\Y_{m-1}\to...\to\Y_0$ and $F$ are defined by
    \begin{align*}
        \blf(\Y_m\to...\to\Y_0,F)&=\sum_{k\geq 1:n\nmid\nu_k}(n-2\gcd(\nu_k,n)+(n-1)\mu_k(\mu_k-3))\\
        &+\sum_{k\geq 1:n\mid\nu_k}(-n+(n-1)\mu_k(\mu_k-1))\\
        \cf(\Y_m\to...\to\Y_0,F)&=\sum_{i=1}^l((\frac{n}{e_i}-1)+(\frac{n}{e_i'}-1))\\
        \cb(\Y_m\to...\to\Y_0,F)&=\sum_{i=1}^l(n-\gcd(\frac{n}{e_i},\frac{n}{e_i'})(\HJL(z_i)+1))\\
        \mbf(\Y_m\to...\to\Y_0,F)&=\sum_{k\geq 1:n\nmid\nu_k}(1+n-2\gcd(\nu_k,n)+(n-1)\mu_k(\mu_k-3))\\
        &+\sum_{k\geq 1:n\mid\nu_k}(1-n+(n-1)\mu_k(\mu_k-1))\\
        \mcb(\Y_m\to...\to\Y_0,F)&=\sum_{i=1}^l(n-\gcd(\frac{n}{e_i},\frac{n}{e_i'}))
    \end{align*}  

    Finally, if $1\leq i\leq l$, the \textit{contributions of $z_i$ to the crossing bonus} and \textit{to the modified crossing bonus} are the quantities $n-\gcd(\frac{n}{e_i},\frac{n}{e_i'})(\HJL(z_i)+1)$ and $n-\gcd(\frac{n}{e_i},\frac{n}{e_i'})$ respectively. 
\end{defn}

\begin{remark}\label{rmk: discussion of bf, cf, etc.}
    The crossing factor and modified crossing bonus are always non-negative. By Remark \ref{rmk: L bound}, we have $\HJL(z_i)\leq\gcd(e_i,e_i')-1$ for all $1\leq i\leq l$, so that the crossing bonus is non-negative as well. Additionally, we have the equality
    \[\mbf(\Y_m\to...\to\Y_0,F)=m+\blf(\Y_m\to...\to\Y_0,F)\]
\end{remark}

\begin{lemma}\label{lemma: quick crossing bonus}
    Keep Notation \ref{notation: cdc set-up}, and let $i\in\{1,...,l\}$. If $e_i\neq e_i'$, then \[n-\gcd(\frac{n}{e_i},\frac{n}{e_i'})(\HJL(z_i)+1)\geq \frac{n}{2}\]
\end{lemma}
\begin{proof}
    Since $e_i\neq e_i'$, we have $\lcm(e_i,e_i')\geq 2\gcd(e_i,e_i')$. Since $\HJL(z)\leq\gcd(e_i,e_i')-1$, it follows that
    \begin{align*}
        n-\gcd(\frac{n}{e_i},\frac{n}{e_i'})(\HJL(z_i)+1)&\geq n-\gcd(\frac{n}{e_i},\frac{n}{e_i'})\gcd(e_i,e_i')=n-n\cdot \frac{\gcd(e_i,e_i')}{\lcm(e_i,e_i')}\geq \frac{n}{2}\qedhere
    \end{align*}
\end{proof}

\subsection{Conductor-Discriminant Contribution of Multiplicity $2$ Points of the Branch Divisor}\label{subsection 2.3}\hfill

For the rest of this section, let $y$ be a multiplicity $2$ point of $D$, and let $F_y$ be the sum of the irreducible components of $D$ that contain $y$, with multiplicities inherited from $\tx{div}_0(f)$.
\begin{defn}\label{dfn: resolution of local branch divisor}
    The \textit{local $n$-resolution of $D$ at $y$} is the sequence of blow-ups $...\to\Y_k\to...\to\Y_0\equalscolon \Y$ centered at reduced closed points $y_{k+1}\in\Y_k$ defined inductively as follows:
    \begin{itemize}
        \item Set $F_0=F_0^*=F_y$ and $D_0=(F_y\tx{ mod }n)^\tx{red}$.
        \item If the divisor $F_k$ on $\Y_k$ contains multiple points lying over $y$ via the map $\Y_k\to\Y_0$, or if $F_k$ contains a unique point lying over $y$ and $D_k$ has simple normal crossings at this point, the local $n$-resolution of $D$ at $y$ terminates at the stage $\Y_k\to...\to\Y_0$.
        \item Otherwise, let $y_{k+1}\in\Y_k$ be the unique point of $F_k$ lying over $y$, let $\Y_{k+1}\to\Y_k$ be the blow-up of $\Y_k$ at $y_{k+1}$, let $F_{k+1}$ and $F_{k+1}^*$ respectively denote the strict and total transforms of $F_y$ under $\Y_{k+1}\to\Y_0$, so that in particular the divisor $F_{k+1}$ contains at least one point lying over $y$ via the map $\Y_{k+1}\to\Y_0$, and let $D_{k+1}=(F_{k+1}^*\tx{ mod }n)^\tx{red}$.
\end{itemize}
\end{defn}

We state the following proposition now, although its proof relies on results of Section \ref{section 3}.

\begin{prop}\label{prop: local resolution is finite}
Keep the notation of Definition \ref{dfn: resolution of local branch divisor}. The local $n$-resolution of $D$ at $y$ terminates at a finite stage $\Y_m\to...\to\Y_0$, and $D_m$ has simple normal crossings.
\end{prop}
\begin{proof}
    This follows from Propositions \ref{prop: local resolution desingularizes components}(iii), \ref{prop: full local resolution for type II}(v), \ref{prop: full local resolution for type I}(vi), and \ref{prop: full local resolution for type I continued}(v).
 \end{proof}

\begin{defn}\label{dfn: cdc of multiplicity $2$ point}
    Keep the notation of Definition \ref{dfn: resolution of local branch divisor}, and let $m$ be as in Proposition \ref{prop: local resolution is finite}. The \textit{blow-up factor} $\blf(y)$, \textit{modified blow-up factor} $\mbf(y)$, \textit{crossing factor} $\cf(y)$, \textit{crossing bonus} $\cb(y)$, \textit{modified crossing bonus} $\mcb(y)$, 
    \textit{conductor-discriminant contribution} $\cdc(y)$, and
    \textit{conductor exponent-discriminant contribution} $\cedc(y)$ of the multiplicity $2$ point $y$ of $D$ are defined to be those associated as in Definition \ref{dfn: cdc of sequence} to the sequence of blow-ups $\Y_m\to...\to\Y_0$ and divisor $F_y$.
\end{defn}

The next lemma relates the quantities $\cdc(y)$ and $\cdc(\Y_m\to...\to\Y_0,\divi_0(f))$: in short, they are the same except that the crossing factor and crossing bonus in $\cdc(\Y_m\to...\to\Y_0,\divi_0(f))$ may have additional contributions coming from points lying outside of the preimage of $y$ on $\Y_m$.

\begin{lemma}\label{lemma: multiplicity in div_0(f)}
    Suppose the local $n$-resolution of $D$ at $y$ begins with the sequence of blow-ups $\Y_k\to...\to\Y_0\equalscolon\Y$.
    \begin{enumerate}[\upshape (i)]
        \item The divisor $D_k$ is the sum of the irreducible components of $(\divi_0(f)\textup{ mod }n)^{\red}$ that intersect the preimage of $y$ on $\Y_k$.
        \item The multiplicity $2$ simple normal crossings points of $D_k$ are exactly those of $(\divi_0(f)$ $\textup{mod }n)^{\red}$ that are contained in the preimage of $y$ on $\Y_k$.
        \item If $y'\in\Y_k$ is a closed point lying over $y$ via $\Y_k\to\Y_0$, then $\mult_{y'}(\divi_0(f))\equiv \mult_{y'}(F_k^*)$ \textup{mod} $n$ and $\mult_{y'}((\divi_0(f)\textup{ mod }n)^{\red})=\mult_{y'}(D_k)$.
    \end{enumerate}
\end{lemma}
\begin{proof}
    The following two facts are immediate from the definitions of $F_y$ and the local $n$-resolution of $D$ at $y$, and together they imply (i) and (iii): 1) there are divisors $G,H$ on $\Y$, relatively prime to $F_y$, such that $y\not\in\tx{Supp}(G)$ and $F_y+G+nH=\divi_0(f)$ on $\Y$; and 2) the center of each blow-up in the sequence $\Y_k\to...\to\Y_0$ lies over $y$. 
    

    It follows from (i) that if $y'\in\Y_k$ is a closed point lying over $y$, the sets of irreducible components of $D_k$ containing $y'$ and of irreducible components of $(\divi_0(f)\tx{ mod }n)^\tx{red}$ containing $y'$ are the same, so that $y'$ is a multiplicity $2$ simple normal crossings point of $D_k$ if and only if it is a multiplicity $2$ simple normal crossings point of $(\divi_0(f)\tx{ mod }n)^\tx{red}$. 

    Suppose now there is a multiplicity $2$ simple normal crossings point $z$ of $D_k$ that is not contained in the preimage of $y$ on $\Y_k$, and let $\Gamma_k\neq\Gamma_k'$ be the irreducible components of $D_k$ containing $z$. Then $\Gamma_k$ and $\Gamma_k'$ are the strict transforms of irreducible components $\Gamma$ and $\Gamma'$ of $F_y$ that intersect at the image $w$ of $z$ in $\Y$. Since $\Gamma$ and $\Gamma'$ also intersect at $y$, Lemma \ref{lemma: unique intersection points}(iii) applied to $\Gamma$ and $\Gamma'$ implies $y=w$, contradicting that $z$ is not in the preimage of $y$ on $\Y_k$. This shows (ii).
\end{proof}

\section{Explicit Description of the Local $n$-Resolution}\label{section 3}

Throughout this section, let $\Y$ be a regular model of $\P^1_K$, let $\X\to\Y$ be the normalization of $\Y$ in $K(X)$, let $D=(\tx{div}_0(f)\tx{ mod }n)^\tx{red}$ be the branch divisor of $\X\to\Y$, and let $y$ be a multiplicity $2$ point of $D$. Let $F_y$ be the sum of the irreducible components of $D$ that contain $y$, with multiplicities inherited from $\tx{div}_0(f)$.

\begin{remark}\label{rmk: $F_y$ remark}
Since $y$ is a multiplicity $2$ point of $D$, the support of $F_y$ consists either of a single irreducible component, in which $y$ is a point of multiplicity $2$, or of two distinct irreducible components, both of which are regular at $y$. In either case, since $F_y$ is supported on components of $D$, each irreducible component of $F_y$ is ramified. Moreover, by Lemma \ref{lemma: unique intersection points new}, the point $y$ is the unique singular point of $F_y$, and by Lemma \ref{lemma: unique intersection points}(ii), if $F_y$ does not have simple normal crossings, at least one irreducible component of $F_y$ is horizontal.
\end{remark}

\begin{defn}
 The point $y$ is called a multiplicity $2$ point of $D$ of \textit{Type I} or \textit{Type II} according as the support of $F_y$ consists of one or two irreducible components.
\end{defn}

\begin{notation}\label{notation: initial parameters}\hfill
\begin{itemize}
    \item Suppose $y$ is a point of Type I. Let $\Gamma$ be the unique irreducible component of $F_y$, let $\nu_0=\ord_\Gamma(f)$, and let $e_0=\frac{n}{\gcd(n,\nu_0)}\geq 2$ be the ramification index of $\Gamma$. Additionally, let $\nu_1\colonequals\mult_y(F_y)=2\nu_0$ and $e_1=\frac{n}{\gcd(n,\nu_1)}$.

    \item Suppose $y$ is a point of Type II. Let $\Gamma$ and $\Gamma'$ be the two irreducible components of $F_y$, let $\nu_0=\ord_\Gamma(f)$ and $\nu_0'=\ord_{\Gamma'}(f)$, and let $e_0=\frac{n}{\gcd(n,\nu_0)}\geq 2$ and $e_0'=\frac{n}{\gcd(n,\nu_0')}\geq 2$ be the ramification indices of $\Gamma$ and $\Gamma'$. Additionally, let $\nu_1\colonequals\mult_y(F_y)=\nu_0+\nu_0'$ and $e_1=\frac{n}{\gcd(n,\nu_1)}$.
\end{itemize}
\end{notation}

\begin{notation}\label{notation: multiplicity 2 point}
    Let $...\to\Y_k\to...\to\Y_0\colonequals\Y$ be the local $n$-resolution of $D$ at $y$, defined in \ref{dfn: resolution of local branch divisor}. For each $k\geq 0$ for which $\Y_k$ is defined:
    \begin{itemize}
        \item Let $F_k$ and $F_k^*$ respectively be the strict and total transforms of $F_y$ under the map $\Y_k\to\Y_0$, let $D_k=(F_k^*\tx{ mod }n)^\tx{red}$, and let $S_k$ be the set of multiplicity $2$ simple normal crossings points of $D_k$.
        \item Let $\Gamma_k$ be strict transform of $\Gamma$ under $\Y_k\to\Y_0$, and when $y$ is a point of Type II, let $\Gamma'_k$ be the strict transform of $\Gamma'$ under $\Y_k\to\Y_0$.
    \end{itemize}
    Additionally, for each $k\geq 1$ for which $\Y_k$ is defined:
    \begin{itemize}
        \item Let $y_k\in\Y_{k-1}$ be the center of the blow-up $\Y_k\to\Y_{k-1}$, and let $\nu_k=\mult_{y_k}(F^*_{k-1})$ and $\mu_k=\mult_{y_k}(D_{k-1})$.
        \item Let $E_k$ be the exceptional divisor of $\Y_k\to\Y_{k-1}$, and let $e_k$ be the ramification index of $E_k$ with respect to the normalization $\X_k\to\Y_k$ of $\Y_k$ in $K(X)$.
        \item By abuse of notation, for each $1\leq l\leq k-1$, conflate\footnote{Since the ramification index of the divisor $E_l$ on $\Y_l$ is equal to the ramification index of its strict transform under the map $\Y_k\to\Y_l$, the meaning of $e_l$ is unambiguous. Furthermore, since the divisor $\Gamma_k$ lies on $\Y_k$ by definition, the intersection number $(\Gamma_k,E_l)$ can mean only the intersection number of divisors on $\Y_k$, and similarly for $(\Gamma_k',E_l)$ if $y$ is of Type II. Likewise, the multiplicity $\mult_{y_{k+1}}(E_l)$ is unambiguous since $y_{k+1}\in\Y_k$ by definition.}
        $E_l$ with its strict transform under the map $\Y_k\to\Y_l$.
    \end{itemize} 
\end{notation}

\begin{remark}\label{rmk: notation is consistent}
    We note that if $\Y_k$ is defined for $k=1$, then $y_1=y$, and the definitions of $\nu_1$ and $e_1$ in Notations \ref{notation: initial parameters} and \ref{notation: multiplicity 2 point} coincide. For $\nu_1$, this is immediate, and for $e_1$, this follows from Lemma \ref{lemma: multiplicity in div_0(f)}(iii).
\end{remark}

\begin{prop}\label{prop: structure of exceptional divisor}
    Suppose the local $n$-resolution of $D$ at $y$ begins with the  sequence of blow-ups $\Y_k\to...\to\Y_0\equalscolon\Y$, and the strict transform $F_k$ does not have simple normal crossings.
    \begin{enumerate}[\upshape (i)]
        \item The divisor $F_k$ contains a unique point $y_{k+1}$ lying over $y$ via the map $\Y_k\to\Y_0$, and $y_{k+1}$ is not a simple normal crossings point of $F_k$.
        \item The local $n$-resolution continues with the blow-up $\Y_{k+1}\to\Y_k$ centered at $y_{k+1}$. 
        \item If $k\geq 1$, then $y_{k+1}$ lies in $E_k$ but not $E_l$ for $1\leq l\leq k-1$. 
        \item The exceptional divisor of $\Y_{k+1}\to\Y_0$ is a chain with dual graph \[E_1-E_2-...-E_{k+1}\]
        \item We have $(\Gamma_{k+1},E_{k+1})=\mult_{y_{k+1}}(\Gamma_k)=\mult_y(\Gamma)$,
        and if $y$ is of Type II, then also $(\Gamma'_{k+1},E_{k+1})=\mult_{y_{k+1}}(\Gamma'_k)=\mult_y(\Gamma')$ and
        $(\Gamma_{k+1},\Gamma_{k+1}')=(\Gamma,\Gamma')-(k+1)$.
        \item For all $1\leq l\leq k+1$, letting $\c{l}$ denote the least residue of $l$ mod $e_1$, we have
        \[\ord_{E_l}(f)\equiv \nu_l=l\nu_1\equiv\c{l}\nu_1\mod n,\quad\quad e_l=\frac{n}{\gcd(\nu_{l},n)}=\frac{e_1}{\gcd(\c{l},e_1)}\]
        and \[\mu_l=\begin{cases}2&e_1\mid l-1\\3&e_1\nmid l-1\end{cases}\]
        \item We have \[F_{k+1}^*=F_{k+1}+\sum_{l=1}^{k+1} \nu_l E_l,\quad\quad D_{k+1}=F_{k+1}^{\red}+\sum_{1\leq l\leq k+1:e_1\nmid l} E_l\]
        More explicitly, if $\alpha,\beta$ are the unique non-negative integers such that $k+1=\alpha e_1+\beta$ and $\beta\leq e_1-1$, then 
         \[D_{k+1}=F_{k+1}^{\red}+\sum_{l=0}^{\alpha-1}\sum_{i=1}^{e_1-1} E_{le_1+i}+\sum_{i=1}^\beta E_{\alpha e_1+i}\]
    \end{enumerate}
\end{prop}
\begin{proof}
    The proof is by induction on $k$, noting that if $F_k$ does not have simple normal crossings, neither does $F_l$ for any $0\leq l\leq k-1$.
    
    First suppose $k=0$, so that $F_k=F_k^*=F_y$. Then (i), (ii), (iii), and (iv) are trivial, and (v) follows from Lemma \ref{lemma: intersection number decreases with blow-up}(i) and Lemma \ref{lemma: intersection number decreases with blow-up}(ii). We next have $\nu_1\colonequals\mult_y(F_y)$, so that using \cite[Proposition 9.2.23]{Liu_book} for the first equality of each line and Lemma \ref{lemma: multiplicity in div_0(f)}(iii) for the congruence in the second line, we obtain \[F^*_1=F_1+\mult_y(F_y)E_1=F_1+\nu_1E_1\] and \[\ord_f(E_1)=\mult_y(\divi_0(f))\equiv \mult_y(F_y)=\nu_1\tx{ mod }n\]
    This shows the non-trivial parts of (vi) and (vii) for $k=0$.
    
    Suppose now $k\geq 1$ and the statements hold for all $0\leq l\leq k-1$. We first show (i) for $k$. By Remark \ref{rmk: $F_y$ remark}, we can assume without loss of generality that $\Gamma$ is horizontal, so that $\Gamma_k$ is horizontal as well. Then $\Gamma_k$ contains a unique closed point $y_{k+1}$ by Lemma \ref{lemma: unique intersection points new}, and since $\Gamma_k$ surjects onto $\Gamma$ via $\Y_k\to\Y_0$, the point $y_{k+1}$ lies over $y$.
    
    If $y$ is of Type I, then $y_{k+1}$ is the unique closed point of $F_k$, and the assumption that $F_k$ does not have simple normal crossings implies $F_k$ does not have simple normal crossings at $y_{k+1}$. If instead $y$ is of Type II, then $\tx{Supp}(F_y)$ consists of the regular irreducible components $\Gamma_k$ and $\Gamma_k'$, and the hypothesis that $F_k$ does not have simple normal crossings implies $\Gamma_k$ and $\Gamma_k'$ intersect non-transversely at $y_{k+1}$. In this case, either $\Gamma_k'$ is horizontal and contains a unique closed point by Lemma \ref{lemma: unique intersection points new}, or $\Gamma_k'$ is vertical and intersects the exceptional divisor of $\Y_k\to\Y_0$ in at most one point by (iv) for $k-1$ and Lemma \ref{lemma: unique intersection points}(iii).
    
    In all cases, we see $y_{k+1}$ is the unique point of $F_k$ lying over $y$ via $\Y_k\to\Y_0$, and $y_{k+1}$ is not a simple normal crossings point of $F_k$, showing (i) for $k$. Statement (ii) for $k$ then follows by the definition of the local $n$-resolution of $D$ at $y$.
    
    We next show (iii) for $k$. By (v) for $k-1$, we have $(\Gamma_k,E_k)=\mult_y(\Gamma)\geq 1$, so that $y_{k+1}$ is contained in $E_k$. Furthermore, if $k\geq 2$, then using Lemma \ref{lemma: intersection number decreases with blow-up}(ii) and the regularity of $E_{k-1}$ for the first equality, and using (v) for $k-2$ and $k-1$ for the second equality, we have \[(\Gamma_k, E_{k-1})=(\Gamma_{k-1},E_{k-1})-\mult_{y_k}(\Gamma_{k-1})=\mult_y(\Gamma)-\mult_y(\Gamma)=0\] so that $y_{k+1}$ is not contained in $E_{k-1}$. Statement (iii) for $k$ then follows from (iv) for $k-1$.
    
    Statement (iv) for $k$ follows from (iii) for $k$, (iv) for $k-1$, and the equalities $(E_k,E_{k-1})=\mult_{y_k}(E_{k-1})=1$, which hold by Lemma \ref{lemma: intersection number decreases with blow-up}(i) and the regularity of $E_{k-1}\cong\P^1_k$.


    We next show (v) for $k$. By Lemma \ref{lemma: intersection number decreases with blow-up}(i), we have $(\Gamma_{k+1},E_{k+1})=\mult_{y_{k+1}}(\Gamma_k)$, and by Lemma \ref{lemma: intersection number decreases with blow-up}(iii) for the second inequality and (v) for $k-1$ for the equality, we have \[1\leq\mult_{y_{k+1}}(\Gamma_k)\leq\mult_{y_k}(\Gamma_{k-1})=\mult_y(\Gamma)\leq 2\]
    If $y$ is of Type I, the assumption that $F_k$ does not have simple normal crossings implies $\mult_{y_{k+1}}(\Gamma_k)\geq2$, so that $\mult_{y_{k+1}}(\Gamma_k)=\mult_y(\Gamma)=2$. If instead $y$ is of Type II, then $\Gamma,\Gamma',\Gamma_k,$ and $\Gamma'_k$ are all regular, hence \[\mult_y(\Gamma)=\mult_y(\Gamma')=\mult_{y_{k+1}}(\Gamma_k)=\mult_{y_{k+1}}(\Gamma_k')=1\] In both cases, we see $\mult_{y_{k+1}}(\Gamma_k)=\mult_y(\Gamma)$.
    
    Additionally, if $y$ is of Type II, we have $(\Gamma_{k+1}',E_{k+1})=\mult_{y_{k+1}}(\Gamma_k')=1$ by Lemma \ref{lemma: intersection number decreases with blow-up}(i), and by Lemma \ref{lemma: intersection number decreases with blow-up}(ii) for the first equality and (v) for $k-1$ for the second equality, we have
    \[(\Gamma_{k+1},\Gamma_{k+1}')=(\Gamma_k,\Gamma_k')-\mult_{y_{k+1}}(\Gamma_k)\cdot\mult_{y_{k+1}}(\Gamma_k')=((\Gamma,\Gamma')-k)-1=(\Gamma,\Gamma')-(k+1)\]

    We next show (vi) for $k$. By (vi) for $k-1$, it suffices to assume $l=k+1$. By \cite[Proposition 9.2.23]{Liu_book} for the equality and Lemma \ref{lemma: multiplicity in div_0(f)}(iii) for the congruence, we have
    \[\ord_{E_{k+1}}(f)=\mult_{y_{k+1}}(\divi_0(f))\equiv \nu_{k+1}\mod n\]
    hence 
    \begin{equation}\label{eq: ram index}
    e_{k+1}=\frac{n}{\gcd(\nu_{k+1},n)}        
    \end{equation}
    
    Next, the equation $\mult_{y_{k+1}}(\Gamma_k)=\mult_y(\Gamma)$ from (v), combined with $\mult_{y_{k+1}}(\Gamma'_k)=\mult_y(\Gamma')$ if $y$ is of Type II, implies
    \begin{equation}\label{eq: mults in F_k and F_k^red}
        \mult_{y_{k+1}}(F_k)=\mult_{y}(F_y)=\nu_1,\quad\quad\mult_{y_{k+1}}(F_k^\tx{red})=2
    \end{equation}
    By statement (vii) for $k-1$ for the first equality of each line, statement (iii) for $k$ and the regularity of $E_k$ for the second equality of each line, Equation \eqref{eq: mults in F_k and F_k^red} for the third equality of the each line, and statement (vi) for $k-1$ for the fourth equality of the first line, we have
    \[\nu_{k+1}=\mult_{y_{k+1}}(F_k+\sum_{i=1}^k\nu_iE_i)=\mult_{y_{k+1}}(F_k)+\nu_k=\nu_1+\nu_k=(k+1)\nu_1\]
    and
    \[\mu_{k+1}=\mult_{y_{k+1}}(F_k^{\red}+\sum_{1\leq l\leq k:e_1\nmid l} E_l)=\mult_{y_{k+1}}(F_k^{\red})+\begin{cases}
        0 & e_1\mid k\\
        1 & e_1\nmid k
    \end{cases}=\begin{cases}
        2 & e_1\mid k\\
        3 & e_1\nmid k
    \end{cases}\]
    Since $e_1=\frac{n}{\gcd(n,\nu_1)}$, we have $n\mid e_1\nu_1$, so that $(k+1)\nu_1\equiv\c{k+1}\nu_1\mod n$, where $\c{k+1}$ is the least residue of $k+1$ mod $e_1$. By Equation \eqref{eq: ram index} and Proposition \ref{prop: ramification index simplified} of the Appendix, we also obtain a second formula for $e_{k+1}$, namely
    \[e_{k+1}=\frac{e_1}{(\c{k+1},e_1)}\]
     
    Finally, statement (vii) for $k$ follows from (vi) for $k$ and \cite[Proposition 9.2.23]{Liu_book}.
\end{proof}

\begin{prop}\label{prop: local resolution desingularizes components}
There is a minimal $j\geq 0$ such that $\Y_j$ is defined and $F_j$ has simple normal crossings. If $y$ is a point of Type II, then explicitly $j=(\Gamma,\Gamma')-1$.
\end{prop}
\begin{proof}
    If $y$ is of Type II, this follows from Proposition \ref{prop: structure of exceptional divisor}(v). Suppose instead $y$ is of Type I. Then for all $k\geq 0$ for which $\Y_k$ is defined, the divisor $F_k$ has simple normal crossings if and only if $\Gamma_k$ is regular. The sequence of blow-ups $\Y_j\to...\to\Y_0$, if it exists, is defined by blowing up the unique singular point $y_{k+1}$ of $\Gamma_k$ for each $k\geq 0$ until $\Gamma_k$ is regular. Since $\O_K$ is excellent, the same argument as for \cite[Lemma 9.2.32]{Liu_book} shows this process terminates.
\end{proof}

\begin{notation}\label{notation: alpha beta}
    Let $j$ be as in Proposition \ref{prop: local resolution desingularizes components}, and let $\alpha,\beta$ be the unique non-negative integers such that $j=\alpha e_1+\beta$ and $\beta\leq e_1-1$.
\end{notation}

\begin{prop}\label{prop: description of Y_j}
    Keep Notation \ref{notation: alpha beta}. 
    \begin{enumerate}[\upshape (i)]
        \item The divisor $F_j$ contains a unique point $y_{j+1}$ lying over $y$ via the map $\Y_j\to\Y_0$. If $j\geq 1$, the point $y_{j+1}$ lies in $E_j$ but not $E_l$ for $1\leq l\leq j-1$, and $F_j^{\red}$ and $E_j$ meet with intersection multiplicity $2$ at $y_{j+1}$.
        \item The point $y_{j+1}$ is a simple normal crossings point of $D_j$ if and only if $\beta=0$.
        \item If $\beta=0$, the local $n$-resolution of $D$ at $y$ terminates at the stage $\Y_j\to...\to\Y_0$, and the divisor $D_j$ has simple normal crossings. Otherwise, if $\beta\geq 1$, the local $n$-resolution continues with the blow-up $\Y_{j+1}\to\Y_j$ centered at $y_{j+1}$. 
        \item The set $S_j$ consists of the unique intersection point $z_{l,i}$ of the divisors $E_{le_1+i}$ and $E_{le_1+i-1}$ for each pair of integers $(l,i)$ such that $0\leq l\leq\alpha-1$ and $2\leq i\leq e_1-1$, or $l=\alpha$ and $2\leq i\leq \beta$, with the addition of the point $y_{j+1}$ if  $\beta=0$ and $y$ is of Type II.
        \item Keep the notation of \textup{(iv)}. The order and ramification pairs of $z_{l,i}$ are respectively \[(i\nu_1,(i-1)\nu_1)),\quad(\frac{e_1}{\gcd(i,e_1)},\frac{e_1}{\gcd(i-1,e_1)})\]
        and if $\beta=0$ and $y$ is of Type II, the order and ramification pairs of $y_{j+1}$ are respectively $(\nu_0,\nu_0')$ and $(e_0,e_0')$.
    \end{enumerate}
\end{prop}
\begin{proof}
    If $j=0$, all statements are easy to check, so we assume $j\geq 1$. Then by Remark \ref{rmk: $F_y$ remark}, we can assume without loss of generality that $\Gamma$ is horizontal, so that $\Gamma_j$ is horizontal as well.

    We first show (i). The divisor $\Gamma_j$ contains a unique closed point $y_{j+1}$ by Lemma \ref{lemma: unique intersection points new}. If $y$ is of Type I, then $y_{j+1}$ is the unique closed point of $F_j$. If instead $y$ is of Type II, then by Proposition \ref{prop: structure of exceptional divisor}(v) for $k=j-1$ and Proposition \ref{prop: local resolution desingularizes components} for two equalities respectively, we have $(\Gamma_j,\Gamma_j')=(\Gamma,\Gamma')-j=1$, so that $\Gamma_j$ and $\Gamma_j'$ intersect, necessarily at the unique closed point $y_{j+1}$ of $\Gamma_j$. In this case, either $\Gamma_j'$ is horizontal and contains a unique closed point by Lemma \ref{lemma: unique intersection points new}, or $\Gamma_j'$ is vertical and intersects the exceptional divisor of $\Y_j\to\Y_0$ in at most one point by (iv) for $k=j-1$ and Lemma \ref{lemma: unique intersection points}(iii). In all cases, we see $y_{j+1}$ is the unique point of $F_j$ lying over $y$ via $\Y_j\to\Y_0$.
    
    Next by Proposition \ref{prop: structure of exceptional divisor}(v) for $k=j-1$, we see that if $y$ is of Type I, \[(F_j^{\red},E_j)=(\Gamma_j,E_j)=\mult_y(\Gamma)=2\] and if $y$ is of Type II, \[(F_j^{\red},E_j)=(\Gamma_j,E_j)+(\Gamma_j',E_j)=\mult_y(\Gamma)+\mult_y(\Gamma')=2\]
    Since $y_{j+1}$ is the unique point of intersection of $F_y$ and the exceptional divisor of $\Y_j\to\Y_0$, we see $y_{j+1}$ is contained in $E_j$, and $F_j^{\red}$ and $E_j$ meet with intersection multiplicity $2$ at $y_{j+1}$. Then $y_{j+1}$ is not contained in $E_l$ for $1\leq l\leq j-2$ by Proposition \ref{prop: structure of exceptional divisor}(iv) for $j-1$. Furthermore, if $k\geq 2$, then using Lemma \ref{lemma: intersection number decreases with blow-up}(ii) and the regularity of $E_{j-1}$ for the first equality, and using (v) for $j-2$ and $j-1$ for the second equality, we have \[(\Gamma_j, E_{j-1})=(\Gamma_{j-1},E_{j-1})-\mult_{y_j}(\Gamma_{j-1})=\mult_j(\Gamma)-\mult_j(\Gamma)=0\] so that $y_{j+1}$ is not contained in $E_{j-1}$. This concludes the argument for (i).
    
    By (i) and Proposition \ref{prop: structure of exceptional divisor}(vii) for $k=j-1$, we have \[D_j=F_j^{\red}+G+\begin{cases}0&\beta=0\\ E_j&\beta\geq 1\end{cases}\] for the divisor $G=\sum_{1\leq l\leq j-1:e_1\nmid l}E_l$, which satisfies $y_{j+1}\notin\tx{Supp}(G)$. Since $F_j$ has simple normal crossings and meets $E_j$ with intersection multiplicity $2$ at $y_{j+1}$, it follows that $D_j$ has simple normal crossings at $y_{j+1}$ if and only if $\beta=0$, showing (ii).
    
    Since $y_{j+1}$ is the unique point of $F_j$ lying over $y$, it follows that the local $n$-resolution terminates at the stage $\Y_j\to...\to\Y_0$ or continues with the blow-up $\Y_{j+1}\to\Y_j$ centered at $y_{j+1}$ according as $\beta=0$ or $\beta\geq 1$. Moreover, if $\beta=0$, then since $F_y$ and the vertical divisor $G$ individually have simple normal crossings and are disjoint by (i), it follows that $D_j$ has simple normal crossings. This concludes the argument for (iii).
    
    Statement (iv) follows from Proposition \ref{prop: structure of exceptional divisor}(iii) and (vii) for $k=j-1$, noting that the irreducible components of $F_j$ are regular, hence $\mult_{y_{j+1}}(D_j)=\mult_{y_{j+1}}(\Gamma_j)=1$ if $\beta=0$ and $y$ is of Type I, and $\mult_{y_{j+1}}(D_j)=\mult_{y_{j+1}}(\Gamma_j+\Gamma_j')=2$ if $\beta=0$ and $y$ is of Type II.

    Finally, statement (v) follows from Proposition \ref{prop: structure of exceptional divisor}(vi) for $k=j-1$ and Notation \ref{notation: initial parameters}.
\end{proof}

The proofs of the next three propositions are routine adaptations of those of Propositions \ref{prop: structure of exceptional divisor} and \ref{prop: description of Y_j} and are omitted for the sake of brevity. After the statement of each proposition, we comment on the key differences in the proofs.




\begin{prop}\label{prop: full local resolution for type II}
    Keep Notation \ref{notation: alpha beta}. Suppose $y$ is of Type II and $\beta\geq1$. Let $y_{j+1}$ and $\Y_{j+1}\to\Y_j$ be as in Proposition \ref{prop: description of Y_j}.
    \begin{enumerate}[\upshape (i)]
        \item The exceptional divisor of $\Y_{j+1}\to\Y_0$ is a chain with dual graph \[E_1-E_2-...-E_{j+1}\]
        \item The divisors $\Gamma_{j+1}$ and $\Gamma_{j+1}'$ are disjoint, intersect $E_{j+1}$ transversely in distinct points, and do not intersect $E_l$ for $1\leq l\leq j$.
        \item We have
        \[\mu_{j+1}=3,\quad \ord_{E_{j+1}}(f)\equiv\nu_{j+1}=(j+1)\nu_1\equiv(\beta+1)\nu_1\textup{ mod }n\]
        and
        \[e_{j+1}=\frac{n}{\gcd(\nu_{j+1},n)}=\frac{e_1}{\gcd(\beta+1,e_1)}\]
        \item We have \[F_{j+1}^*=F_{j+1}+\sum_{l=1}^{j+1} \nu_l E_l\] and
        \[D_{j+1}=F_{j+1}^{\red}+\sum_{l=0}^{\alpha-1}\sum_{i=1}^{e_1-1} E_{le_1+i}+\sum_{i=1}^\beta E_{\alpha e_1+i}+\begin{cases}
           E_{j+1} &\beta\leq e_1-2\\
           0 &\beta=e_1-1
        \end{cases}\]
        \item The local $n$-resolution of $D$ at $y$ terminates at the stage $\Y_{j+1}\to...\to\Y_0$, and the divisor $D_{j+1}$ has simple normal crossings.
        \item The set $S_{j+1}$ consists of the unique intersection point $z_{l,i}$ of the divisors $E_{le_1+i}$ and $E_{le_1+i-1}$ for each pair of integers $(l,i)$ such that $0\leq l\leq\alpha-1$ and $2\leq i\leq e_1-1$, or $l=\alpha$ and $2\leq i\leq \beta$, with the addition of the intersection points $z_1,z_2,z_3$ of $E_{j+1}$ with $E_j,\Gamma_{j+1},\Gamma_{j+1}'$ respectively if $\beta\leq e_1-2$.
        \item Keep the notation of \textup{(vi)}.
        The order and ramification pairs of $z_{l,i}$ are respectively \[(i\nu_1,(i-1)\nu_1),\quad(\frac{e_1}{\gcd(i,e_1)},\frac{e_1}{\gcd(i-1,e_1)})\] and if $\beta\leq e_1-2$, the order pairs of $z_1,z_2,z_3$ are respectively
        \[((\beta+1)\nu_1,\beta\nu_1),\quad((\beta+1)\nu_1,\nu_0),\quad((\beta+1)\nu_1,\nu_0')\]
        and their ramification pairs are respectively \[(\frac{e_1}{\gcd(\beta+1,e_1)},\frac{e_1}{\gcd(\beta,e_1)}),\quad (\frac{e_1}{\gcd(\beta+1,e_1)},e_0),\quad(\frac{e_1}{\gcd(\beta+1,e_1)},e_0')\]
    \end{enumerate}
\end{prop}
\begin{proof}[Comments]
    The proof is analogous to those of Propositions \ref{prop: structure of exceptional divisor} and \ref{prop: description of Y_j} with the key difference that in this case, the divisors $\Gamma_j$ and $\Gamma_j'$ on $\Y_j$ meet \textit{transversely} at $y_{j+1}$ by Proposition \ref{prop: local resolution desingularizes components}, so that their strict transforms $\Gamma_{j+1}$ and $\Gamma_{j+1}'$ after the blow-up $\Y_{j+1}\to\Y_j$ at $y_{j+1}$ are disjoint by Lemma \ref{lemma: intersection number decreases with blow-up}(ii). In particular, the local $n$-resolution of $D$ at $y$ terminates at the stage $\Y_{j+1}\to...\to\Y_0$. The points of intersection of $\Gamma_{j+1}$ and $\Gamma_{j+1}'$ with $E_{j+1}$ are multiplicity $2$ simple normal crossings points of $D_{j+1}$ if and only if $E_{j+1}$ is ramified. We have $\ord_{E_{j+1}}(f)\equiv(\beta+1)\nu_1\tx{ mod }n$ by the same argument as in Proposition \ref{prop: structure of exceptional divisor}(vi), and we have $1\leq\beta\leq e_1-1$, so this occurs if and only $e_1\nmid \beta+1$, if and only if $\beta\leq e_1-2$.
\end{proof}

\begin{prop}\label{prop: full local resolution for type I}
    Keep Notation \ref{notation: alpha beta}. Suppose $y$ is of Type I and $\beta\geq 1$, and let $\Y_{j+1}\to\Y_j$ be as in Proposition \ref{prop: description of Y_j}.
    \begin{enumerate}[\upshape (i)]
        \item The exceptional divisor of $\Y_{j+1}\to\Y_0$ is a chain with dual graph \[E_1-E_2-...-E_{j+1}\]
        \item The unique closed point $y_{j+2}$ of $F_{j+1}$ is the intersection point of $E_{j+1}$ and $E_j$, and $\Gamma_{j+1},E_{j+1}$ and $E_j$ meet pairwise transversely at $y_{j+2}$.
        \item We have
        \[\ord_{E_{j+1}}(f)\equiv\nu_{j+1}=(2j+1)\nu_0\equiv(2\beta+1)\nu_0\textup{ mod }n,\quad \mu_{j+1}=2\]
        and
        \[e_{j+1}=\frac{n}{\gcd(\nu_{j+1},n)}=\frac{e_0}{\gcd(2\beta+1,e_0)},\]
        \item We have \[F_{j+1}^*=F_{j+1}+\sum_{l=1}^{j+1}\nu_l E_l\]  and
        \[D_{j+1}=F_{j+1}^{\red}+\sum_{l=0}^{\alpha-1}\sum_{i=1}^{e_1-1} E_{le_1+i}+\sum_{i=1}^\beta E_{\alpha e_1+i}+\begin{cases}
           E_{j+1} &e_0\nmid(2\beta+1)\\
           0 &e_0\mid(2\beta+1)
        \end{cases}\]
        \item If $e_0\mid(2\beta+1)$, the local $n$-resolution of $D$ at $y$ terminates at the stage $\Y_{j+1}\to...\to\Y_0$, and the divisor $D_{j+1}$ has simple normal crossings. Otherwise, if $e_0\nmid(2\beta+1)$, the local $n$-resolution continues with the blow-up $\Y_{j+2}\to\Y_{j+1}$ centered at $y_{j+2}$
        \item The set $S_{j+1}$ consists of the unique intersection point $z_{l,i}$ of the divisors $E_{le_1+i}$ and $E_{le_1+i-1}$ for each pair of integers $(l,i)$ such that $0\leq l\leq\alpha-1$ and $2\leq i\leq e_1-1$, or $l=\alpha$ and $2\leq i\leq \beta$, with the addition of the point $y_{j+2}$ if $e_0\mid (2\beta+1)$.
        \item Keep the notation of \textup{(vi)}.
        The order and ramification pairs of $z_{l,i}$ are respectively \[(i\nu_1,(i-1)\nu_1),\quad(\frac{e_1}{\gcd(i,e_1)},\frac{e_1}{\gcd(i-1,e_1)})\]
        and if $e_0\mid(2\beta+1)$, those of $y_{j+2}$ are respectively \[(\beta\nu_1,\nu_0),\quad(\frac{e_1}{\gcd(\beta,e_1)},e_0)\]
    \end{enumerate}
\end{prop}
\begin{proof}[Comments]
    The proof is analogous to those of Propositions \ref{prop: structure of exceptional divisor} and \ref{prop: description of Y_j} with the key difference that in this case, the divisor $\Gamma_j$ is \textit{regular}. The divisors $\Gamma_j$ and $E_j$ meet with intersection multiplicity $2$ by Proposition \ref{prop: description of Y_j}(i), so by Lemma \ref{lemma: intersection number decreases with blow-up}(i) and (ii), their strict transforms $\Gamma_{j+1}$ and $E_j$ on $\Y_{j+1}$ and the exceptional divisor $E_{j+1}$ meet pairwise transversely at a common point $y_{j+2}$. The local $n$-resolution of $D$ at $y$ terminates at the stage $\Y_{j+1}\to...\to\Y_0$ if and only if $y_{j+2}$ is a simple normal crossings point of $D_{j+1}$, if and only if $E_{j+1}$ is unramified. We have $\ord_{E_{j+1}}\equiv(2\beta+1)\nu_0$ mod $n$ by an analogous argument to that of Proposition \ref{prop: structure of exceptional divisor}(vi) using $\mult_{y_{j+1}}(\Gamma_j)=1$, so that $E_{j+1}$ is unramified if and only if $e_0\mid(2\beta+1)$. 
\end{proof}

\begin{prop}\label{prop: full local resolution for type I continued}
   Keep Notation \ref{notation: alpha beta}. Suppose $y$ is of Type I, suppose $\beta\geq 1$ and $e_0\nmid(2\beta+1)$, and let $\Y_{j+2}\to\Y_{j+1}$ be as in Proposition \ref{prop: full local resolution for type I}.
    \begin{enumerate}[\upshape (i)]
        \item The exceptional divisor of $\Y_{j+2}\to\Y_0$ is a chain with dual graph \[E_1-E_2-...-E_{j-1}-E_j-E_{j+2}-E_{j+1}\]
        \item The unique closed point $y_{j+3}$ of $F_{j+2}$ lies in $E_{j+2}$ but not $E_l$ for $1\leq l\leq j+1$, and $\Gamma_{j+2}$ and $E_{j+2}$ meet transversely at $y_{j+3}$.
        \item We have
        \[\ord_{E_{j+2}}(f)\equiv\nu_{j+2}=(4j+2)\nu_0\equiv(4\beta+2)\nu_0\textup{ mod }n,\quad \mu_{j+2}=3\]
        and
        \[e_{j+2}=\frac{n}{\gcd(\nu_{j+2},n)}=\frac{e_0}{\gcd(4\beta+2,e_0)}\]
        \item We have \[F_{j+2}^*=F_{j+2}+\sum_{l=1}^{j+2} \nu_l E_l\]
        \[D_{j+2}=F_{j+2}^{\red}+\sum_{l=0}^{\alpha-1}\sum_{i=1}^{e_1-1} E_{le_1+i}+\sum_{i=1}^\beta E_{\alpha e_1+i}+E_{j+1}+\begin{cases}
           E_{j+2} &e_0\nmid(4\beta+2)\\
           0 &e_0\mid(4\beta+2)
        \end{cases}\]
        \item The local $n$-resolution of $D$ at $y$ terminates at the stage $\Y_{j+2}\to...\to\Y_0$, and the divisor $D_{j+2}$ has simple normal crossings.
        \item The set $S_{j+2}$ consists of the unique intersection point $z_{l,i}$ of the divisors $E_{le_1+i}$ and $E_{le_1+i-1}$ for each pair of integers $(l,i)$ such that $0\leq l\leq\alpha-1$ and $2\leq i\leq e_1-1$, or $l=\alpha$ and $2\leq i\leq \beta$, with the addition of the intersection points $z_1=y_{j+3},z_2,z_3$ of $E_{j+2}$ with $\Gamma_{j+2},E_j,E_{j+1}$ respectively if $e_0\nmid(4\beta+2)$.
        \item Keep the notation of \textup{(vi)}.
        The order and ramification pair of $z_{l,i}$ are respectively 
        \[(i\nu_1,(i-1)\nu_1),\quad (\frac{e_1}{\gcd(i,e_1)},\frac{e_1}{\gcd(i-1,e_1)})\] and if $e_0\nmid(4\beta+2)$, the order pairs of $z_1,z_2,z_3$ are respectively \[((4\beta+2)\nu_0,\nu_0),\quad((4\beta+2)\nu_0,\beta\nu_1),\quad((4\beta+2)\nu_0,(2\beta+1)\nu_0)\]
        and their ramification pairs are respectively \[(\frac{e_0}{\gcd(4\beta+2,e_0)},e_0),\quad(\frac{e_0}{\gcd(4\beta+2,e_0)},\frac{e_1}{\gcd(\beta,e_1)})\] and \[(\frac{e_0}{\gcd(4\beta+2,e_0)},\frac{e_0}{\gcd(2\beta+1,e_0)})\]
    \end{enumerate}
\end{prop}
\begin{proof}[Comments]
     The proof is again analogous to those of Propositions \ref{prop: structure of exceptional divisor} and \ref{prop: description of Y_j}. In this case, by Lemma \ref{lemma: intersection number decreases with blow-up}(i) and (ii), the divisors $\Gamma_{j+1}$, $E_j$, $E_{j+1}$ separate after the blow-up $\Y_{j+2}\to\Y_{j+1}$ at $y_{j+2}$, and their strict transforms meet the exceptional divisor $E_{j+2}$ in distinct points $z_1,z_2,z_3$, which are simple normal crossings points of $D_{j+2}$. In particular, the local $n$-resolution of $D$ at $y$ terminates at the stage $\Y_{j+2}\to...\to\Y_0$. The points $z_1,z_2,z_3$ are multiplicity $2$ points of $D_{j+2}$ if and only if $E_{j+2}$ is ramified. We have $\ord_{E_{j+2}}\equiv(4\beta+2)\nu_0$ mod $n$ by an analogous argument to that of Proposition \ref{prop: structure of exceptional divisor}(vi) using that $y_{j+2}$ is a regular point of $\Gamma_{j+1}$, $E_j$, and $E_{j+1}$, so that $E_{j+1}$ is ramified if and only if $e_0\nmid(4\beta+2)$.
\end{proof}

\section{Lower Bounds on $\cdc$ and $\cedc$}\label{section 4}

Keep all notation of Section \ref{section 3}. In this section, we use the explicit descriptions of the local $n$-resolution of $D$ at $y$ from Section \ref{section 3} to give lower bounds on $\cedc(y)$ and $\cdc(y)$. In particular, we show  $\cedc(y)$ is always non-negative (Propositions \ref{cor: cedc positivity for type II} and \ref{cor: cedc positivity for type I}) and $\cdc(y)$ is non-negative except under highly restrictive conditions on $n$ and the structure of $D$ at $y$ (Propositions \ref{prop: full positivity for type II} and \ref{prop: full positivity for type I}).

The idea is simple:\footnote{See also \cite[Section 1.2]{Part1}.} By Remark \ref{rmk: discussion of bf, cf, etc.}, we have $\cedc(y)\geq\mbf(y)$ and $\cdc(y)\geq \blf(y)$. By inspection, the only negative terms in the sums defining $\mbf(y)$ and $\blf(y)$ come from blow-ups in the local $n$-resolution of $D$ at $y$ centered at points $y_k$ for which $\mu_k=2$ and $n\nmid\nu_k$. In all cases for $\cedc(y)$, and in most cases for $\cdc(y)$, when such a blow-up occurs in the local $n$-resolution, there are enough positive contributions to $\cedc(y)$ or $\cdc(y)$, coming from blow-ups not of this form or from crossings points of $D_m$, to counteract the negative contribution of the blow-up at $y_k$. Only when the local $n$-resolution terminates very quickly and the ramification indices of the components of $F_y$ are small does the negative contribution to $\cdc(y)$ win out. The results in this section make this discussion precise.

\begin{prop}\label{prop: preliminary cdc}
    The sum $\blf(\Y_j\to...\to\Y_0,F_y)+\cf(\Y_j\to...\to\Y_0,F_y)$ is given by
    \begin{align*} \alpha(2n-\frac{2n}{e_1}+e_1(n-2))&+ \begin{cases}
            0& \beta=0\tx{ and Type I}\\
            \frac{n}{e_0}+\frac{n}{e_0'}-2&\beta=0\tx{ and Type II}\\
            (\beta-2)(n-2)-\frac{n}{e_1}(1+\gcd(\beta,e_1))&\beta\geq 1
    \end{cases}
    \end{align*}
\end{prop}
\begin{proof} 
    Keep the notation of Proposition \ref{prop: description of Y_j}. 
    
    \textbf{Case:} $e_1\leq 2$. In this case, we have $\beta\leq 1$. By Proposition \ref{prop: description of Y_j}(iv) and (v), the set of multiplicity $2$ simple normal crossings points of $D_j$ is empty if $y$ is of Type I or $\beta=1$, and consists of the single point $y_{j+1}$, with ramification pair $(e_0,e_0')$, if $y$ is of Type II and $\beta=0$. It follows that
    \begin{align*}
        \cf(\Y_j\to...\to\Y_0,F_y)&=\begin{cases}
            0&\beta=1 \tx{ or Type I}\\
            \frac{n}{e_0}+\frac{n}{e_0'}-2&\beta=0\tx{ and Type II}
        \end{cases}
    \end{align*}

    If $e_1=1$, then by Proposition \ref{prop: structure of exceptional divisor}(vi), we have $n\mid\nu_k$ and $\mu_k=2$ for all $1\leq k\leq j$, so that
    \[\blf(\Y_j\to...\to\Y_0,F_y)=j(n-2)=\alpha(n-2)\]

    If instead $e_1=2$, then by Proposition \ref{prop: structure of exceptional divisor}(vi), for all $1\leq k\leq j$, we have $n\nmid\nu_k$, $\mu_k=2$, and $\gcd(\nu_k,n)=\frac{n}{e_1}$ if $k$ is odd, and $n\mid\nu_k$ and $\mu_k=3$ if $k$ is even, so that
    \begin{align*}\blf(\Y_j\to...\to\Y_0,F_y)&=\alpha((2-n-\frac{2n}{e_1})+(5n-6))+\begin{cases}
        0&\beta=0\\2-n-\frac{2n}{e_1}&\beta=1
    \end{cases}\\
    &=\alpha(3n-4)+\begin{cases}
        0&\beta=0\\2-2n&\beta=1
    \end{cases}
    \end{align*}

    The result follows by combining the formulas for the crossing factor and blow-up factor.
    
    \textbf{Case: $e_1\geq 3$}. By Proposition \ref{prop: description of Y_j}(iv), the set of multiplicity $2$ simple normal crossings points of $D_j$ consists of the points $z_{l,i}$ for each pair of integers $(l,i)$ such that $0\leq l\leq\alpha-1$ and $2\leq i\leq e_1-1$, or $l=\alpha$ and $2\leq i\leq \beta$, with the addition of the point $y_{j+1}$ if $\beta=0$ and $y$ is of Type II. By Proposition \ref{prop: description of Y_j}(v), the ramification pair of $z_{l,i}$ is $(\frac{e_1}{\gcd(i,e_1)},\frac{e_1}{\gcd(i-1,e_1)})$, and if $\beta=0$ and $y$ is of Type II, the ramification pair of the point $y_{j+1}$ is $(e_0,e_0')$. We obtain
    \begin{align*}
        &\cf(\Y_j\to...\to\Y_0,F_y)\\=&\alpha\cdot\sum_{i=2}^{e_1-1}((\frac{n}{e_1}\cdot\gcd(i,e_1)-1)+(\frac{n}{e_1}\cdot\gcd(i-1,e_1)-1))\\
        &\quad+\sum_{i=2}^\beta((\frac{n}{e_1}\cdot\gcd(i,e_1)-1)+(\frac{n}{e_1}\cdot\gcd(i-1,e_1)-1))\\
        &\quad+\begin{cases}
            0&\beta=0\tx{ and Type I, or } \beta\geq 1\\
            \frac{n}{e_0}+\frac{n}{e_0'}-2&\beta=0\tx{ and Type II}\\
        \end{cases}\\
        =&\alpha(\frac{2n}{e_1}-2+2\cdot\sum_{i=2}^{e_1-2}(\frac{n}{e_1}\cdot\gcd(i,e_1)-1))\\
        &\quad+\begin{cases}
            0& \beta=0 \tx{ and Type I, or } \beta=1\\
            \frac{n}{e_0}+\frac{n}{e_0'}-2&\beta=0\tx{ and Type II}\\
            \frac{n}{e_1}(1+\gcd(\beta,e_1))-2+2\cdot\sum_{i=2}^{\beta-1}(\frac{n}{e_1}\cdot\gcd(i,e_1)-1)&\beta\geq 2
        \end{cases}
    \end{align*}
    
    Next by Proposition \ref{prop: structure of exceptional divisor}(vi), for all $1\leq k\leq j$, we have $n\mid\nu_k$ if and only if $k\equiv 0$ mod $e_1$, we have $\mu_k=2$ or $\mu_k=3$ according as $k\equiv 1$ mod $e_1$, and we have $\gcd(\nu_k,n)=\frac{n}{e_1}\cdot\gcd(\c{k},e_1)$, where $\c{k}$ is the least residue of $k$ mod $e_1$. It follows that
    \begin{align*}
    \blf(\Y_j\to...\to\Y_0,F_y)
    &=\alpha((2-n-\frac{2n}{e_1})+\sum_{i=2}^{e_1-1}(n-\frac{2n}{e_1}\cdot\gcd(i,e_1))+(5n-6))\\
    &\quad+\begin{cases}
    0&\beta=0\\
    (2-n-\frac{2n}{e_1})+\sum_{i=2}^{\beta}(n-\frac{2n}{e_1}\cdot\gcd(i,e_1))&\beta\geq1\end{cases}
    \end{align*}
    In all, we obtain
    \begin{align*}
    &\blf(\Y_j\to...\to\Y_0,F_y)+\cf(\Y_j\to...\to\Y_0,F_y)\\
    =&\alpha((2-n-\frac{2n}{e_1})+\sum_{i=2}^{e_1-1}(n-\frac{2n}{e_1}\cdot\gcd(i,e_1))+(5n-6))\\
    &\quad+\alpha(\frac{2n}{e_1}-2+2\cdot\sum_{i=2}^{e_1-2}(\frac{n}{e_1}\cdot\gcd(i,e_1)-1))\\
    &\quad+\begin{cases}
            0& \beta=0\tx{ and Type I}\\
            \frac{n}{e_0}+\frac{n}{e_0'}-2&\beta=0\tx{ and Type II}\\
            2-n-\frac{2n}{e_1}&\beta=1\\
            (2-n-\frac{2n}{e_1})+\sum_{i=2}^{\beta}(n-\frac{2n}{e_1}\cdot\gcd(i,e_1))&\\
            \quad+(\frac{n}{e_1}(1+\gcd(\beta,e_1))-2+2\cdot\sum_{i=2}^{\beta-1}(\frac{n}{e_1}\cdot\gcd(i,e_1)-1))&\beta\geq 2
    \end{cases}\\
    =&\alpha[2n-\frac{2n}{e_1}+e_1(n-2)]\\
    &\quad+\begin{cases}
            0& \beta=0\tx{ and Type I}\\
            \frac{n}{e_0}+\frac{n}{e_0'}-2&\beta=0\tx{ and Type II}\\
            (\beta-2)(n-2)-\frac{n}{e_1}(1+\gcd(\beta,e_1))&\beta\geq 1
    \end{cases}\qedhere
    \end{align*}
\end{proof}

For the remainder of this section, we separate the cases that $y$ is a multiplicity point of $D$ of Type I or II. We first suppose $y$ is of Type II, as in this case, it is easier to obtain bounds on $\cdc(y)$ and $\cedc(y)$.

\subsection{Bounds for Type II Points}

\begin{prop}\label{prop: cdc for type II}
    Suppose $y$ is a multiplicity $2$ point of $D$ of Type II. The sum $\blf(y)+\cf(y)$ is given by
    \begin{align*}
        \alpha(2n&-\frac{2n}{e_1}+e_1(n-2))\\
        &+ \begin{cases}
        \frac{n}{e_0}+\frac{n}{e_0'}-2&\beta=0\\
        \beta(n-2)-n-2+\frac{n}{e_0}+\frac{n}{e_0'}+\frac{n}{e_1}(\gcd(\beta+1,e_1)-1)&1\leq \beta\leq e_1-2\\
        2n-\frac{2n}{e_1}+e_1(n-2)&1\leq\beta=e_1-1
        \end{cases}
    \end{align*}
\end{prop}

\begin{proof} 
    Suppose $\beta=0$. By Proposition \ref{prop: local resolution desingularizes components}(iii), the local $n$-resolution of $D$ at $y$ terminates at the stage $\Y_j\to...\to\Y_0$, so that $\blf(y)=\blf(\Y_j\to...\to\Y_0,F_y)$ and $\cf(y)=\cf(\Y_j\to...\to\Y_0,F_y)$. The result then follows by Proposition \ref{prop: preliminary cdc}.
    
    Suppose now $\beta\geq 1$. By Proposition \ref{prop: full local resolution for type II}(v), the local $n$-resolution terminates after the additional blow-up $\Y_{j+1}\to\Y_j$. By Proposition \ref{prop: full local resolution for type II}(iii), we have $\mu_{j+1}=3$ and $\nu_{j+1}\equiv(\beta+1)\nu_1\tx{ mod }n$, so that $n\mid\nu_{j+1}$ if and only if $e_1\colonequals\frac{n}{\gcd(n,\nu_1)}\mid(\beta+1)$ if and only if $\beta=e_1-1$, where the last equivalence follows from the inequality $1\leq\beta\leq e_1-1$. We obtain
    \begin{align}\label{eq: Type II bf}
        \blf(y)=\blf(\Y_j\to...\to\Y_0,F_y)+\begin{cases}
        5n-6&\beta=e_1-1\\
        n-\frac{2n}{e_1}\cdot\gcd(\beta+1,e_1)&\beta\leq e_1-2
        \end{cases}
    \end{align}

    Next let $S_{j+1}'$ be the set of the points $z_{l,i}$ of in Proposition \ref{prop: full local resolution for type II}(vi). By Proposition \ref{prop: full local resolution for type II}(vi), the multiplicity $2$ simple normal crossings points of $D_{j+1}$ consists of the points of $S'_{j+1}$, with the addition of the intersection points $z_1,z_2,z_3$ of $E_{j+1}$ with $E_j,\Gamma_{j+1},\Gamma_{j+1}'$ respectively if $\beta\leq e_1-2$. In this case, by Proposition \ref{prop: full local resolution for type II}(vii), the ramification pairs of $z_1,z_2,$ and $z_3$ are respectively
    \[(\frac{e_1}{\gcd(\beta+1,e_1)},\frac{e_1}{\gcd(\beta,e_1)}),\quad(\frac{e_1}{\gcd(\beta+1,e_1)},e_0),\quad (\frac{e_1}{\gcd(\beta+1,e_1)},e_0')\]
    By Proposition \ref{prop: description of Y_j}(iv), there is an obvious bijection of $S_{j+1}'$ with the set $S_j$ of multiplicity $2$ simple normal crossings points of $D_j$, preserving ramification pairs. It follows that
    \begin{equation}\label{eq: Type II cf}
    \begin{aligned}
        \cf(y)=&\cf(\Y_j\to...\to\Y_0,F_y)\\
        &+\begin{cases}
        0&\beta=e_1-1\\
        3(\frac{n}{e_1}\cdot\gcd(\beta+1,e_1)-1)+(\frac{n}{e_1}\cdot\gcd(\beta,e_1)-1)&\\
        \quad+(\frac{n}{e_0}-1)+(\frac{n}{e_0'}-1)&\beta\leq e_1-2
        \end{cases}
    \end{aligned}
    \end{equation}

    \textbf{Case:} $1\leq \beta=e_1-1$. Using \eqref{eq: Type II bf} and \eqref{eq: Type II cf} for the first equality and Proposition \ref{prop: preliminary cdc} for the second equality, we obtain
    \begin{align*}
        \blf(y)+\cf(y)=&\blf(\Y_j\to...\to\Y_0,F_y)+\tx{cf}(\Y_j\to...\to\Y_0,F_y)+5n-6\\
        =&\alpha(2n-\frac{2n}{e_1}+e_1(n-2))+(e_1-3)(n-2)-\frac{n}{e_1}(1+\gcd(e_1-1,e_1))+5n-6\\
        =&\alpha(2n-\frac{2n}{e_1}+e_1(n-2))+2n-\frac{2n}{e_1}+e_1(n-2)
    \end{align*}
    
    \textbf{Case:} $1\leq\beta\leq e_1-2$.  Using \eqref{eq: Type II bf} and \eqref{eq: Type II cf} for the first equality and Proposition \ref{prop: preliminary cdc} for the second equality, we obtain
    \begin{align*}
        \tx{bf}(y)+\tx{cf}(y)=&\tx{bf}(\Y_j\to...\to\Y_0,F_y)+\tx{cf}(\Y_j\to...\to\Y_0,F_y)\\
        &+n-\frac{2n}{e_1}\cdot\gcd(\beta+1,e_1)+3(\frac{n}{e_1}\cdot\gcd(\beta+1,e_1)-1)+(\frac{n}{e_1}\cdot \gcd(\beta,e_1)-1)\\
        &+(\frac{n}{e_0}-1)+(\frac{n}{e_0'}-1)\\
        =&\alpha(2n-\frac{2n}{e_1}+e_1(n-2))+(\beta-2)(n-2)-\frac{n}{e_1}(1+\gcd(\beta,e_1))\\
        &+n-\frac{2n}{e_1}\cdot\gcd(\beta+1,e_1)+3(\frac{n}{e_1}\cdot\gcd(\beta+1,e_1)-1)+(\frac{n}{e_1}\cdot \gcd(\beta,e_1)-1)\\
        &+(\frac{n}{e_0}-1)+(\frac{n}{e_0'}-1)\\
        =&\alpha(2n-\frac{2n}{e_1}+e_1(n-2))+\beta(n-2)-n-2\\
        &+\frac{n}{e_0}+\frac{n}{e_0'}+\frac{n}{e_1}(\gcd(\beta+1,e_1)-1)\qedhere
    \end{align*}
\end{proof}

\begin{corollary}\label{cor: first positivity for type II}
    Suppose $y$ is a multiplicity $2$ point of $D$ of Type II. Then $\blf(y)+\cf(y)\geq0$ except possibly in the following case:
    \begin{table}[h!]
    \centering
    \renewcommand{\arraystretch}{1.4}
    \begin{tabular}{|cc|c|}
        \hline
        & \textbf{\textup{Case}} & $\blf(y)+\cf(y)$ \\
        \hline
        $j=1$ & $n=\max(e_0,e_0')\geq e_1\geq3$ & $\geq-2$ \\
        \hline
    \end{tabular}
    \end{table}
\end{corollary}
\begin{proof}
    This follows from Proposition \ref{prop: cdc for type II}, Proposition \ref{prop: proof of positivity for type II} of the Appendix, and the definitions of $\alpha$ and $\beta$.
\end{proof}

\begin{corollary}\label{cor: cedc positivity for type II}
    Suppose $y$ is a multiplicity $2$ point of $D$ of Type II. Then $\cedc(y)\geq0$.
\end{corollary}
\begin{proof}
    Let $m$ be the length of the local $n$-resolution of $D$ at $y$. By Remark \ref{rmk: discussion of bf, cf, etc.}, we have $\cedc(y)\geq m+\blf(y)+\cf(y)$, so that by Corollary \ref{cor: first positivity for type II}, it suffices to assume $j=1,$ $n=\max(e_0,e_0')\geq e_1\geq3$, and $\blf(y)+\cf(y)\geq -2$. In this case, we have $\beta=1$, so that $m=j+1=2$ by Proposition \ref{prop: full local resolution for type II}(vi), hence $\cedc(y)\geq m+\blf(y)+\cf(y)= 2-2=0$.
\end{proof}

The next result is stronger than necessary for Proposition \ref{prop: full positivity for type II} but is used as a standalone result in \cite{Part1} to obtain Theorem \ref{thm: reduction to cdc}.

\begin{prop}\label{prop: stronger positivity for type II}
    Suppose $y$ is a multiplicity $2$ point of $D$ of Type II. If $j=0$, so that $y$ is a simple normal crossings point of $D$, or if $e_0\neq e_0'$, then \[\cedc(y)\geq\cdc(y)\geq2\frac{n}{e_0}-2\geq0\]
\end{prop}
\begin{proof}
    Since $e_0\geq 2$ and $\cedc(y)\geq\cdc(y)$, it suffices to show $\tx{cdc}(y)\geq 2\frac{n}{e_0}-2$. 
    
    First suppose $j=0$, so that $m=0$ and $D_m=D_0=(F_y\tx{ mod }n)^\tx{red}$. Then $y$ is the unique multiplicity $2$ simple normal crossings point of $D_m$, and the ramification pair of $y$ is $(e_0,e_0')$. By Proposition \ref{prop: cdc for type II} with $\alpha=\beta=0$, we have
    \[\blf(y)+\cf(y)=\frac{n}{e_0}+\frac{n}{e_0'}-2\]
    If $e_0=e_0'$, then $\cdc(y)\geq\blf(y)+\cf(y)=2\frac{n}{e_0}-2$, as desired. Otherwise, if $e_0\neq e_0'$, we have $\cb(y)\geq \frac{n}{2}\geq\frac{n}{e_0}$ by Lemma \ref{lemma: quick crossing bonus},
    so that
    \[\cdc(y)=\blf(y)+\cf(y)+\cb(y)\geq\frac{n}{e_0}+\frac{n}{e_0'}-2+\frac{n}{2}\geq2\frac{n}{e_0}-1\]
    and we are done.

    Suppose now $j\geq 1$ and $e_0\neq e_0'$.
    If $\beta=0$, hence $\alpha=j\geq 1$, or if $1\leq\beta=e_1-1$, then by Proposition \ref{prop: cdc for type II}, we have 
    \[\tx{bf}(y)+\tx{cf}(y)\geq 2n-\frac{2n}{e_1}+e_1(n-2)\geq n-2\geq2\frac{n}{e_0}-2\]
    as desired. We can therefore assume $1\leq \beta\leq e_1-2$, hence
    $n\geq e_1\geq 3$.
    
    By Proposition \ref{prop: full local resolution for type II}(v), the local $n$-resolution of $D$ at $y$ terminates at the stage $\Y_{j+1}\to...\to\Y_0$. By Proposition \ref{prop: full local resolution for type II}(vi) and (vii), the intersection points $z_1,z_2,$ and $z_3$ of $E_{j+1}$ with $E_j$, $\Gamma_{j+1}$, and $\Gamma_{j+1}'$ respectively are multiplicity $2$ simple normal crossings points of $D_{j+1}$, and the ramification pairs of these points are $(e_{j+1},e_j),(e_{j+1},e_0),$ and $(e_{j+1},e_0')$ respectively, where $e_j=\frac{e_1}{\gcd(\beta,e_1)}$ and $e_{j+1}=\frac{e_1}{\gcd(\beta+1,e_1)}$.

    Since $e_0\neq e_0'$ by assumption, we have either $e_{j+1}\neq e_0$ or $e_{j+1}\neq e_0'$. By Lemma \ref{lemma: quick crossing bonus}, the points $z_1$ and $z_2$ together contribute at least $\frac{n}{2}$ to $\tx{cb}(y)$, so that 
    \begin{equation}\label{equation: prelim bound}
        \cdc(y)\geq\blf(y)+\cf(y)+\frac{n}{2}
    \end{equation}
    It therefore suffices to show $\blf(y)+\cf(y)\geq 2\frac{n}{e_0}-2-\frac{n}{2}$.

    If $\beta\geq 2$, this follows from Proposition \ref{prop: cdc for type II}, Equation \eqref{equation: b geq 2 bound} of the Appendix, and the inequality $\frac{n}{e_0'}\geq0\geq \frac{n}{e_0}-\frac{n}{2}$. If $\beta=1$ and $\alpha\geq1$, this follows from Proposition \ref{prop: cdc for type II}, Equation \eqref{equation: bound b=1,a>=1} of the Appendix, and the inequality $\frac{13}{3}n-8\geq n-2$. Finally, if $\beta=1$ and $\alpha=0$, this follows from Proposition \ref{prop: cdc for type II}, Equation \eqref{equation: bound $j=1$} of the Appendix, and the inequality $\frac{n}{e_0'}-2\geq\frac{n}{e_0}-\frac{n}{2}$, except in the case $e_0'=n$ and $e_0=2$, where in general \eqref{equation: bound $j=1$} implies only $\blf(y)+\cf(y)\geq\frac{n}{2}-3=\frac{n}{e_0}-3$.
    
    Assume this last case. Then $n$ is even, $j=\beta=1$, $e_j=e_1$, and $e_{j+1}=e_2=\frac{e_1}{\gcd(2,e_1)}$. If $e_1$ is odd, then $e_2=e_1$ is also odd, hence $e_2\neq e_0=2$ and $e_2\neq e_0'=n$, and by Lemma \ref{lemma: quick crossing bonus}, the points $z_1$ and $z_2$ both contribute at least $\frac{n}{2}$ to $\cb(y)$. If instead $e_1$ is even, then $e_2\neq e_1$, so that by Lemma \ref{lemma: quick crossing bonus}, the intersection point $z_3$ of $E_2$ and $(E_1)_2'$ contributes at least $\frac{n}{2}$ to $\cb(y)$. In either case, we can improve \eqref{equation: prelim bound} to 
    \[\cdc(y)\geq \blf(y)+\cf(y)+n\]
    Since $\blf(y)+\cf(y)\geq \frac{n}{2}-3$, we obtain \[\cdc(y)\geq \frac{n}{2}-3+n>n-2=2\frac{n}{e_0}-2\]
    and we are done.
\end{proof}

\begin{prop}\label{prop: full positivity for type II}
    Suppose $y$ is a multiplicity $2$ point of $D$ of Type II. Then  $\cdc(y)\geq0$ except in the following case:
    \begin{table}[h!]
    \centering    
    \renewcommand{\arraystretch}{1.4}
    \begin{tabular}{|cc|c|}
        \hline
        & \textbf{\textup{Case}} & $\cdc(y)$ \\
        \hline
        $j=1$ & $n=e_0=e_0'=e_1=3$ & $=-1$ \\
        \hline
    \end{tabular}
    \end{table}
\end{prop}
\begin{proof}
    We first note $\cdc(y)\geq\blf(y)+\cf(y)$ by Remark \ref{rmk: discussion of bf, cf, etc.}. By Corollary \ref{cor: first positivity for type II} and Proposition \ref{prop: stronger positivity for type II}, we can then assume $j=1$, $n=e_0=e_0'\geq e_1\geq 3$, and
    \begin{equation}\label{eq: cdc Type II}
        \cdc(y)=\blf(y)+\cf(y)+\cb(y)\geq-2+\cb(y)
    \end{equation}
    
    By Proposition \ref{prop: full local resolution for type II}(v), the local $n$-resolution of $D$ at $y$ consists of two blow-ups $\Y_2\to\Y_1\to\Y_0$. By Proposition \ref{prop: full local resolution for type II}(iii), (vi), and (vii), the multiplicity $2$ simple normal crossings points of $D_2$ are the intersection points $z_1$, $z_2$, and $z_3$ of $E_2$ with $E_1$, $\Gamma_2$, and $\Gamma'_2$ respectively, and the ramification pairs of these points are $(e_1,e_2),$ $(e_0,e_2),$ and $(e_0',e_2)$, where $e_2=\frac{e_1}{\gcd(2,e_1)}$. In Notation \ref{notation: initial parameters}, we have $\ord_{\Gamma}(f)=\nu_0$ and $\ord_{\Gamma'}(f)=\nu_0'$, and by Propositions \ref{prop: structure of exceptional divisor}(vi) and \ref{prop: full local resolution for type II}(iii), we have $\ord_{E_1}(f)\equiv\nu_1=\nu_0+\nu_0'$ and $\ord_{E_2}(f)\equiv \nu_2=2\nu_1=2(\nu_0+\nu_0')$ mod $n$.
    
    If $e_2\neq n$, then by Lemma \ref{lemma: quick crossing bonus} and the assumption $e_0=e_0'=n$, the contributions of $z_2$ and $z_3$ to $\cb(y)$ are both at least $\frac{n}{2}\geq 2$, hence $\cdc(y)\geq0$ by \eqref{eq: cdc Type II}. We can therefore assume $e_2=n$, hence $e_1=n$. It follows that $\nu_0,\nu_0',\nu_1,$ and $\nu_2$ are all prime to $n$. 
    
    Under this assumption, the expression for $\cb(y)$ simplifies to
    \begin{equation}\label{eq: cb Type II}
        \cb(y)=3n-3-\sum_{i=1}^3\HJL(z_i)
    \end{equation}
    where
    \begin{align*}
    \HJL(z_1)=&\L(n,-\nu_1^{-1}\nu_2)=\L(n,-2)\\
    \HJL(z_2)=&\L(n,-\nu_0^{-1}\nu_2)=\L(n,-(2+2\nu_0^{-1}\nu_0'))\\
    \HJL(z_3)=&\L(n,-(\nu_0')^{-1}\nu_2)=\L(n,-(2+2\nu_0(\nu_0')^{-1}))
    \end{align*}

    By Remark \ref{rmk: L bound}, the length $\HJL(z_1)$ is bounded above by the least residue of $-2$ mod $n$, which is $n-2$ since $n\geq 3$. It follows from \eqref{eq: cdc Type II} and \eqref{eq: cb Type II} that $\cdc(y)\geq 0$ unless $\HJL(z_2)=\HJL(z_3)=n-1$. Applying Remark \ref{rmk: L bound} again, we see this can occur only if $-(2+2\nu_0^{-1}\nu_0')\equiv-(2+2\nu_0(\nu_0')^{-1})\equiv-1$ mod $n$.
    
    In this case, we have $\nu_0^{-1}\nu_0'\equiv\nu_0(\nu_0')^{-1}$ mod $n$, so that $0\equiv \nu_0^2-(\nu_0')^2=(\nu_0+\nu_0')(\nu_0-\nu_0')$ mod $n$, hence $\nu_0\equiv \nu_0'$ mod $n$ by the condition that $\nu_1=\nu_0+\nu_0'$ is prime to $n$. The condition $-(2+2\nu_0^{-1}\nu_0')\equiv-1$ mod $n$ then gives $-4\equiv-1$ mod $n$, hence $n=3$.

    Conversely, if $n=3$, the condition that $\nu_1=\nu_0+\nu_0'$ is prime to $n$ implies $\nu_0=\nu_0'$. Then by Proposition \ref{prop: explicit hirzebruch-jung lengths} of the Appendix, we obtain $\HJL(z_1)=\L(3,1)=1 $ and $\HJL(z_2)=\HJL(z_3)=\L(3,2)=2$, so that $\cb(y)=1$ by \eqref{eq: cb Type II} and $\cdc(y)=-1$ by \eqref{eq: cdc Type II}.
\end{proof}

\subsection{Bounds for Type I Points}

\begin{prop}\label{prop: cdc for type I}
    Suppose $y$ is a multiplicity $2$ point of $D$ of Type I. The sum $\blf(y)+\cf(y)$ is given by
    \begin{align*}
        \alpha(2n&+e_1(n-2)-\frac{2n}{e_1})+\beta(n-2)\\
        &+ \begin{cases}
    0&\beta=0\\
    -n&\beta\geq1,e_0\mid (2\beta+1)\\
    n-\frac{2n}{e_1}&\beta\geq1,e_0\nmid (2\beta+1),e_0\mid (4\beta+2)\\
    -2n+\frac{n}{e_0}(\gcd(2,e_0)-1)(\gcd(2\beta+1,e_0)-1)&\beta\geq1,e_0\nmid (4\beta+2)\end{cases}
    \end{align*}
\end{prop}
\begin{proof}
    Suppose $\beta=0$. By Proposition \ref{prop: local resolution desingularizes components}(iii), the local $n$-resolution of $D$ at $y$ terminates at the stage $\Y_j\to...\to\Y_0$, so that $\blf(y)=\blf(\Y_j\to...\to\Y_0,F_y)$ and $\cf(y)=\cf(\Y_j\to...\to\Y_0,F_y)$. The result then follows by Proposition \ref{prop: preliminary cdc}.
    
    Suppose now $\beta\geq 1$. By Proposition \ref{prop: description of Y_j}(iii), the local $n$-resolution continues with the blow-up $\Y_{j+1}\to\Y_j$. By Proposition \ref{prop: full local resolution for type I}(iii), we have $\mu_{j+1}=2$ and $\nu_{j+1}\equiv(2\beta+1)\nu_0\tx{ mod }n$, so that $n\mid \nu_{j+1}$ if and only if $e_0\colonequals\frac{n}{\gcd(n,\nu_0)}\mid(2\beta+1)$.
    It follows that
    \begin{equation}\label{eq: Type I first bf} 
    \begin{aligned}
    \blf(\Y_{j+1}\to...\to\Y_0,F_y)=&\blf(\Y_j\to...\to\Y_0,F_y)\\
    &+
    \begin{cases}
        n-2& e_0\mid(2\beta+1)\\
        2-n-\frac{2n}{e_0}\cdot\gcd(2\beta+1,e_0)& e_0\nmid (2\beta+1)
    \end{cases}
    \end{aligned}
    \end{equation}

    Next let $S_{j+1}'$ be the set of the points $z_{l,i}$ of in Proposition \ref{prop: full local resolution for type I}(vi). By Proposition \ref{prop: full local resolution for type I}(vi), the multiplicity $2$ simple normal crossings points of $D_{j+1}$ consists of the points of $S'_{j+1}$, with the addition of the point $y_{j+2}$ if $e_0\mid(2\beta+1)$. In this case, by Proposition \ref{prop: full local resolution for type I}(vii), the ramification pairs of $y_{j+1}$ is
    \[(\frac{e_1}{\gcd(\beta,e_1)},e_0)\]
    By Proposition \ref{prop: description of Y_j}(iv), there is an obvious bijection of $S_{j+1}'$ with the set $S_j$ of multiplicity $2$ simple normal crossings points of $D_j$, preserving ramification pairs. It follows that
    \begin{equation}\label{eq: Type I first cf}
    \begin{aligned}
    \cf(\Y_{j+1}\to...\to\Y_0,F_y)
        =&\cf(\Y_j\to...\to\Y_0,F_y)\\
        &+\begin{cases}
        (\frac{n}{e_1}\cdot\gcd(\beta,e_1)-1)+(\frac{n}{e_0}-1)&e_0\mid(2\beta+1)\\
        0&e_0\nmid(2\beta+1)
        \end{cases}
    \end{aligned}
    \end{equation}

    \textbf{Case: $e_0\mid(2\beta+1)$}. By Proposition \ref{prop: full local resolution for type I}(v), the local $n$-resolution terminates at the stage $\Y_{j+1}\to...\to\Y_0$, so that $\blf(y)=\blf(\Y_{j+1}\to...\to\Y_0,F_y)$ and $\cf(y)=\cf(\Y_{j+1}\to...\to\Y_0,F_y)$. Using \eqref{eq: Type I first bf} and \eqref{eq: Type I first cf} for the first equality and Proposition \ref{prop: preliminary cdc} for the second equality, and noting that $e_1=\frac{e_0}{\gcd(2,e_0)}=e_0$ since $e_0\mid 2\beta+1$ is odd, we obtain
    \begin{align*}
\blf(y)+\cf(y)=&\blf(\Y_j\to...\to\Y_0,F_y)+\cf(\Y_j\to...\to\Y_0,F_y)\\
        &+n-2+(\frac{n}{e_1}\cdot\gcd(\beta,e_1)-1)+(\frac{n}{e_0}-1)\\
        =&\alpha(2n-\frac{2n}{e_1}+e_1(n-2))+(\beta-2)(n-2)-\frac{n}{e_1}(1+\gcd(\beta,e_1))\\
        &+n-2+(\frac{n}{e_1}\cdot\gcd(\beta,e_1)-1)+(\frac{n}{e_0}-1)\\
        =&\alpha(2n-\frac{2n}{e_1}+e_1(n-2))+\beta(n-2)-n
    \end{align*}

    Suppose now $e_0\nmid(2\beta+1)$. By Proposition \ref{prop: full local resolution for type I continued}(v), the local $n$-resolution terminates after the additional blow-up $\Y_{j+2}\to\Y_{j+1}$. By Proposition \ref{prop: full local resolution for type I continued}(iii), we have $\mu_{j+2}=3$ and $\nu_{j+2}\equiv(4\beta+2)\nu_0\tx{ mod }n$, so that $n\mid\nu_{j+2}$ if and only if $e_0\colonequals\frac{n}{\gcd(n,\nu_0)}\mid(4\beta+2)$. We see 
    \begin{equation}\label{eq: Type I second bf}
    \blf(y)=\blf(\Y_{j+1}\to...\to\Y_0,F_y)+\begin{cases}
        5n-6& e_0\mid (4\beta+2)\\
        n-\frac{2n}{e_0}\cdot\gcd(4\beta+2,e_0)&e_0\nmid (4\beta+2)
    \end{cases}\end{equation}

    Let $S_{j+2}'$ be the set of the points $z_{l,i}$ of in Proposition \ref{prop: full local resolution for type I continued}(vi). By Proposition \ref{prop: full local resolution for type I continued}(vi), the multiplicity $2$ simple normal crossings points of $D_{j+2}$ consists of the points of $S'_{j+2}$, with the addition of the intersection points $z_1,z_2,z_3$ of $E_{j+2}$ with $\Gamma_{j+2},E_j,E_{j+1}$ respectively if $e_0\nmid(4\beta+2)$.  In this case, by Proposition \ref{prop: full local resolution for type I continued}(vii), the ramification pairs of $z_1,z_2,$ and $z_3$ are respectively
        \[(\frac{e_0}{\gcd(4\beta+2,e_0)},e_0),\quad(\frac{e_0}{\gcd(4\beta+2,e_0)},\frac{e_1}{\gcd(\beta,e_1)}),\quad(\frac{e_0}{\gcd(4\beta+2,e_0)},\frac{e_0}{\gcd(2\beta+1,e_0)})\]
    By Proposition \ref{prop: full local resolution for type I}(vi), there is an obvious bijection of $S_{j+2}'$ with the set $S_{j+1}$ of multiplicity $2$ simple normal crossings points of $D_{j+1}$, preserving ramification pairs. It follows that
    \begin{equation}\label{eq: Type I second cf}
    \begin{aligned}
    \cf(y)
        =&\cf(\Y_{j+1}\to...\to\Y_0,F_y)\\
        &+\begin{cases}
        0&e_0\mid(4\beta+2)\\
        3(\frac{n}{e_0}\cdot\gcd(4\beta+2,e_0)-1)+(\frac{n}{e_0}\cdot\gcd(2\beta+1,e_0)-1)&\\
        \quad+(\frac{n}{e_1}\cdot\gcd(\beta,e_1)-1)+(\frac{n}{e_0}-1)&e_0\nmid(4\beta+2)
        \end{cases}
    \end{aligned}
    \end{equation}
    
    \textbf{Case: $e_0\mid(4\beta+2)$}. In this case, we have $2\mid e_0$, $\gcd(2\beta+1,e_0)=\frac{e_0}{2}=e_1$, and $\gcd(\beta,e_1)=1$. Using \eqref{eq: Type I first bf}, \eqref{eq: Type I first cf}, \eqref{eq: Type I second bf}, and \eqref{eq: Type I second cf}, for the first equality and Proposition \ref{prop: preliminary cdc} for the second equality, we obtain
    \begin{align*}
        \blf(y)+\cf(y)=&\blf(\Y_j\to...\to\Y_0,F_y)+\cf(\Y_j\to...\to\Y_0,F_y)\\
        &+(2-n-\frac{2n}{e_0}\cdot\gcd(2\beta+1,e_0))+5n-6\\
        =&\alpha(2n-\frac{2n}{e_1}+e_1(n-2))+(\beta-2)(n-2)-\frac{n}{e_1}(1+\gcd(\beta,e_1))\\
        &+(2-n-\frac{2n}{e_0}\cdot\gcd(2\beta+1,e_0))+5n-6\\
        =&\alpha(2n-\frac{2n}{e_1}+e_1(n-2))+\beta(n-2)+n-\frac{2n}{e_1}
    \end{align*}

    \textbf{Case: $e_0\nmid(4\beta+2)$}. Using \eqref{eq: Type I first bf}, \eqref{eq: Type I first cf}, \eqref{eq: Type I second bf}, and \eqref{eq: Type I second cf} for the first equality and Proposition \ref{prop: preliminary cdc} for the second equality, we obtain
    \begin{align*}
        \blf(y)&+\cf(y)\\=&\blf(\Y_j\to...\to\Y_0,F_y)+\cf(\Y_j\to...\to\Y_0,F_y)\\
        &+(2-n-\frac{2n}{e_0}\cdot\gcd(2\beta+1,e_0))+(n-\frac{2n}{e_0}\cdot\gcd(4\beta+2,e_0))\\
        &+3(\frac{n}{e_0}\cdot\gcd(4\beta+2,e_0)-1)+(\frac{n}{e_0}\cdot(2\beta+1,e_0)-1)+(\frac{n}{e_1}\cdot\gcd(\beta,e_1)-1)+(\frac{n}{e_0}-1)\\
        =&\alpha(2n-\frac{2n}{e_1}+e_1(n-2))+(\beta-2)(n-2)-\frac{n}{e_1}(1+\gcd(\beta,e_1))\\
        &+(2-n-\frac{2n}{e_0}\cdot\gcd(2\beta+1,e_0))+(n-\frac{2n}{e_0}\cdot\gcd(4\beta+2,e_0))\\
        &+3(\frac{n}{e_0}\cdot\gcd(4\beta+2,e_0)-1)+(\frac{n}{e_0}\cdot(2\beta+1,e_0)-1)+(\frac{n}{e_1}\cdot\gcd(\beta,e_1)-1)+(\frac{n}{e_0}-1)\\
        =&\alpha(2n-\frac{2n}{e_1}+e_1(n-2))+\beta(n-2)-2n\\
        &+\frac{n}{e_0}(\gcd(4\beta+2,e_0)-\gcd(2,e_0)-\gcd(2\beta+1,e_0)+1)\\
        =&\alpha(2n-\frac{2n}{e_1}+e_1(n-2))+\beta(n-2)-2n+\frac{n}{e_0}(\gcd(2,e_0)-1)(\gcd(2\beta+1,e_0)-1)\qedhere
    \end{align*}
\end{proof}

\begin{corollary}\label{cor: first positivity for type I}
    Suppose $y$ is a multiplicity $2$ point of $D$ of Type I. Then $\blf(y)+\cf(y)\geq0$ except in the following cases:
    \begin{table}[h!]
    \centering
    \renewcommand{\arraystretch}{1.4}
    \begin{tabular}{|cc|c|}
        \hline
         \textbf{\textup{Case}} &  & $\blf(y)+\cf(y)$ \\
        \hline
        $j=1$ & $e_0= 3$ & $=-\frac{2n}{3}$ \\
        \hline
        $j=1$ & $6\nmid e_0, e_0\not\in\{2,3\}$ & $=-n-2$ \\
        \hline
        $j=1$ & $6\mid e_0, e_0\neq 6$ & $=-n-2+\frac{2n}{e_0}$ \\
        \hline
        $j=2$ & $10\nmid e_0,e_0\not\in\{2,4,5\}$ & $=-4$ \\
        \hline
        $j=3$ & $n=e_0=5$ & $=-1$ \\
        \hline
    \end{tabular}
    \end{table}
    
\end{corollary}
\begin{proof}
    This follows from Proposition \ref{prop: cdc for type I}, Proposition \ref{prop: proof of positivity for type I} of the Appendix, and the definitions of $\alpha$ and $\beta$.
\end{proof}

\begin{corollary}\label{cor: cedc positivity for type I}
    Suppose $y$ is a multiplicity $2$ point of $D$ of Type I. Then $\cedc(y)\geq0$.
\end{corollary}

\begin{proof}
    Let $\Y_m\to...\to\Y_0$ be the local $n$-resolution of $D$ at $y$. By Remark \ref{rmk: discussion of bf, cf, etc.}, 
    \begin{align*}
        \cedc(y)=m+\blf(y)+\cf(y)+\mcb(y)\geq\blf(y)+\cf(y)
    \end{align*}    
    so we need only consider the cases listed in Corollary \ref{cor: first positivity for type I}. 

    If $j=2, 10\nmid e_0,$ and $e_0\not\in\{2,4,5\}$, then $\blf(y)+\cf(y)=-4$ by Corollary \ref{cor: first positivity for type I} and $m=4$ by Proposition \ref{prop: full local resolution for type I continued}(v), so that $\cedc(y)=4-4+\mcb(y)\geq0$. Similarly, if $j=3$ and $n=e_0=5$, the same results show $\blf(y)+\cf(y)=-1$ and $m=5$, so that $\cedc(y)=5-1+\mcb(y)\geq0$.

    Suppose now $j=1$ and $e_0=3$, so that $e_1=3$. Then $\blf(y)+\cf(y)=-\frac{2n}{3}$ by Corollary \ref{cor: first positivity for type I} and $m=2$ by Proposition \ref{prop: full local resolution for type I}(v). By Proposition \ref{prop: full local resolution for type I}(vi) and (vii), the unique multiplicity $2$ simple normal crossings point of $D_2$ is the common intersection point $y_{3}$ of $\Gamma_2, E_1,$ and $E_2$ and has ramification pair $(e_1,e_0)=(3,3)$, so that $\mcb(y)=n-\gcd(\frac{n}{3},\frac{n}{3})=\frac{2n}{3}$. Then $\cedc(y)=m+\blf(y)+\cf(y)+\mcb(y)\geq 2-\frac{2n}{3}+\frac{2n}{3}\geq0$.

    Suppose lastly $j=1$ and $e_0\not\in\{2,3,6\}$. Then $\blf(y)+\cf(y)\geq-n-2$ by Corollary \ref{cor: first positivity for type I} and $m=3$ by Proposition \ref{prop: full local resolution for type I continued}(v). By Proposition \ref{prop: full local resolution for type I continued}(vi) and (vii), the  multiplicity $2$ simple normal crossings points of $D_3$ are the intersection points $z_1,z_2,$ and $z_3$ of $E_3$ with $\Gamma_3,E_1,$ and $E_2$ respectively, and the ramification pairs of these points are \[(\frac{e_0}{\gcd(6,e_0)},e_0),\quad(\frac{e_0}{\gcd(6,e_0)},\frac{e_0}{\gcd(2,e_0)}),\quad(\frac{e_0}{\gcd(6,e_0)},\frac{e_0}{\gcd(3,e_0)})\]
    It follows that $\mcb(y)=3n-\frac{n}{e_0}(1+\gcd(2,e_0)+\gcd(3,e_0))\geq 3n-\frac{6n}{e_0}\geq\frac{3n}{2}$ since $e_0\geq 4$, so that $\cedc(y)=m+\blf(y)+\cf(y)+\mcb(y)\geq 3-n-2+\frac{3n}{2}=1+\frac{n}{2}\geq0$.   
\end{proof}

\begin{prop}\label{prop: full positivity for type I}
    Suppose $y$ is a multiplicity $2$ point of $D$ of Type I. Then $\cdc(y)\geq0$ except in the following cases:
    \begin{table}[h!]
    \centering
    \renewcommand{\arraystretch}{1.4}
    \begin{tabular}{|cc|c|}
        \hline
         \textbf{\textup{Case}} &  & $\cdc(y)$ \\
        \hline
        $j=1$ & $e_0= 3$ & $=-\frac{n}{3}$ \\
        \hline
        $j=1$ & $e_0=4$ & $=-2$ \\
        \hline
        $j=1$ & $e_0=5$ & $=-\frac{n}{5}-2$ \\
        \hline
        $j=2$ & $n=e_0=3$ & $=-2$ \\
        \hline
    \end{tabular}
    \end{table} 
\end{prop}

It is useful for the proof of Proposition \ref{prop: full positivity for type I} to introduce the following notation:

\begin{notation}\label{notation: crossing of type (a,b)}
    Suppose $y$ is a multiplicity $2$ point of $D$ of Type I. Let $\Y_m\to...\to\Y_0$ be the local $n$-resolution of $D$ at $y$, and let $a,b\in\N$. A \textit{crossing point of Type $(a,b)$ with respect to $y$} is a multiplicity $2$ simple normal crossings point $z$ of $D_m$ such that the orders of $f$ along the two irreducible components of $D_m$ containing $z$ are congruent to $a\nu_0$ and $b\nu_0$ mod $n$ in some order.
\end{notation}

\begin{remark}\label{rmk: contribution of (a,b) crossing}
    If $y$ is a multiplicity $2$ point of $D$ of Type I and $z$ is a crossing point of Type $(a,b)$ with respect to $y$, then by Remark \ref{rmk: Hirzebruch-Jung length} and Proposition \ref{prop: crossing-bonus in terms of nu0} of the Appendix, the Hirzebruch-Jung length $\HJL(z)$ is given by
    \[\L(\frac{e_0}{\gcd(e_0,\lcm(a,b))},-(\frac{a}{\gcd(e_0,a)})^{-1}\frac{b}{\gcd(e_0,b)})\]
    and the contribution of $z$ to $\cb(y)$ is given by 
    \[n-\frac{n}{e_0}\cdot\gcd(e_0,a,b)\cdot(\L(\frac{e_0}{\gcd(e_0,\lcm(a,b))},-(\frac{a}{\gcd(e_0,a)})^{-1}\frac{b}{\gcd(e_0,b)})+1)\]
\end{remark}

\begin{proof}[Proof of Proposition \ref{prop: full positivity for type I}]
    By Remark \ref{rmk: discussion of bf, cf, etc.}, we have $\cdc(y)\geq\blf(y)+\cf(y)$, so we need only consider the cases listed in Corollary \ref{cor: first positivity for type I}.
    
    \textbf{Case:} $j=1$ and $e_0=3$. Then $e_1=3,\beta=1$, and $\blf(y)+\cf(y)=-\frac{2n}{3}$ by Corollary \ref{cor: first positivity for type I}. By Proposition \ref{prop: full local resolution for type I}(v), (vi), and (vii), the local $n$-resolution of $D$ at $y$ terminates at the stage $\Y_2\to\Y_1\to\Y_0$; the unique multiplicity $2$ simple normal crossings point of $D_2$ is the common intersection point $y_3$ of $\Gamma_2, E_1,$ and $E_2$; and the order 
    pair of $y_2$ is $(\nu_1,\nu_0)=(2\nu_0,\nu_0)$
    , so that $y_2$ is a crossing point of Type $(1,2)$ with respect to $y$. By Remark \ref{rmk: contribution of (a,b) crossing}, with $\L(3,-2)$ computed directly from Definition \ref{dfn: Hirzebruch-Jung length of a pair}, the contribution of $y_2$ to $\cb(y)$ is
    \[n-\frac{n}{3}\cdot(\L(3,-2)+1)=n-\frac{n}{3}(\frac{3-1}{2}+1)=\frac{n}{3}\]
    We obtain $\cdc(y)=\blf(y)+\cf(y)+\cb(y)\geq-\frac{2n}{3}+\frac{n}{3}=-\frac{n}{3}$.

    \textbf{Case:} $j=1$ and $e_0\not\in\{2,3,6\}$. Then $\beta=1$ and $\blf(y)+\cf(y)\geq-n-2$ by Corollary \ref{cor: first positivity for type I}. By Proposition \ref{prop: full local resolution for type I continued}(v), (vi), and (vii),  the local $n$-resolution of $D$ at $y$ terminates at the stage $\Y_3\to...\to\Y_0$; the  multiplicity $2$ simple normal crossings points of $D_3$ are the intersection points $z_1,z_2,z_3$ of $E_3$ with $\Gamma_3,E_1,$ and $E_2$ respectively; and the order pairs of $z_1,z_2,z_3$ are respectively
    \[(6\nu_0,\nu_0),\quad(6\nu_0,2\nu_0),\quad(6\nu_0,3\nu_0)\]
    The points $z_1,z_2,z_3$ are therefore crossing points of Types $(1,6),(2,6),$ and $(3,6)$ respectively, with respect to $y$. By Remark \ref{rmk: contribution of (a,b) crossing}, their contributions  to $\cb(y)$ are respectively

    \[n-\frac{n}{e_0}\cdot(\L(\frac{e_0}{\gcd(e_0,6)},-\frac{6}{\gcd(e_0,6)})+1)\]
    \[n-\frac{n}{e_0}\cdot\gcd(e_0,2)\cdot(\L(\frac{e_0}{\gcd(e_0,6)},-(\frac{2}{(\gcd(e_0,2)})^{-1}\frac{6}{\gcd(e_0,6)})+1)\]
    \[n-\frac{n}{e_0}\cdot\gcd(e_0,3)\cdot(\L(\frac{e_0}{\gcd(e_0,6)},-(\frac{3}{(\gcd(e_0,3)})^{-1}\frac{6}{\gcd(e_0,6)})+1)\]

    In Proposition \ref{prop: cb computations} of the Appendix, we compute these contributions and their sum $\cb(y)$ according to the residue of $e_0$ mod $6$. We obtain
    \[\cb(y)=\begin{cases}
        2n&e_0\equiv0\textup{ mod }6\\
        \frac{2n}{e_0}(e_0-1)&e_0\equiv1,2,3\textup{ mod }6\\
        \frac{2n}{e_0}(e_0-2)&e_0\equiv4\textup{ mod }6\\
        \frac{2n}{e_0}(e_0-3)&e_0\equiv5\textup{ mod }6\\
    \end{cases}\]



    Suppose now $e_0\geq 7$, so that $n\geq e_0\geq 7>\frac{14}{3}$ and $\frac{3n}{7}\geq 2$. If $e_0\not\equiv 5\tx{ mod }6$, then $\cb(y)\geq \frac{2n}{e_0}(e_0-2)\geq \frac{10n}{7}\geq n+2$. If instead $e_0\equiv 5$ mod $6$, then in fact $n\geq e_0\geq 11>\frac{22}{5}$, so that $\frac{5n}{11}\geq 2$ and $\tx{cb}(y)=\frac{2n}{e_0}(e_0-3)\geq\frac{16n}{11}\geq n+2$. In either case, we obtain $\cdc(y)=\blf(y)+\cf(y)+\cb(y)\geq-n-2+n+2=0$.

    Suppose instead $e_0<7$. Since $e_0\not\in\{2,3,6\}$, we have $e_0=4$ or $e_0=5$, hence $\blf(y)+\cf(y)=-n-2$ by Corollary \ref{cor: first positivity for type I}. If $e_0=4$, we have $\cb(y)=n$, so that $\cdc(y)=-n-2+n=-2$, and if $e_0=5$, we have $\cb(y)=\frac{4n}{5}$, so that $\cdc(y)=-n-2+\frac{4n}{5}=-\frac{n}{5}-2$. 

    \textbf{Case:} $j=2$, $10\nmid e_0$, and $e_0\not\in\{2,4,5\}$. Then $\beta=2$ and $\blf(y)+\cf(y)=-4$ by Corollary \ref{cor: first positivity for type I}. By Proposition \ref{prop: full local resolution for type I continued}(v), (vi), and (vii),  the local $n$-resolution of $D$ at $y$ terminates at the stage $\Y_4\to...\to\Y_0$; the multiplicity $2$ simple normal crossings points of $D_4$ consist of the intersection point $z_{0,2}$ of $E_2$ and $E_1$, as well as the intersection points $z_1,z_2,z_3$ of $E_4$ with $\Gamma_4,E_2,$ and $E_3$ respectively; and the order pairs of $z_{0,2},z_1,z_2,z_3$ are respectively
    \[(4\nu_0,2\nu_0),\quad (10\nu_0,\nu_0),\quad (10\nu_0,4\nu_0),\quad (10\nu_0,5\nu_0)\]
    The points $z_{0,2},z_1,z_2,z_3$ are thus crossing points of Types $(2,4),(1,10),(4,10)$ and $(5,10)$ respectively, with respect to $y$. By Remark \ref{rmk: contribution of (a,b) crossing}, their contributions to $\cb(y)$ are respectively
    \[n-\frac{n}{e_0}\cdot\gcd(e_0,2)\cdot(\L(\frac{e_0}{\gcd(e_0,4)},-(\frac{2}{\gcd(e_0,2)})^{-1}\frac{4}{\gcd(e_0,4)})+1)\]

    \[n-\frac{n}{e_0}\cdot(\L(\frac{e_0}{\gcd(e_0,10)},-\frac{10}{\gcd(e_0,10)})+1)\]
    
    \[n-\frac{n}{e_0}\cdot\gcd(e_0,2)\cdot(\L(\frac{e_0}{\gcd(e_0,20)},-(\frac{4}{\gcd(e_0,4)})^{-1}\frac{10}{\gcd(e_0,10)})+1)\]

    \[n-\frac{n}{e_0}\cdot\gcd(e_0,5)\cdot(\L(\frac{e_0}{\gcd(e_0,10)},-(\frac{5}{\gcd(e_0,5)})^{-1}\frac{10}{\gcd(e_0,10)})+1)\]

    By assumption, we have $\gcd(e_0,10)\neq 10$.
    
    If $\gcd(e_0,10)=5$, then $n\geq e_0\geq 5$, and by Remark \ref{rmk: L bound}, we have $L(\frac{e_0}{5},\rho)\leq\frac{e_0}{5}-1$ for all $\rho\in(\Z/\frac{e_0}{5}\Z)^\times$, hence the contribution of $z_1$ 
    to $\cb(y)$ is at least 
    \[n-\frac{n}{e_0}(\frac{e_0}{5}-1+1)=\frac{4n}{5}\geq 4\]
    so that $\cb(y)\geq 4$.
 
    If $\gcd(e_0,10)=2$, then $n\geq e_0\geq 6$ by the assumption $e_0\not\in\{2,4\}$, and by Remark \ref{rmk: L bound}, the contributions of $z_1$ and $z_3$ to $\cb(y)$ are both at least 
    \[n-\frac{n}{e_0}(\frac{e_0}{2}-1+1)=\frac{n}{2}\] so that $\cb(y)\geq n>4$.
        
    If $\gcd(e_0,10)=1$, then $e_0\equiv 1$ mod $2$, and using Proposition \ref{prop: explicit hirzebruch-jung lengths} of the Appendix to compute $\L(e_0,-2)$, we see the contributions of $z_{0,2}$ and $z_3$ to $\cb(y)$ are both equal to
    \[n-\frac{n}{e_0}(\L(e_0,-2)+1)=n-\frac{n}{e_0}(\frac{e_0-1}{2}+1)=\frac{n}{e_0}(\frac{e_0-1}{2})\]
    so that $\cb(y)\geq \frac{n}{e_0}(e_0-1)$. If $e_0\neq 3$, the assumption $\gcd(e_0,10)=1$ implies $n\geq e_0\geq 7$, so that $\cb(y)\geq\frac{6n}{7}\geq 6$. If $e_0=3$ and $n\neq 3$, we have $n\geq 6$ since $e_0\mid n$, so that $\cb(y)\geq\frac{2n}{3}\geq 4$. Finally, if $n=e_0=3$, the contributions of $z_1$ and $z_2$ to $\cb(y)$ are both zero, so that $\cb(y)=\frac{n}{e_0}(e_0-1)=2$.  
    
    In summary, we have $\cdc(y)=\blf(y)+\cf(y)+\cb(y)\geq -4+4=0$ unless $n=e_0=3$, in which case $\cdc(y)=-4+2=-2$.
    
    \textbf{Case:} $j=3$ and $n=e_0=5$. Then $e_1=5$, $\beta=3$, and $\blf(y)+\cf(y)=-1$ by Corollary \ref{cor: first positivity for type I}.  By Proposition \ref{prop: full local resolution for type I continued}(v), (vi), and (vii), the local $n$-resolution of $D$ at $y$ terminates at the stage $\Y_5\to...\to\Y_0$; the intersection point $z_{0,2}$ of $E_2$ and $E_1$ is a multiplicity $2$ simple normal crossings point of $D_5$; and the order pair of $z_{0,2}$ is $(4\nu_0,2\nu_0)$, so that $z_{0,2}$ is a crossing point of Type (2,4) with respect to $y$. By Remark \ref{rmk: contribution of (a,b) crossing}, using Proposition \ref{prop: explicit hirzebruch-jung lengths} of the Appendix to compute $\L(5,-2)$, the contribution of $z_{0,2}$ to $\cb(y)$ is
    \[5-(\L(5,-2)+1)=5-(\frac{5-1}{2}+1)=2\]
    We obtain $\cdc(y)=\blf(y)+\cf(y)+\cb(y)\geq-1+2=1$
\end{proof}

\section{Multiplicity $2$ Points with Negative $\cdc$}\label{section 5}

Recall that a \textit{rational double point} is the spectrum of a normal Noetherian local domain $R$ of dimension $2$ and Hilbert-Samuel multiplicity $2$ for which there exists a proper birational morphism $S\to\Spec R$ from a regular scheme $S$ with $H^1(S,\O_S)=0$ (cf.\,\cite[Definition 1.1]{Lip69}).
In this section, we adapt results of \cite[Section 24]{Lip69} to show the following proposition:

\begin{prop}\label{prop: rational double points}
    Let $\Y$ be a regular model of $\P^1_K$, let $\X\to\Y$ be the normalization of $\Y$ in $K(X)$, and let $D$ be the branch divisor of $\X\to\Y$. Let $y$ be a multiplicity $2$ point of $D$, and let $x\in\X$ lie over $y$. Keep Notation \ref{notation: initial parameters}, and let $j$ be as in Proposition \ref{prop: local resolution desingularizes components}.
    
    Suppose one of the following cases holds:
    \begin{enumerate}[\upshape (i)]
            \item $y$ is of Type II, $j=1$, and $e_0=e_0'=e_1=3$ 
            \item $y$ is of Type I, $j=1$, and $e_0=3$
            \item $y$ is of Type I, $j=1$, and $e_0=4$
            \item $y$ is of Type I, $j=1$, and $e_0=5$
            \item $y$ is of Type I, $j=2$, and $e_0=3$
    \end{enumerate}
    
    Then $\Spec\widehat{\O}_{\X,x}$ is a rational double point, and in cases \textup{(i)}, \textup{(ii)}, \textup{(iii)}, \textup{(iv)}, and \textup{(v)} respectively, the exceptional divisor of its minimal resolution is a tree of $6,4,6,8,$ or $8$ components, all isomorphic to $\P^1_k$, whose dual graph is an $E_6,D_4,E_6,E_8,$ or $E_8$ Dynkin diagram.
        
\end{prop}

\begin{remark}\label{rmk: cdc negative}
   Let $D$ be as in Proposition \ref{prop: rational double points}, and let $y$ be a multiplicity $2$ point of $D$ with $\cdc(y)<0$. Then by Propositions \ref{prop: full positivity for type II} and \ref{prop: full positivity for type I}, some case of Proposition \ref{prop: rational double points} is satisfied.
\end{remark}

\begin{corollary}\label{cor: comps upstairs}
    Let $\X$ and $y$ be as in Proposition \ref{prop: rational double points}, and let $\X'\to\X$ be the minimal resolution of $\X$. In cases \textup{(i)}, \textup{(ii)}, \textup{(iii)}, \textup{(iv)}, and \textup{(v)} respectively, the preimage of $y$ on $\X'$ is the union of $2n,$ $\frac{4n}{3},$ $\frac{3n}{2}$, $\frac{8n}{5}$, or $\frac{8n}{3}$ irreducible components of the special fiber $\X'_s$.
\end{corollary}
\begin{proof}
    By the proof of \cite[Lemma 3.5]{ObusWewers}, for each point $x\in\X$ lying over $y$, the description in Proposition \ref{prop: rational double points} of the exceptional divisor of the minimal resolution of $\Spec\widehat{\O}_{\X,x}$ also applies to the preimage of $x$ under $\X'\to\X$. The result then follows from Lemma \ref{lemma: etale}(ii).
\end{proof}

\subsection{Preliminaries}\hfill

The necessary results of \cite[Section 24]{Lip69} are summarized as follows:

\begin{prop}\cite[Section 24]{Lip69}\label{prop: Lipman results}
    Let $R$ be a Noetherian local ring with maximal ideal $\m$ and residue field $k$ such that the associated graded ring $\gr_\m R\colonequals \oplus_{n\geq0}(\m^n/\m^{n+1})$ is isomorphic as a $k$-algebra to $k[X,Y,Z]/(Z^2)$, where $X,Y,Z$ are indeterminates. For $r\in R$, let $\c{r}$ denote the class of $r$ \textup{mod} $\m$.
    
    Consider the following cases:
    \begin{enumerate}[\upshape (i)]
        \item
        The maximal ideal $\m$ is generated by three elements $x,y,z$ such that
        \[z^2+G(x,y)\in z\m^2\]
        for some homogeneous polynomial $G\in R[X,Y]$ of degree $3$ whose reduction in $k[X,Y]$ is a product of pairwise non-associate linear forms.
        \item 
        The maximal ideal $\m$ is generated by three elements $x,y,z$ such that
        \[\eta z^2+\alpha y^3+\beta zx^2+\gamma x^4\in(x^3y,x^2y^2,xyz,y^2z)\] for some $\eta,\alpha,\beta,\gamma\in R$ such that $\c{\eta}=1$, $\c{\alpha}\neq0$, and the polynomial $Z^2+\c{\beta}ZX+\c{\gamma}X^2\in k[X,Z]$ is a product of non-associate linear forms. 
        \item 
        The maximal ideal $\m$ is generated by three elements $x,y,z$ such that there are $\eta,\alpha,\beta,\gamma,\delta,$ $d_1,d_2,d_3\in R$ with $\c{\eta}=1$ and $\c{\alpha}\neq 0$ such that
        \begin{enumerate}[\upshape (a)]
            \item $\eta z^2+\alpha y^3+\beta x^2z+\gamma x^4\in (y^2z)$
            \item  $Z^2+\c{\beta}ZX+\c{\gamma}X^2=(Z+\c{\delta}X)^2$ in $k[X,Z]$
            \item $(z+\delta x^2)^2-(\eta z^2+\beta x^2z+\gamma x^4)+d_1x^3z+d_2xz^2+d_3x^5
            \in(z^3,z^2y)$
            \item $-\delta d_1+\delta^2d_2+d_3\in R^\times$
        \end{enumerate}
    \end{enumerate}
    In each case, the scheme $\Spec R$ is a rational double point, and in cases \textup{(i)}, \textup{(ii)}, and \textup{(iii)} respectively, the exceptional divisor of its minimal resolution is a tree of $4,6,$ or $8$ components, all isomorphic to $\P^1_k$, whose dual graph is a $D_4,$ $E_6$, or $E_8$ Dynkin diagram.
\end{prop}
\begin{proof}
    Case (i) follows from Case III c.\;of \cite[Section 24, pg.\,265]{Lip69} and the remark beginning with ``\textit{Conversely}..." immediately following Case III c.

    Case (ii) follows from Case V b.\;of \cite[Section 24, pg.\,268]{Lip69} and the remark beginning with ``\textit{Conversely}..." in the first paragraph of Case V of  \cite[Section 24, pg.\,268]{Lip69}.

    Case (iii) follows from Case V d.\;of \cite[Section 24, pg.\,268]{Lip69} and the remark beginning with ``\textit{Conversely}..." in the first paragraph of Case V of \cite[Section 24, pg.\,268]{Lip69}, where we explicitly carry out the coordinate change from Equation 5 of \cite[Section 24, pg.\,267]{Lip69} to Equation 5' of \cite[Section 24, pg.\,268]{Lip69} under two simplifying assumptions, as follows:

    First, in Equation 5 of \cite[Section 24, pg.\,267]{Lip69}, we make the simplifying assumption that the element $\eta z^2+\alpha y^3+\beta x^2z+\gamma x^4\in R$ lies in the ideal $(y^2z)$ rather than the larger ideal $(x^3y,x^2y^2,xyz,y^2z)$. Then we assume the form $P(X,Z)\in k[X,Z]$ is a square and choose $\delta\in R$ as Lipman does. Next we make the simplifying assumption that the congruence $\eta z^2+\beta x^2z+\gamma x^4\equiv(z+\delta x^2)^2$ holds modulo the ideal $(x^3z,xz^2,x^5,z^3,z^2y)$ rather than the larger ideal $(z,x^2)^2\m$, so that there are $d_1,d_2,d_3\in R$ with
    \[(z+\delta x^2)^2-(\eta z^2+\beta x^2z+\gamma x^4)+d_1x^3z+d_2xz^2+d_3x^5\in(z^3,z^2y)\]
    One then checks that the coefficients $\rho$ and $\sigma$ in Equation 5' of \cite[Section 24, pg.\,268]{Lip69} satisfy $\rho\equiv 0$ and $\sigma\equiv-\delta d_1+\delta^2d_2+d_3$  mod $\m$, so that the conditions of Case V d.\;hold if and only if $-\delta d_1+\delta^2d_2+d_3\in R^\times$.
\end{proof}

The next two lemmas are used in Subsection \ref{subsection 5.2} to check the criteria of Proposition \ref{prop: Lipman results}.

\begin{lemma}\label{lemma: associated graded ring}
    Let $R$ be a regular Noetherian local ring of dimension $2$ with maximal ideal $\m=(s,v)$ and residue field $k$. Let $t\in R$ such that $t\equiv s^2\textup{ mod }\m^3$, and let $e\geq3$. Let $S=R[[w]]/(w^e-t)$, and let $\n$ be the maximal ideal of $S$. Then $\gr_\n S$ is isomorphic as a $k$-algebra to $k[X,Y,Z]/(Z^2)$, where $X,Y,Z$ are indeterminates.
\end{lemma}
\begin{proof}
    The ring $R[[w]]$ is a regular Noetherian local ring of dimension $3$ with maximal ideal $\n'\colonequals(s,v,w)$, hence the associated graded ring $\gr_{\n'}R[[w]]$ is a polynomial algebra over $k$ generated by the classes $\c{s},\c{v},\c{w}$ of $s,v,w$ in $\n'/(\n')^2$ (see, e.g., \cite[Tag 00NO]{stacks-project}). By \cite[Exercises 5.2 and 5.3]{Eis}, there is an isomorphism of $k$-algebras $\gr_\n S\cong\gr_{\n'}R[[w]]/(\tx{in}(w^e-t))$, where $\tx{in}(w^e-t)\in\gr_{\n'}R[[w]]$ is the initial form of $w^e-t$. Since $e\geq 3$ and $t\equiv s^2$ mod $\m^3$, we see $\tx{in}(w^e-t)=\c{s}^2\in(\n')^2/(\n')^3$, hence $\gr_\n S\cong k[\c{s},\c{v},\c{w}]/(\c{s}^2)\cong k[X,Y,Z]/(Z^2)$.
\end{proof}

\begin{lemma}\label{lemma: factorization in completion}
Let $R$ be a Noetherian local ring with maximal ideal $\m=(r,s)$. If $t\in R$ satisfies $t\equiv rs\textup{ mod }\m^3$, then there are $r_\infty,s_\infty\in\widehat{R}$ with $(r_\infty,s_\infty)=\m\widehat{R}$ and $t=r_\infty s_\infty$.
\end{lemma}
\begin{proof}
    We inductively define sequences $(r_i)_{i\geq1}$ and $(s_i)_{i\geq 1}$ in $R$ such that $r_i,s_i\in\m^i$ and  
    $t-(r_1+...+r_i)(s_1+...+s_i)\in \m^{i+2}$ for all $i\geq 1$.
    Set $r_1\colonequals r,s_1\colonequals s$, so that $r_1,s_1\in\m$ and $t-r_1s_1\in\m^3$. Now let $i\geq 1$, and suppose $r_j,s_j\in\m^j$ are defined for $j\leq i$ and satisfy
    $t-(r_1+...+r_i)(s_1+...+s_i)\in \m^{i+2}$. Since $\m=(r_1,s_1)$, there are $a_1,...,a_{i+2}\in R$ with
    \[t-(r_1+...+r_i)(s_1+...+s_i)=\sum_{l=0}^{i+2} a_rr_1^ls_1^{i+2-l}\]
    Setting $r_{i+1}\colonequals\sum_{l=0}^{i+1}r_1^ls_1^{i+1-l}\in\m^{i+1}$ and $s_{i+1}\colonequals a_{i+2}r_1^{i+1}\in\m^{i+1}$, we see \[t-(r_1+...+r_i)(s_1+...+s_i)=r_{i+1}s_1+s_{i+1}r_1\]
    so that \[t-(r_1+...+r_{i+1})(s_1+...+s_{i+1})=-r_{i+1}(s_2+...+s_{i+1})-s_{i+1}(r_2+...+r_{i+1})\in\m^{i+3}\]
    This concludes the definition of the sequences $(r_i)_{i\geq1}$ and $(s_i)_{i\geq 1}$. The sums $r_\infty\colonequals\sum_{i=1}^\infty r_i$ and $s_\infty\colonequals\sum_{i=1}^\infty s_i$
    are then well-defined elements of $\m\widehat{R}$ satisfying $t=r_\infty s_\infty$. Moreover, since $r_\infty\equiv r$ and $s_\infty\equiv s$ mod $(\m\widehat{R})^2$, we have $(r_\infty,s_\infty)=\m\widehat{R}$ by Nakayama's lemma.
\end{proof}

\subsection{Applications of Lipman's Criteria}\label{subsection 5.2}\hfill

Let $\Y$ be a regular integral proper flat relative curve over $\O_K$, let $y\in\Y$ be a closed point, and let $R=\O_{\Y,y}$. Let $\m$ be the maximal ideal of $R$, so that the inclusion $\O_K\subseteq R$ induces an isomorphism $k\cong R/\m$, and let $u,v\in R$ generate $\m$. For $r\in R$, let $\c{r}$ denote the class of $r$ mod $\m$. We note that since $R$ is regular, every prime ideal of $R$ of height $1$ is principal.

\begin{prop}\label{prop: rational double points - Type II}
    Assume $\textup{char}(k)\neq 3$. Suppose $\Gamma$ and $\Gamma'$ are distinct regular irreducible divisors on $\Y$ meeting with intersection multiplicity $2$ at $y$, and let $t_1,t_2\in R$ respectively generate the height $1$ prime ideals corresponding to the restrictions of $\Gamma$ and $\Gamma'$ to $\Spec R$. Then the ring $S=\widehat{R}[[w]]/(w^3-t_1t_2)$ satisfies the conditions of Proposition \ref{prop: Lipman results}\textup{(ii)}.
\end{prop}
\begin{proof}
    We first note that since $k$ is algebraically closed with $\tx{char}(k)\neq 3$, all units in $\widehat{R}$ are cubes, so scaling $t_1$ or $t_2$ by an element of $R^\times$ preserves the isomorphism class of $S$.
    
    Next by \cite[Example 9.2.21]{Liu_book}, since $\Gamma$ and $\Gamma'$ are regular at $y$, there are non-zero linear forms $P_1,P_2\in R[U,V]$ with coefficients in $R^\times\cup\{0\}$ such $t_j\equiv P_j(u,v)\tx{ mod }\m^2$ for $j\in\{1,2\}$. Let $i=\tx{length}_R(R/(t_1,t_2)R)$ be the intersection multiplicity of $\Gamma$ and $\Gamma'$ at $y$. Since $i>1$, it follows that $(t_1,t_2)\neq\m$, so that by Nakayama's Lemma, the images of $P_1(u,v)$ and $P_2(u,v)$ in $\m/\m^2$ are $k$-linearly dependent. Scaling $t_2$ by an element of $\O_K^\times$ if necessary, we can assume $P_1(u,v)=P_2(u,v)\equalscolon s$, and switching the labels of $u$ and $v$ if necessary, we can assume $(s,v)=\m$. Then there are $c_1,c_2\in\m$ and $d_1,d_2\in R$ such that $t_j=(1+c_j)s+d_jv^2$ for $j\in\{1,2\}$. Scaling $t_1,d_1,t_2,d_2$ by elements of $R^\times$, we can further assume $t_j=s+d_jv^2$ for $j\in\{1,2\}$, so that $(t_1,t_2)=(t_1,(d_1-d_2)v^2)$. Since $(t_1,v)=(s,v)=\m$, the hypothesis $i=2$ implies $d_1\not\equiv d_2$ mod $\m$.

    Let $\n=(s,v,w)$ be the maximal ideal of $S$. By Lemma \ref{lemma: associated graded ring}, since $t_1t_2=(s+d_1v^2)(s+d_2v^2)\equiv s^2\tx{ mod }\m^3$, there is an isomorphism of $k$-algebras $\gr_\n S\cong k[X,Y,Z]/(Z^2)$. The generators $s,v,w$ of $\n$ moreover satisfy \[\eta s^2+\alpha w^3+\beta sv^2+\gamma v^4=t_1t_2-w^3=0\]
    for $\eta=1,\alpha=-1,\beta=d_1+d_2,$ and $\gamma=d_1d_2$, where since $d_1\not\equiv d_2$ mod $\m$, the polynomial $Z^2+\c{\beta} ZX+\c{\gamma} X^2=(Z+\c{d_1}X)(Z+\c{d}_2X)\in k[X,Z]$ is a product of non-associate linear forms. Taking $x=v,y=w,z=s,$ and $\eta,\alpha,\beta,\gamma$ as above, we see $S$ satisfies the conditions of Proposition \ref{prop: Lipman results}(ii).
\end{proof}

\begin{lemma}\label{lemma: analytically irreducible}
    Let $\Gamma$ be an irreducible horizontal divisor on $\Y$ containing $y$. If $t\in R$ generates the height $1$ prime ideal corresponding to the restriction of $\Gamma$ to $\Spec R$, then $t$ is irreducible in $\widehat{R}$.
\end{lemma}
\begin{proof}
    By Lemma \ref{lemma: unique intersection points new}, we see firstly that $\Gamma$ is local with closed point $y$, hence $\O_\Gamma(\Gamma)\cong R/tR$, and secondly that $\widehat{\O_\Gamma(\Gamma)}\cong \widehat{R/tR}\cong\widehat{R}/t\widehat{R}$ is a domain, hence $t$ is irreducible in $\widehat{R}$.
\end{proof}

\begin{prop}\label{prop: rational double points - Type I}
    Let $e\in\{3,4,5\}$, and assume $\textup{char}(k)\nmid e$. Suppose $\Gamma$ is an irreducible horizontal divisor on $\Y$ in which $y$ is a point of multiplicity $2$, let $t\in R$ generate the height $1$ prime ideal corresponding to the restriction of $\Gamma$ to $\Spec R$, and let $S=\widehat{R}[[w]]/(w^e-t)$.
    
    Let $\Gamma_1$ be the strict transform of $\Gamma$ under the blow-up $\Y_1\to\Y$ of $\Y$ at $y$, and let $\Gamma_2$ be the strict transform of $\Gamma_1$ under the blow-up of $\Y_1$ at the unique closed point of $\Gamma_1$.
    \begin{enumerate}[\upshape (i)]
        \item If $\Gamma_1$ is regular, the ring $S$ satisfies the conditions of Proposition \ref{prop: Lipman results}\textup{(i)}, \textup{(ii)}, or \textup{(iii)} according as $e=3,4,$ or $5$.
        \item If $e=3$ and $\Gamma_1$ is singular but $\Gamma_2$ is regular, the ring $S$ satisfies the conditions of Proposition \ref{prop: Lipman results}\textup{(iii)}.
    \end{enumerate}
\end{prop}
\begin{proof}
    We first note since $k$ is algebraically closed with $\tx{char}(k)\nmid e$, all units in $\widehat{R}$ are $e$th powers, so scaling $t$ by an element of $R^\times$ preserves the isomorphism class of $S$.

    Next by \cite[Example 9.2.21]{Liu_book}, since $y$ is a multiplicity $2$ point of $\Gamma$, we have $t\equiv P(u,v)\tx{ mod }\m^3$ for some non-zero degree $2$ homogeneous polynomial $P\in R[U,V]$ with coefficients in $R^\times\cup\{0\}$. By Lemma \ref{lemma: analytically irreducible}, the element $t$ is irreducible in $\widehat{R}$, so $P\neq aUV$ for any $a\in R^\times$ by Lemma \ref{lemma: factorization in completion}. Scaling $t$ by an element of $R^\times$ and switching the labels of $u$ and $v$ if necessary, we can then assume $P=U^2+cUV+dV^2$ for some $c,d\in R$. Since $k$ is algebraically closed, the reduction $\c{P}$ of $P$ in $k[U,V]$ factors as $\c{P}=(U+r_1V)(U+r_2V)$ for some $r_1,r_2\in k$. Let $\tilde{r}_1,\tilde{r}_2\in\O_K^\times\cup\{0\}$ be lifts of $r_1,r_2$. 
    Then $t\equiv P(u,v)\equiv (u+\tilde{r}_1v)(u+\tilde{r}_2v)$ mod $\m^3$, so that $(u+\tilde{r}_1v,u+\tilde{r}_2v)=(u+\tilde{r}_1v,(\tilde{r}_1-\tilde{r}_2)v)\neq \m$ by Lemma \ref{lemma: factorization in completion} again. It follows that $\tilde{r}_1\equiv \tilde{r}_2$ mod $\m$, so we can assume $\tilde{r}_1=\tilde{r}_2\colonequals r$. Taking $s=u+rv$, we have $t\equiv s^2$ mod $\m^3$ and $(s,v)=\m$. It follows that $t=s^2+ps+qv^3$ for some $p\in\m^2$ and $q\in R$, and we further have $p=\tilde{p}s+\tilde{q}v^2$ for some $\tilde{p}\in\m$ and $\tilde{q}\in R$, and $\tilde{p}=\tilde{\tilde{p}}s+\tilde{\tilde{q}}v$ for some $\tilde{\tilde{p}},\tilde{\tilde{q}}\in R$.
    
    Let $B\to\Spec R$ be the blow-up at $\m$, and consider the affine chart of $B$ whose coordinate ring is $R[s']$, where $s'=\frac{s}{v}\in\tx{Frac}(R)$. The divisor\footnote{The divisor $\Gamma$ and its restriction to $\Spec R$, denoted $\Gamma_y$, are isomorphic as schemes. By \cite[Tag 0863 (2)]{stacks-project}, the same is true of $\Gamma_1$ and the strict transform $(\Gamma_y)_1$ of $\Gamma_y$ under $B\to\Spec R$. We conflate these schemes, and we similarly conflate $\Gamma_2$ with the strict transform of $(\Gamma_y)_1$ under the blow-up of $B$ at the unique closed point of $(\Gamma_y)_1$.} $\Gamma_1$ is defined in this chart by the element $t'=\frac{t}{v^2}$, for which $t'\equiv qv$ mod $(s',v)^2$. It follows that the point of $B$ corresponding to the ideal $(s',v)$ is the unique closed point of $\Gamma_1$, and $\Gamma_1$ is regular if and only if $q\in R^\times$.

    Let $\n=(s,v,w)$ be the maximal ideal of $S$. By Lemma \ref{lemma: associated graded ring}, since $t\equiv s^2$ mod $\m^3$, we have $\gr_\n S\cong k[X,Y,Z]/(Z^2)$. 
    
    \textbf{Case 1:} Suppose $e=3,4,$ or $5$ and $\Gamma_1$ is regular, so that $q\in R^\times$.
    
    If $e=3$, the generators $s,v,w$ of $\n$ satisfy
    \[s^2+G(v,w)=t-w^3-ps=-ps\] for the degree $3$ homogeneous polynomial $G(X,Y)=qX^3-Y^3\in S[X,Y]$. Since $p\in\m^2$, we have $-ps\in s\n^2$, and since $q\in R^\times$ and $\tx{char}(k)\neq e=3$, the reduction of $G(X,Y)$ in $k[X,Y]$ splits as a product of pairwise non-associate linear forms. Taking $x=v,y=w,z=s$, we see $S$ satisfies the conditions of Proposition \ref{prop: Lipman results}(i).

    If $e=4$, the generators $s,v,w$ of $\n$ satisfy \[\eta s^2+\alpha v^3+\beta sw^2+\gamma w^4=t-w^4-\tilde{q}v^2s=-\tilde{q}v^2s\]
    for $\eta=1+\tilde{p},$ $\alpha=q,$ $\beta=0,$ and $\gamma=-1$. Since $\tilde{p}\in\m$ and $q\in R^\times$, we have $\c{\eta}=1$ and $\c{\alpha}=\c{q}\neq0$, and since $\tx{char}(k)\nmid e=4$, the polynomial $Z^2+\c{\beta} ZX+\c{\gamma} X^2=Z^2-X^2\in k[X,Z]$ is a product of non-associate linear forms. Taking $x=w,y=v,z=s,$ and $\eta,\alpha,\beta,\gamma$ as above, we see $S$ satisfies the conditions of Proposition \ref{prop: Lipman results}(ii).

    Finally, if $e=5$, the ring $S$ satisfies the conditions of Proposition \ref{prop: Lipman results}(iii) with $x=w,$ $y=v,$ $z=s$, $\eta=1+\tilde{p},$ $\alpha=q,$ $\gamma=-w,$ $d_3=-1$, and $\beta=\delta=d_1=d_2=0$, where $\c{\eta}=1$ and $\c{\alpha}\neq0$ as in the case $n=4$ above. 

    \textbf{Case 2:} Suppose $e=3$ and $\Gamma_1$ is singular, so that $q\in\m$, but $\Gamma_2$ is regular.

    In this case, we can write $q=As+Bv$ for some $A,B\in R$. Then
    \begin{align*}
    t'=\frac{t}{v^2}&=\frac{(1+\tilde{p})s^2+\tilde{q}sv^2+Asv^3+Bv^4}{v^2}&\\
    &=(1+\tilde{p})(s')^2+\tilde{q}s'v+As'v^2+Bv^2&\\
    &\equiv(s')^2+\tilde{q}s'v+Bv^2 &\tx{mod }(s',v)^3
    \end{align*}

    Since $\Gamma_1$ is irreducible, a similar argument to the above shows the polynomial $S+\c{\tilde{q}}SV+\c{B}V^2\in k[S,V]$ is equal to $(S+\c{\delta} V)^2$ for some $\delta\in\O_K$. Then $2\delta\equiv\tilde{q}$ and $\delta^2\equiv B$ mod $\m$. Write $2\delta-\tilde{q}=Cs+Dv$, $\delta^2-B=Es+Fv$, and $\tilde{p}=Gs+Hv$ for some $C,D,E,F,G,H\in R$. 

    Let $B'\to B$ be the blow-up at the ideal $(s',v)$, and consider the affine chart of $B$ whose coordinate ring is $R[s'']$ where $s''=\frac{s'}{v}\in\tx{Frac}(R)$. The divisor $\Gamma_2$ is defined in this chart by the element $t''=\frac{t'}{v^2}$. Letting $j=s''+\delta$, we have
    \begin{align*}
    t''&=\frac{(1+\tilde{p})(s')^2+\tilde{q}s'v+As'v^2+Bv^2}{v^2}\\
        &=(1+\tilde{p})(s'')^2+\tilde{q}s''+As''v+B&\\
        &=(1+Gs+Hv)(s'')^2+(2\delta-Cs-Dv)s''+As''v+(\delta^2-Es-Fv)&\\
        &=Gv^4(s'')^3+(1+Hv-Cv^4)(s'')^2+(2\delta+(A-D)v-Ev^4)s''+(\delta^2-Fv)&\\
        &\equiv(1+Hv)(j-\delta)^2+(2\delta+(A-D)v)(j-\delta)+(\delta^2-Fv)&\tx{mod }(j,v)^2\\
        &\equiv -2\delta j+\delta^2 +\delta^2Hv+2\delta j-2\delta^2-\delta(A-D)v+\delta^2-Fv&\tx{mod }(j,v)^2\\
        &=(\delta^2H-\delta(A-D)-F)v
    \end{align*}
    
    We see that the unique closed point of $\Gamma_2$ is $(j,v)$, and the hypothesis that $\Gamma_2$ is regular at this point implies $\delta^2H-\delta(A-D)-F\in R^\times$. Taking $x=v,y=w,z=s$, $\eta=1+\tilde{p},\alpha=-1,\beta=\tilde{q}+Av,\gamma=B$, $d_1=A-D-Ev$, $d_2=H-Cv$, $d_3=-F$, and $\delta$ as above, we see $S$ satisfies the conditions of Proposition \ref{prop: Lipman results}(iii). 
\end{proof}

\subsection{Proof of Proposition \ref{prop: rational double points}}\hfill

Let $\Y$ be a regular model of $\P^1_K$, let $\X\to\Y$ be the normalization of $\Y$ in $K(X)$, let $D$ be the branch divisor of $\X\to\Y$, and let $y$ be a closed point of $D$. Let $e\in\N$ divide $n$.

\begin{lemma}\label{lemma: etale} Let $D_e$ be the sum of the irreducible components of $D$ of ramification index not dividing $e$, and suppose $\Gamma$ is an irreducible component of $D$ of ramification index $e$.
    \begin{enumerate}[\upshape (i)]
        \item The cover $\mathcal{Z}\colonequals\X/(\Z/e\Z)\to\Y$ is \'{e}tale over $\Y\setminus\textup{Supp}(D_e)$.
        \item If $y\in\Gamma\setminus\textup{Supp}(D_e)$, then the preimage of $y$ on $\mathcal{Z}$ consists of exactly $\frac{n}{e}$ closed points, each lying under a unique point of $\X$.
        \item If $\Gamma$ is vertical and intersects $\textup{Supp}(D_e)$ in at most $1$ point, then the preimage of $\Gamma$ on $\X$ consists of exactly $\frac{n}{e}$ irreducible components of $\X_s$.
    \end{enumerate}
\end{lemma}
\begin{proof}
    The cover $\mathcal{Z}\to\Y$ is \'{e}tale over all codimension $1$ points of $\Y\setminus\tx{Supp}(D_e)$ by the definition of $D_e$, hence \'{e}tale over $\Y\setminus\tx{Supp}(D_e)$ by the purity of the branch locus (see, e.g., \cite[Theorem 5.2.13]{Szamuely}). This shows (i).

    Suppose now $y\in\Gamma\setminus\tx{Supp}(D_e)$. By (i), the degree $\frac{n}{e}$ cover $\W\to\Y$ is \'{e}tale over $y$, so the preimage of $y$ on $\W$ consists of $\frac{n}{e}$ closed points. Each such point $w$ is contained in an irreducible component of the preimage of $\Gamma$ on $\W$, which is totally ramified in $\X\to\W$ by the assumption that $\Gamma$ has ramification index $e$, hence $w$ lies under a unique point of $\X$. This shows (ii).
    
    Finally, suppose $\Gamma$ is as in (iii). Then there is a closed point $\gamma\in\Gamma$ such that $\Gamma\setminus\{\gamma\}\subset\Y\setminus\textup{Supp}(D_e)$. Then $\W\to\Y$ is \'{e}tale over $\Gamma\setminus\{\gamma\}$, hence trivial over $\Gamma\setminus\{\gamma\}$ since $\Gamma\setminus\{\gamma\}\cong\A^1_k$ by Lemma \ref{lemma: unique intersection points}(ii). The preimage of $\Gamma\setminus\{\gamma\}$ on $\W$ thus consists of $\frac{n}{e}$ irreducible components isomorphic to $\A^1_k$, whose closures in $\W$ are pairwise distinct irreducible components of $\W_s$ that together comprise the preimage of $\Gamma$ on $\W$. Since $\Gamma$ has ramification index $e$, each irreducible component of the preimage of $\Gamma$ on $\W$ is totally ramified in the cover $\X\to\W$, hence is the image of a unique irreducible component of $\X_s$, showing the result.
\end{proof}

Let $R=\O_{\Y,y}$, let $t\in R$ generate the principal ideal corresponding to the restriction of $D$ to $\Spec R$, and let $S=\widehat{R}[[w]]/(w^e-t)$. Let $x\in\X$ lie over $y$. 

\begin{lemma}\label{lemma: local ring upstairs}
    Suppose there is $a\in\N$ such that $\gcd(e,a)=1$ and $\ord_\Gamma(f)\equiv a\frac{n}{e}$ \textup{mod} $n$ for each irreducible component $\Gamma$ of $D$ containing $y$. If $S$ is normal, then $S\cong\widehat{\O}_{\X,x}$. 
\end{lemma}

\begin{proof}
    Let $h\in K(X)$ be an arbitrary $n$th root of $f$. The ring $\widehat{\O}_{\X,x}$ is the integral closure of $\widehat{R}$ in $\tx{Frac}(\widehat{\O}_{\X,x})=\tx{Frac}(\widehat{R})(h)$. Since units in $\widehat{R}$ are $n$th powers, the assumption that $\ord_\Gamma(f)\equiv a\frac{n}{e}$ \textup{mod} $n$ for each irreducible component $\Gamma$ of $D$ containing $y$ implies $f=g^nt^{a\frac{n}{e}}$ for some $g\in\widehat{R}$, so that $h^e=(ug)^et^a$ for an $n$th root of unity $u\in\O_K^\times$. Since $\gcd(e,a)=1$, the elements $(ug)^et^a$ and $t$ generate the same subgroup of $\tx{Frac}(\widehat{R})^{\times}/\tx{Frac}(\widehat{R})^{\times e}$. By Kummer theory, we see there is $s\in\tx{Frac}(\widehat{\O}_{\X,x})$ with $s^e=t$ and $\tx{Frac}(\widehat{\O}_{\X,x})=\tx{Frac}(\widehat{R})(s)$.
    If $S'=\widehat{R}[s]\subseteq\tx{Frac}(\mathcal{\O}_{\X,x})$ is the subring generated over $\widehat{R}$ by $s$, then $S'$ is integral over $\widehat{R}$ 
     and $\tx{Frac}(S')=\tx{Frac}(\widehat{\mathcal{\O}}_{\X,x})$. Since $t$ is a square-free non-unit in $R$, the same is true of $t$ as an element of $\widehat{R}$ by \cite[Tag 07NZ (3)]{stacks-project}, so that by Eisenstein's criterion, the polynomial $w^e-t\in\widehat{R}[w]$ is the minimal polynomial of $s$ over $\tx{Frac}(\widehat{R})$. It follows that $S'\cong\widehat{R}[w]/(w^e-t)\cong S$. Since the ring $S'$ is normal by hypothesis, so is $S'$, and it follows that $\widehat{\O}_{\X,x}=S'\cong S$.
\end{proof}

\begin{proof}[Proof of Proposition \ref{prop: rational double points}]
   Resume the notation of Proposition \ref{prop: rational double points}. Let $R=\O_{\Y,y}$, let $t\in R$ generate the principal ideal corresponding to the restriction of $D$ to $\Spec R$, and let $S=\widehat{R}[[w]]/(w^{e_0}-t)$. By Propositions \ref{prop: rational double points - Type II} and \ref{prop: rational double points - Type I}, the ring $S$ satisfies the conditions of cases (ii), (i), (ii), (iii), and (iii) of Proposition \ref{prop: Lipman results} in cases (i), (ii), (iii), (iv), and (v) of Proposition \ref{prop: rational double points} respectively. In particular, the ring $S$ is normal.
   
   In cases (ii), (iii), (iv), and (v) of Proposition \ref{prop: rational double points}, the hypotheses of Lemma \ref{lemma: local ring upstairs} are satisfied with $e=e_0=\frac{n}{\gcd(\nu_0,n)}$ and $a=\frac{\nu_0}{\gcd(\nu_0,n)}$, where $\nu_0$ is as in Notation \ref{notation: initial parameters}.
   
   Assume case (i) of Proposition \ref{prop: rational double points}, and let $\nu_0,\nu_0',\nu_1$ be as in Notation \ref{notation: initial parameters}. Since $e_0=e_0'=3$, we have $\nu_0=a\cdot\frac{n}{3}$ and $\nu_0'=b\cdot\frac{n}{3}$ for some $a,b\in\Z$ prime to $3$, and since $e_1=3$, we have $\frac{n}{3}=\gcd(\nu_1,n)=\gcd(\nu_0+\nu_0',n)=\frac{n}{3}\cdot\gcd(a+b,3)$, so that $\gcd(a+b,3)=1$. It follows that $a\equiv b$ mod $3$, hence $\nu_0'\equiv a\cdot\frac{n}{3}$ mod $n$. We see the hypotheses of Lemma \ref{lemma: local ring upstairs} are satisfied with $e=e_0=3$ and $a$ as above.
   
   Then in all cases of Proposition \ref{prop: rational double points}, Lemma \ref{lemma: local ring upstairs} implies $\widehat{\O}_{\X,x}\cong S$, and the result follows from Proposition \ref{prop: Lipman results}. 
\end{proof}

\section{Existence of Contractions}\label{section 6}

Let $\Y\to\P^1_K$ be a regular model of $\P^1_K$, let $\X\to\Y$ be the normalization of $\Y$ in $K(X)$, let $D$ be the branch divisor of $\X\to\Y$, and let $y$ be a multiplicity $2$ point of $D$ with $\cdc(y)<0$. On one hand, let $\X'\to\X$ be the minimal resolution of $\X$. On the other hand, let $\Y_m\to...\to\Y_0\equalscolon \Y$ be the local $n$-resolution of $D$ at $y$, let $\X_m\to\Y_m$ be the normalization of $\Y_m$ in $K(X)$, and $\X'_m\to\X_m$ be the minimal resolution of $\X_m$.

\begin{prop}\label{prop: contractions}
    The canonical morphism $\pi:\X_m'\to\X'$ contracts at least $-\cdc(y)$ irreducible components of the special fiber $(\X_m')_s$ contained in the preimage of $y$ on $\X_m'$.
\end{prop}

Before turning to the proof, we state two preliminary results from \cite{Part1}.

\begin{prop}\label{prop: Part1 results}
    Let $z$ be a multiplicity $2$ simple normal crossings point of the branch divisor of $\X_m\to\Y_m$ with ramification pair $(e,e')$. Then the preimage of $z$ on $\X_m'$ consists of exactly $\frac{n}{\lcm(e,e')}\cdot\HJL(z)$ irreducible components of $(\X_m')_s$.
\end{prop}
\begin{proof}
    Since the singular points of the branch divisor of $\X_m\to\Y_m$ are isolated, this follows from \cite[Proposition 7.6(iii)]{Part1} and \cite[Remark 7.10]{Part1}.
\end{proof}

\begin{lemma}\label{lemma: Part1 results singular points}
    Every singular point of $\X_m$ contained in the preimage of $y$ lies over a multiplicity $2$ simple normal crossings point of $D_m$. 
\end{lemma}
\begin{proof}
    By Lemmas \ref{prop: local resolution is finite} and \ref{lemma: multiplicity in div_0(f)}(i), the divisor $D_m$ has simple normal crossings and is the part of the branch divisor of $\X_m\to\Y_m$ intersecting the preimage of $y$ on $\Y_m$. Since the singular points of the branch divisor of $\X_m\to\Y_m$ are isolated, the result follows as in \cite[Proposition 7.6(i)]{Part1}.
\end{proof}

\begin{proof}[Proof of Proposition \ref{prop: contractions}]
    Since $\pi$ consists of contractions of irreducible components of the special fiber $(\X_m')_s$ and fits into a commutative diagram with the morphisms $\X_m'\to\Y$ and $\X'\to\Y$, it suffices to show that the preimages of $y$ on $\X'$ and $\X_m'$ are unions of $N$ and $M$ components of the special fibers $\X'_s$ and $(\X_m')_s$ respectively, for some $N,M\in\N$ with $M\geq N-\cdc(y)$.

    Keep Notations \ref{notation: initial parameters} and \ref{notation: alpha beta}.

    \textbf{Case:} $y$ is of Type II, $j=1$, and $n=e_0=e_0'=e_1=3$. Then $\cdc(y)=-1$ by Proposition \ref{prop: full positivity for type II}, and the preimage of $y$ on $\X'$ consists of $N=6$ components of $\X_s'$ by Corollary \ref{cor: comps upstairs}. We show the preimage of $y$ on $\X_m'$ consists of $N-\cdc(y)=7$ components of $(\X'_m)_s$.
    
    By Proposition \ref{prop: full local resolution for type II}(v), (vi), and (vii), the local $3$-resolution of $D$ at $y$ consists of two blow-ups $\Y_2\to\Y_1\to\Y_0$, so that $m=2$; the multiplicity $2$ simple normal crossings points of $D_2$ are the intersection points $z_1$, $z_2$, and $z_3$ of $E_2$ with $E_1$, $\Gamma_2$, and $\Gamma'_2$ respectively; the ramification pairs of $z_1,z_2,z_3$ are all equal to $(3,3)$; and the order pairs of $z_1,z_2,z_3$ are respectively 
    \[(\nu_0+\nu_0',2(\nu_0+\nu_0')),\quad (\nu_0,2(\nu_0+\nu_0')),\quad (\nu_0',2(\nu_0+\nu_0'))\]
    Moreover, as in the proof of Proposition \ref{prop: rational double points}, we have $\nu_0\equiv \nu_0'$ mod $3$, so that $2(\nu_0+\nu_0')\equiv 4\nu_0\equiv 4\nu_0'$ mod $3$. It follows from Proposition \ref{prop: crossing-bonus in terms of nu0} of the Appendix that $\HJL(z_1)=\L(3,1)=1$ and $\HJL(z_2)=\HJL(z_3)=\L(3,2)=2$. 
    
    By Proposition \ref{prop: Part1 results}, for $i\in\{1,2,3\}$, the preimage of $z_i$ on $\X_2'$ 
    consists of exactly $\HJL(z_i)$ components of $(\X_2')_s$. By Lemma \ref{lemma: Part1 results singular points}, the preimage of $y$ on $\X_2'$ then consists of the $1+2+2=5$ components of $(\X_2')_s$ arising this way, along with the strict transforms on $\X_2'$ of the preimages of $E_1$ and $E_2$ on $\X_2$, which are also components of $(\X_2')_s$. Since $E_1$ and $E_2$ are totally ramified by Propositions \ref{prop: structure of exceptional divisor}(vi) and \ref{prop: full local resolution for type II}(iii), we see both lie under unique components of $(\X_2)_s$, so that the preimage of $y$ on $\X_2'$ consists of $7$ components of $(\X_2')_s$, as desired.

    \textbf{Case:} $y$ is of Type I, $j=1$, and $e_0=3$. Then $\cdc(y)=-\frac{n}{3}$ by Proposition \ref{prop: full positivity for type I}, and the preimage of $y$ on $\X'$ consists of $N=\frac{4n}{3}$ components of $\X_s'$ by Corollary \ref{cor: comps upstairs}. We show the preimage of $y$ on $\X_m'$ consists of $N-\cdc(y)=\frac{5n}{3}$ components of $(\X'_m)_s$.

    By Proposition \ref{prop: full local resolution for type I}(v), (vi), and (vii), the local $n$-resolution of $D$ at $y$ consists of two blow-ups $\Y_2\to\Y_1\to\Y_0$, so that $m=2$; the unique multiplicity $2$ simple normal crossings point of $D_2$ is the common intersection point $y_3$ of $\Gamma_2$, $E_1,$ and $E_2$; and the ramification and order pairs of $y_3$ are respectively $(3,3)$ and $(\nu_0,2\nu_0)$. It follows from Proposition \ref{prop: crossing-bonus in terms of nu0} of the Appendix that $\HJL(y_3)=\L(3,1)=1$.

    By Proposition \ref{prop: Part1 results}, the preimage of $y_3$ on $\X_2'$ consists of exactly $\frac{n}{3}$ components of $(\X_2')_s$. 
    By Lemma \ref{lemma: Part1 results singular points}, the preimage of $y$ on $\X_2'$ consists of these $\frac{n}{3}$ components of $(\X_2')_s$, along with the strict transforms on $\X_2'$ of the preimages of $E_1$ and $E_2$ on $\X_2$, which are also components of $(\X_2')_s$. By Propositions \ref{prop: structure of exceptional divisor}(vi) and \ref{prop: full local resolution for type I}(iii), the divisors $E_1$ and $E_2$ have ramification indices $e_1=3$ and $e_2=1$ respectively. Since $y_3$ is the only point of intersection of either divisor with a ramified divisor distinct from itself, Lemma \ref{lemma: etale}(iii) implies $E_1$ and $E_2$ lie under $\frac{n}{3}$ and $n$ components of $(\X_2)_s$ respectively. The preimage of $y$ on $\X_2'$ then consists of $\frac{n}{3}+\frac{n}{3}+n=\frac{5n}{3}$ irreducible components of $(\X_2')_s$, as desired.

    \textbf{Case:} $y$ is of Type I, $j=1$, and $e_0=4$. Then $\cdc(y)=-2$ by Proposition \ref{prop: full positivity for type I}, and the preimage of $y$ on $\X'$ consists of $N=\frac{3n}{2}$ components of $\X_s'$ by Corollary \ref{cor: comps upstairs}. We show the preimage of $y$ on $\X_m'$ consists of at least $N-\cdc(y)=\frac{3n}{2}+2$ components of $(\X'_m)_s$.

    By Proposition \ref{prop: full local resolution for type I continued}(v), (vi), and (vii), the local $n$-resolution of $D$ at $y$ consists of three blow-ups $\Y_3\to...\to\Y_0$, so that $m=3$; the multiplicity $2$ simple normal crossings points of $D_3$ are the intersection points $z_1$, $z_2$, and $z_3$ of $E_3$ with $\Gamma_3$, $E_1$, and $E_2$ respectively; the ramification pairs of $z_1,z_2,z_3$ are respectively $(2,4)$, $(2,2)$, and $(2,4)$; and the order pairs of $z_1,z_2,z_3$ are respectively $(6\nu_0,\nu_0),(6\nu_0,2\nu_0)$, and $(6\nu_0,3\nu_0)$. It follows from Proposition \ref{prop: crossing-bonus in terms of nu0} of the Appendix that $\HJL(z_i)=\L(2,1)=1$ for $i\in\{1,2,3\}$.

    By Proposition \ref{prop: Part1 results}, the preimages of $z_1,z_2,z_3$ on $\X_3'$ consist of exactly $\frac{n}{4},\frac{n}{2},$ and $\frac{n}{4}$ components of $(\X_3')_s$ respectively. 
    By Lemma \ref{lemma: Part1 results singular points}, the preimage of $y$ on $\X_3'$ then consists of the $\frac{n}{4}+\frac{n}{2}+\frac{n}{4}=n$ components of $(\X_3')_s$ arising in this way, along with the strict transforms on $\X_3'$ of the preimages of $E_1$, $E_2$, and $E_3$ on $\X_3$, which are also components of $(\X_3')_s$. Since $E_1$ has ramification index $2$ by Proposition \ref{prop: structure of exceptional divisor}(vi) and meets other branch components only in the point $z_2$, Lemma \ref{lemma: etale}(iii) implies $E_1$ lies under exactly $\frac{n}{2}$ components of $(\X_3)_s$. Since $E_2$ and $E_3$ both lie under at least $1$ component of $\X_3$, it follows that the preimage of $y$ on $\X_3'$ consists of at least $n+\frac{n}{2}+2=\frac{3n}{2}+2$ components of $(\X_3')_s$, as desired.
    


    \textbf{Case:} $y$ is of Type I, $j=1$, and $e_0=5$. Then $\cdc(y)=-\frac{n}{5}-2$ by Proposition \ref{prop: full positivity for type I}, and the preimage of $y$ on $\X'$ consists of $N=\frac{8n}{5}$ components of $\X_s'$ by Corollary \ref{cor: comps upstairs}. We show the preimage of $y$ on $\X_m'$ consists of at least $N-\cdc(y)=\frac{9n}{5}+2$ components of $(\X'_m)_s$.

    By Proposition \ref{prop: full local resolution for type I continued}(v), (vi), and (vii), the local $n$-resolution of $D$ at $y$ consists of three blow-ups $\Y_3\to...\to\Y_0$, so that $m=3$; the multiplicity $2$ simple normal crossings points of $D_3$ are the intersection points $z_1$, $z_2$, and $z_3$ of $E_3$ with $\Gamma_3$, $E_1$, and $E_2$ respectively; the ramification pairs of $z_1,z_2,z_3$ are all $(5,5)$; and the order pairs of $z_1,z_2,z_3$ are respectively $(6\nu_0,\nu_0),(6\nu_0,2\nu_0)$, and $(6\nu_0,3\nu_0)$. It follows from Proposition \ref{prop: crossing-bonus in terms of nu0} of the Appendix that $\HJL(z_1)=\L(5,4)=4$, $\HJL(z_2)=\L(5,2)=2$, and $\HJL(z_3)=\L(5,3)=2$.

    By Proposition \ref{prop: Part1 results}, for $i\in\{1,2,3\}$, the preimage of $z_i$ on $\X_3'$ consists of exactly $\frac{n}{5}\cdot\HJL(z_i)$ components of $(\X_3')_s$. 
    By Lemma \ref{lemma: Part1 results singular points}, the preimage of $y$ on $\X_3'$ then consists of the $\frac{4n}{5}+\frac{2n}{5}+\frac{2n}{5}=\frac{8n}{5}$ components of $(\X_3')_s$ arising this way, along with the strict transforms on $\X_3'$ of the preimages of $E_1$, $E_2$, and $E_3$ on $\X_3$, which are also components of $(\X_3')_s$. Since $E_1$, $E_2,$ and $E_3$ all have ramification index $5$ by Propositions \ref{prop: structure of exceptional divisor}(vi), \ref{prop: full local resolution for type I}(iii), and \ref{prop: full local resolution for type I continued}(iii) and meet only branch components of ramification index $5$,
    Lemma \ref{lemma: etale}(iii) implies $E_1,E_2,$ and $E_3$ all lie under exactly $\frac{n}{5}$ components of $(\X_3)_s$. Since $n\geq 5$ in this case, the preimage of $y$ on $(\X_3')_s$ then consists of $\frac{8n}{5}+\frac{3n}{5}=\frac{11n}{5}\geq \frac{9n}{5}+2$ components of $(\X_3')_s$, as desired.


    
    \textbf{Case:} $y$ is of Type I, $j=2$, and $e_0=3$. Then $\cdc(y)=-2$ by Proposition \ref{prop: full positivity for type I}, and the preimage of $y$ on $\X'$ consists of $N=\frac{8n}{3}$ components of $\X_s'$ by Corollary \ref{cor: comps upstairs}. We show the preimage of $y$ on $\X_m'$ consists of at least $N-\cdc(y)=\frac{8n}{3}+2$ components of $(\X'_m)_s$.

    By Proposition \ref{prop: full local resolution for type I continued}(v), (vi), and (vii), the local $n$-resolution of $D$ at $y$ consists of four blow-ups $\Y_4\to...\to\Y_0$, so that $m=4$; the multiplicity $2$ simple normal crossings points of $D_4$ are the intersection points $z_1$, $z_2$, and $z_3$ of $E_4$ with $\Gamma_4$, $E_2$, and $E_3$ respectively, along with the intersection point $z_4\colonequals z_{0,2}$ of $E_1$ and $E_2$; the ramification pairs of $z_1,z_2,z_3,z_4$ are all $(3,3)$; and the order pairs of $z_1,z_2,z_3,z_4$ are respectively $(10\nu_0,\nu_0)$, $(10\nu_0,4\nu_0)$, $(10\nu_0,5\nu_0)$, and $(2\nu_0,4\nu_0)$. It follows from Proposition \ref{prop: crossing-bonus in terms of nu0} of the Appendix that $\HJL(z_1)=\HJL(z_2)=\L(3,2)=2$ and $\HJL(z_3)=\HJL(z_{0,2})=\L(3,1)=1$.

    By Proposition \ref{prop: Part1 results}, for $i\in\{1,2,3,4\}$, the preimage of $z_i$ on $\X_4'$ consists of exactly $\frac{n}{3}\cdot\HJL(z_i)$ components of $(\X_4')_s$. By Lemma \ref{lemma: Part1 results singular points}, the preimage of $y$ on $\X_4'$ then consists of the $\frac{2n}{3}+\frac{2n}{3}+\frac{n}{3}+\frac{n}{3}=2n$ components of $(\X_4')_s$ arising this way, along with the strict transforms on $\X_4'$ of the preimages of $E_1$, $E_2$, $E_3$, and $E_4$ on $\X_4$, which are also components of $(\X_4')_s$. Since $E_1$, $E_2,$ and $E_3$ all have ramification index $3$ by Propositions \ref{prop: structure of exceptional divisor}(vi), \ref{prop: full local resolution for type I}(iii), and \ref{prop: full local resolution for type I continued}(iii) and meet only branch components of ramification index $3$,
    Lemma \ref{lemma: etale}(iii) implies $E_1,E_2,E_3,$ and $E_4$ all lie under exactly $\frac{n}{3}$ components of $(\X_4)_s$. Since $n\geq 3$ in this case, the preimage of $y$ on $(\X_4')_s$ then consists of $2n+\frac{4n}{3}=\frac{10n}{3}\geq \frac{8n}{3}+2$ components of $(\X_4')_s$, as desired. 
\end{proof}

\begin{remark}
    One can alternatively show Proposition \ref{prop: contractions} without using the results of Section \ref{section 5} by computing the self-intersection number of each (rational) component of the preimage of $y$ on $\X_m'$. While this method yields a more involved proof of Proposition \ref{prop: contractions} than the one given above, it allows us to explicitly identify the components of $(\X_m')_s$ that are contracted by $\pi$. We summarize these results below. Keep the notation of the proof of Proposition \ref{prop: contractions}.
    
    If $y$ is of Type II, $j=1,$ and $n=e_0=e_0'=e_1=3$, so that $\cdc(y)=-1$ and $m=2$, the map $\pi$ contracts the unique irreducible component of $(\X_2')_s$ lying over $E_1$. If $y$ is of Type I, $j=1$, and $e_0=3$, so that $\cdc(y)=-\frac{n}{3}$ and $m=2$, the map $\pi$ contracts the $\frac{n}{3}$ components of $(\X_2')_s$ lying over $E_1$. If $y$ is of Type I, $j=2$, and $e_0=3$, so that $\cdc(y)=-2$ and $m=4$, the map $\pi$ contracts the $\frac{n}{3}$ components of $(\X_4')_s$ lying over $E_1$ and the $\frac{n}{3}$ components of $(\X_4')_s$ lying over $E_3$.
    
    If $y$ is of Type I, $j=1$, and $e_0=4$, so that $\cdc(y)=-2$ and $m=3$, then $E_2$ has ramification index $4$ by Proposition \ref{prop: full local resolution for type I}(vi) and meets other branch components only at the point $z_3$, hence lies under exactly $\frac{n}{4}$ components of $(\X_3)_s$ by Lemma \ref{lemma: etale}(iii). The map $\pi$ contracts the strict transforms on $(\X_3')_s$ of each of these $\frac{n}{4}$ components, as well as the $\frac{n}{4}$ components of the preimage of $z_3$ on $(\X_3')_s$.

    Finally, if $y$ is of Type I, $j=1$, and $e_0=5$, so that $\cdc(y)=-\frac{n}{5}-2$ and $m=3$, the map $\pi$ contracts the $\frac{n}{5}$ components of $(\X_3')_s$ lying over $E_1$ and the $\frac{n}{5}$ components of $(\X_3')_s$ lying over $E_2$. Moreover, there are $\frac{n}{5}$ points of $\X_3$ lying over $z_3$; for each such point $x\in\X_3$, the spectrum of the complete local ring at this point is a tame cyclic quotient singularity of type $\frac{1}{5}(1,3)$ by \cite[Proposition 7.6(iii)]{Part1}; the preimage of $x$ under $\X_3'\to\X_3$ consists of two components of $(\X_3')_s$, of self-intersection numbers $-2$ and $-3$ respectively, by \cite[Proposition 7.4]{Part1} with $b_1=2$ and $b_2=3$ and by the proof of \cite[Lemma 3.5]{ObusWewers}; and $\pi$ contracts the $-2$-component of the preimage of $x$ on $\X_3'$.
\end{remark}

\appendix

\section{Elementary Calculations}\label{appendix}

\begin{prop}\label{prop: proof of positivity for type II}
    Let $n,e_0,e_0',e_1,\alpha,\beta\in\Z$ satisfy $n,e_0,e_0'\geq 2$, $e_1\geq 1$, $\lcm(e_0,e_0',e_1)\mid n$, $\alpha\geq 0$, and $0\leq \beta\leq e_1-1$. Then the quantity $f(n,e_0,e_0',e_1,\alpha,\beta)$ given by the formula
    \begin{align*}
        \alpha(2n-\frac{2n}{e_1}&+e_1(n-2))\\
        &+ \begin{cases}
        \frac{n}{e_0}+\frac{n}{e_0'}-2&\beta=0\\
        \beta(n-2)-n-2+\frac{n}{e_0}+\frac{n}{e_0'}+\frac{n}{e_1}((\beta+1,e_1)-1)&1\leq \beta\leq e_1-2\\
        2n-\frac{2n}{e_1}+e_1(n-2)&1\leq\beta=e_1-1
        \end{cases}
    \end{align*}
    is non-negative except possibly in the following cases:
    \begin{table}[h!]
    \centering
    \renewcommand{\arraystretch}{1.4}
    \begin{tabular}{|ccc|c|}
        \hline
        & \textbf{\textup{Case}} &  & $f(n,e_0,e_0',e_1,\alpha,\beta)$ \\
        \hline
        $\alpha=0$ & $\beta=1$ & $e_1\geq 3,\; 2\mid n,\; \{e_0,e_0'\}=\{n,\frac{n}{2}\}$ & $\geq-1$ \\
        \hline
        $\alpha=0$ & $\beta=1$ & $e_1\geq 3,\; e_0=e_0'=n$ & $\geq-2$ \\
        \hline
    \end{tabular}
    \end{table}
\end{prop}
\begin{proof} We first note that since $e_1\geq 1$ and $n\geq 2$, we have
    \begin{equation}\label{eq: main positivity}
        2n-\frac{2n}{e_1}+e_1(n-2)\geq 0
    \end{equation}
    It follows immediately that $f(n,e_0,e_0',e_1,\alpha,\beta)\geq 0$ whenever $\beta=0$ or $1\leq\beta=e_1-1$.
    
    Suppose now $1\leq\beta\leq e_1-2$, so that
    \[f(n,e_0,e_0',e_1,\alpha,\beta)=\alpha(2n-\frac{2n}{e_1}+e_1(n-2))+\beta(n-2)-n-2+\frac{n}{e_0}+\frac{n}{e_0'}+\frac{n}{e_1}((\beta+1,e_1)-1)\]

    If $\beta\geq 2$, then by \eqref{eq: main positivity} and the inequality $n\geq e_1\geq \beta+2=4$,
    we have
    \begin{equation}\label{equation: b geq 2 bound}
    f(n,e_0,e_0',e_1,\alpha,\beta)\geq n-6+\frac{n}{e_0}+\frac{n}{e_0'}+\frac{n}{e_1}((\beta+1,e_1)-1)\geq \frac{n}{e_0}+\frac{n}{e_0'}-2\geq 0
    \end{equation}
    
    If $\beta=1$ and $\alpha\geq 1$, then by \eqref{eq: main positivity} and the inequality $n\geq e_1\geq\beta+2=3$, we have
    \begin{equation}\label{equation: bound b=1,a>=1}
        f(n,e_0,e_0',e_1,\alpha,\beta)\geq2n+3(n-2)-\frac{2n}{3}-2=\frac{13}{3}n-8\geq0
    \end{equation}
    
    If $\beta=1$ and $\alpha=0$, again $n\geq e_1\geq\beta+2=3$ and we have
    \begin{equation}\label{equation: bound $j=1$}
        f(n,e_0,e_0',e_1,\alpha,\beta)
        =-4+\frac{n}{e_0}+\frac{n}{e_0'}+\frac{n}{e_1}((2,e_1)-1)
        \geq-4+\frac{n}{e_0}+\frac{n}{e_0'}
    \end{equation}
    so $f(n,e_0,e_0',e_1,\alpha,\beta)\geq0$ except possibly if $e_0=e_0'=n$, in which case $f(n,e_0,e_0',e_1,\alpha,\beta)$ $\geq-2$, or if $2\mid n$ and $\{e_0,e_0'\}=\{n,\frac{n}{2}\}$, in which case $f(n,e_0,e_0',e_1,\alpha,\beta)\geq-1$.
\end{proof}

\begin{prop}\label{prop: proof of positivity for type I}
    Let $n,e_0,e_1,\alpha,\beta\in\Z$ satisfy $n,e_0\geq 2$, $e_0\mid n$, $e_1=\frac{e_0}{\gcd(2,e_0)}$, $\alpha\geq 0$, and $0\leq \beta\leq e_1-1$. Then the quantity $f(n,e_0,e_1,\alpha,\beta)$ given by the formula
    \begin{align*}
        \alpha(2n&-\frac{2n}{e_1}+e_1(n-2))+\beta(n-2)\\
        &+ \begin{cases}
            0&\beta=0\\
            -n&\beta\geq1,e_0\mid (2\beta+1)\\
            n-\frac{2n}{e_1}&\beta\geq1,e_0\nmid (2\beta+1),e_0\mid (4\beta+2)\\
            -2n+\frac{n}{e_0}(
            \gcd(2,e_0)-1)(\gcd(2\beta+1,e_0)-1)&\beta\geq1,e_0\nmid (4\beta+2)\end{cases}
    \end{align*}
    is non-negative except in the following cases: 
    \begin{table}[h!]
    \centering
    \renewcommand{\arraystretch}{1.4}
    \begin{tabular}{|ccc|c|}
        \hline
        & \textbf{\textup{Case}} &  & $f(n,e_0,e_1,\alpha,\beta)$ \\
        \hline
        $\alpha=0$ & $\beta=1$ & $e_0=3$ & $=-2$\\
        \hline
        $\alpha=0$ & $\beta=1$ & $6\nmid e_0,e_0\not\in\{2,3\}$ & $=-n-2$ \\
        \hline
        $\alpha=0$ & $\beta=1$ & $6\mid e_0,e_0\neq 6$ & $= -n-2+\frac{2n}{e_0}$ \\
        \hline
        $\alpha=0$ & $\beta=2$ & $10\nmid e_0,e_0\not\in\{2,5\}$ & $=-4$ \\
        \hline
        $\alpha=0$ & $\beta=3$ & $n=e_0=5$ & $=-1$ \\
        \hline
    \end{tabular}
    \end{table}
\end{prop}

\begin{proof}
    Since $e_1\geq 1$ and $n\geq 2$, then \eqref{eq: main positivity} holds as in the proof Proposition \ref{prop: proof of positivity for type II}. It follows immediately that $f(n,e_0,e_1,\alpha,\beta)\geq0$ whenever $\beta=0$. 
    
    \textbf{Case: $\beta\geq 1,e_0\mid(2\beta+1)$}.
    
    In this case, the quantity $e_0\geq 2$ is odd, so that $3\leq e_0=e_1\leq n$, and we have
    \[f(n,e_0,e_1,\alpha,\beta)=\alpha(2n-\frac{2n}{e_1}+e_1(n-2))+\beta(n-2)-n\]
    If $\alpha\geq 1$, then by \eqref{eq: main positivity} and the inequality $3\leq e_1\leq n$, we have
    \[f(n,e_0,e_1,\alpha,\beta)
        \geq 2n-\frac{2n}{3}+3(n-2)+\beta(n-2)-n= \beta(n-2)+\frac{10}{3}n-6
        \geq 0\]
    Suppose now $\alpha=0$, so that $f(n,e_0,e_1,\alpha,\beta)=\beta(n-2)-n$.
    
    If $\beta\geq 3$, then since $n\geq\beta$, we have \[f(n,e_0,e_1,\alpha,\beta)\geq2n-6\geq 0\]
    
    If $\beta=2$, then since $e_0\geq 3$ divides $2\beta+1=5$, we have $n\geq e_0\geq 5$, and \[f(n,e_0,e_1,\alpha,\beta)=n-4\geq0\]
    
    If $\beta=1$, we have $f(n,e_0,e_1,\alpha,\beta)=-2$, and since $e_0\geq 3$ divides $2\beta+1$, we have $e_0=3$.

    In summary, the conditions $\beta\geq 1,e_0\mid (2\beta+1)$, and $f(n,e_0,e_1,\alpha,\beta)<0$ hold simultaneously if and only if $\alpha=0$ and $\beta=1$, in which case $f(n,e_0,e_1,\alpha,\beta)=-2$ and $e_0=3$.
    
    \textbf{Case: $\beta\geq1,e_0\nmid (2\beta+1),e_0\mid (4\beta+2)$}.
     
    In this case, we have $2\mid e_0$, $e_1=\frac{e_0}{2}\leq n$, $e_1\mid(2\beta+1)$, so that $e_1$ is odd, and $e_1>\beta\geq1$, so that $e_1\geq 3$. We also have
    \[f(n,e_0,e_1,\alpha,\beta)=\alpha(2n-\frac{2n}{e_1}+e_1(n-2))+\beta(n-2)+n-\frac{2n}{e_1}\]
    By \eqref{eq: main positivity} and the inequality $3\leq e_1\leq n$, we have
    \[f(n,e_0,e_1,\alpha,\beta)\geq\beta(n-2)+\frac{n}{3}\geq0\]
    In summary, the conditions $\beta\geq1,e_0\nmid (2\beta+1),e_0\mid (4\beta+2)$ and $f(n,e_0,e_1,\alpha,\beta)<0$ never hold simultaneously.
    
    \textbf{Case: $\beta\geq1,e_0\nmid (4\beta+2)$.}
    
    In this case, we have $e_0\geq 3$, $e_1=\frac{e_0}{\gcd(2,e_0)}\geq 2$, and 
    \[f(n,e_0,e_1,\alpha,\beta)=\alpha(2n-\frac{2n}{e_1}+e_1(n-2))+\beta(n-2)-2n+\frac{n}{e_0}(\gcd(2,e_0)-1)(\gcd(2\beta+1,e_0)-1)\]
    
    If $\alpha\geq 1$, then by \eqref{eq: main positivity}, the inequalities $\beta\geq1$, $3\leq e_0\leq n$, and $2\leq e_1$, and the non-negativity of the term $\frac{n}{e_0}(\gcd(2,e_0)-1)(\gcd(2\beta+1,e_0)-1)$, we have
    \[f(n,e_0,e_1,\alpha,\beta)\geq2n+2(n-2)-n+\beta(n-2)-2n\geq 2n-6\geq0\]
    Suppose now $\alpha=0$, so that 
    \begin{align*}
    f(n,e_0,e_1,\alpha,\beta)=\beta(n-2)-2n+\frac{n}{e_0}(\gcd(2,e_0)-1)(\gcd(2\beta+1,e_0)-1)
    \end{align*}
    
    If $\beta\geq 4$, then since $n\geq e_0\geq e_1>\beta$, we have:
    \[f(n,e_0,e_1,\alpha,\beta)\geq \beta(n-2)-2n\geq 2n-8\geq0\]

    If $\beta=3$, the condition $e_0\nmid (4\beta+2)$ is equivalent to $e_0\neq 2,7$, and the inequalities $n\geq e_0\geq e_1=\frac{e_0}{\gcd(2,e_0)}>\beta$ imply $n\geq e_0\geq 5$. Then 
    \[f(n,e_0,e_1,\alpha,\beta)\geq\beta(n-2)-2n\geq n-6\]
    so that $f(n,e_0,e_1,\alpha,\beta)<0$ only if $n=e_0=5$, in which case $f(n,e_0,e_1,\alpha,\beta)=-1$.

    If $\beta=2$, the condition $e_0\nmid (4\beta+2)=10$ is equivalent $e_0\neq 2,5,10$, and we have
    \[f(n,e_0,e_1,\alpha,\beta)=-4+\begin{cases}
        0&10\nmid e_0\\
        \frac{4n}{e_0}&10\mid e_0
    \end{cases}\]
    so that $f(n,e_0,e_1,\alpha,\beta)<0$
    if and only if $10\nmid e_0$, in which case $f(n,e_0,e_1,\alpha,\beta)=-4$.

    If $\beta=1$, the condition $e_0\nmid (4\beta+2)=6$ is equivalent to $e_0\neq 2,3,6$, and we have
    \[f(n,e_0,e_1,\alpha,\beta)=-n-2+\begin{cases}
        0&6\nmid e_0\\
        \frac{2n}{e_0}&6\mid e_0
    \end{cases}\]
    so that since $e_0\geq 4$, we have $f(n,e_0,e_1,\alpha,\beta)\leq -\frac{n}{2}-2<0$.
    
    In summary, the conditions $\beta\geq1,e_0\nmid (4\beta+2),$ and $f(n,e_0,e_1,\alpha,\beta)<0$ hold simultaneously if and only if $\alpha=0,\beta=3,$ and $n=e_0=5$, in which case $f(n,e_0,e_1,\alpha,\beta)=-1$; or $\alpha=0,\beta=2,10\nmid e_0,e_0\neq 2,5$, in which case $f(n,e_0,e_1,\alpha,\beta)=-4$; or $\alpha=0,\beta=1,e_0\neq 2,3,6$, in which case $f(n,e_0,e_1,\alpha,\beta)=-n-2$ or $f(n,e_0,e_1,\alpha,\beta)=-n-2+\frac{2n}{e_0}$ according as $6\nmid e_0$ or $6\mid e_0$.
\end{proof}

\begin{prop}\label{prop: ramification index simplified}
    Let $a,\nu,n\in\N$ with $n\geq 1$. Let $e=\frac{n}{\gcd(n,\nu)}$, and let $\c{a}$ be the least residue of $a\textup{ mod }e$. Then
    \[\frac{n}{\gcd(n,a\nu)}=\frac{e}{\gcd(e,\c{a})}\]
\end{prop}
\begin{proof}
    We have
    \[\frac{e}{\gcd(e,a)}=\frac{\frac{n}{\gcd(n,\nu)}}{\gcd(\frac{n}{\gcd(n,\nu)},a)}=\frac{n}{\gcd(n,a\cdot \gcd(n,\nu))}=\frac{n}{\gcd(n,an,a\nu)}=\frac{n}{\gcd(n,a\nu)}\]
    Since $a\equiv \c{a}$ mod $e$, we also have $\gcd(e,a)=\gcd(e,\c{a})$, showing the result.
\end{proof}

\begin{prop}\label{prop: crossing-bonus in terms of nu0}
    Let $a,b,\nu_0,n\in\N$ with $n\geq 1$, let $\nu_1=a\nu_0$, $\nu_2=b\nu_0$, and for $0\leq i\leq 2$, let $e_i=\frac{n}{\gcd(n,\nu_i)}$. Then 
    \[\gcd(\frac{n}{e_1},\frac{n}{e_2})=\frac{n}{e_0}\cdot\gcd(e_0,a,b)\]
    Additionally, if $\L$ is as in Definition \ref{dfn: Hirzebruch-Jung length of a pair}, we have
    \[\L(\gcd(e_1,e_2),-(\frac{\nu_1}{\gcd(\nu_1,n)})^{-1}\frac{\nu_2}{\gcd(\nu_2,n)})=\L(\frac{e_0}{\gcd(e_0,\lcm(a,b))},-(\frac{a}{\gcd(e_0,a)})^{-1}\frac{b}{\gcd(e_0,b)})\]
\end{prop}
\begin{proof} For the first part of the proposition, we have
    \begin{equation*}
    \begin{aligned}
    \gcd(\frac{n}{e_1},\frac{n}{e_2})=\gcd(\gcd(n,a\nu_0),\gcd(n,b\nu_0))
    &=\gcd(n,a\gcd(\nu_0,n),b\gcd(\nu_0,n))
    \\&=\frac{n}{e_0}\cdot\gcd(e_0,a,b)
    \end{aligned}
    \end{equation*}
    
    Next we have
    \begin{equation}\label{eqn: gcd2}
    \begin{aligned}
        \gcd(e_1,e_2)=\gcd(\frac{n}{\gcd(n,a\nu_0)},\frac{n}{\gcd(n,b\nu_0)})&=\frac{n}{\gcd(n,\nu_0\cdot \tx{lcm}(a,b))}\\&=\frac{n}{\gcd(n,\gcd(\nu_0,n)\cdot \tx{lcm}(a,b))}\\&=\frac{e_0}{\gcd(e_0,\tx{lcm}(a,b))}
    \end{aligned}
    \end{equation}
    We also have
    \begin{equation}\label{eqn: gcd3}
    \begin{aligned}
    \frac{\nu_1}{\gcd(n,\nu_1)}=\frac{a\nu_0}{\gcd(n,a\nu_0)}=\frac{a\nu_0}{n}\cdot\frac{n}{\gcd(n,a\nu_0)}&=\frac{a\nu_0}{n}\cdot\frac{n}{\gcd(n,a\gcd(\nu_0,n))}\\
    &=\frac{a\nu_0}{n}\cdot\frac{e_0}{\gcd(e_0,a)}\\
    &=\frac{a}{\gcd(e_0,a)}\cdot\frac{\nu_0}{\gcd(n,\nu_0)}
    \end{aligned}
    \end{equation}
    and similarly
    \begin{equation}\label{eqn: gcd4}
    \begin{aligned}\frac{\nu_2}{\gcd(n,\nu_2)}=\frac{b}{\gcd(e_0,b)}\cdot\frac{\nu_0}{\gcd(n,\nu_0)}
    \end{aligned}
    \end{equation}
    
    We note the integers $\frac{b}{\gcd(e_0,b)},\frac{b}{\gcd(e_0,b)},\frac{\nu_0}{\gcd(n,\nu_0)}$ in \eqref{eqn: gcd3} and \eqref{eqn: gcd4} are units mod $\gcd(e_1,e_2)$ since they are prime to $e_0\mid n$ and since $\gcd(e_1,e_2)\mid e_0$ by \eqref{eqn: gcd2}. The second part of the proposition then follows by combining \eqref{eqn: gcd2}, \eqref{eqn: gcd3}, and \eqref{eqn: gcd4}. 
\end{proof}

\begin{prop}\label{prop: explicit hirzebruch-jung lengths}
    Let $\nu,\rho\in\Z$ with $\nu\geq1$ and $\gcd(\nu,\rho)=1$.\,Let $\L$ be as in Definition \ref{dfn: Hirzebruch-Jung length of a pair}.
    \begin{enumerate}[\upshape (i)]
        \item If $\rho\equiv1\mod \nu$, then $\L(\nu,-\rho)=\nu-1$.
        \item If $\rho\equiv2\mod\nu$, then $\L(\nu,-\rho)=\frac{\nu-1}{2}$.
        \item If $\rho\equiv3\mod\nu$, then $\L(\nu,-\rho)=\begin{cases}
            \frac{\nu-1}{3}&\nu\equiv 1\mod3\\
            \frac{\nu+1}{3}&\nu\equiv 2\mod3
        \end{cases}$.
        \item If $\rho\equiv6\mod \nu$, then $\L(\nu,-\rho)=\begin{cases}
            \frac{\nu-1}{6}&\nu\equiv 1\mod6\\
            \frac{\nu+19}{6}&\nu\equiv 5\mod6
        \end{cases}$.
    \end{enumerate}
\end{prop}

\begin{proof}

    The statements hold trivially for $\nu=1$, so assume $\nu\geq 2$. Let $k\in \Z$ satisfy $1\leq k<\frac{\nu}{2}$ and $\gcd(\nu,k)=1$. We claim 
    \[\L(\nu,\nu-k)=\L(\nu-k,\nu-2k)+1\]
    Indeed, since $1=\gcd(\nu,k)=\gcd(\nu,\nu-k)=\gcd(\nu-k,\nu-2k)$, the Hirzebruch-Jung lengths $\L(\nu,\nu-k)$ and $l\colonequals \L(\nu-k,\nu-2k)$ are well-defined. Since $0<\nu-2k<\nu-k$ by assumption, the least residue of $\nu-2k$ mod $\nu-k$ is $\nu-2k$. The sequence of integers $r_i$ defining $l$ in Definition \ref{dfn: Hirzebruch-Jung length of a pair} is then of the form $r_0=\nu-k,r_1=\nu-2k,r_2,...,r_l,r_{l+1}=0$, and the sequence defining $\L(\nu,\nu-k)$ is obtained by adding $1$ to the indices of the $r_i$ and appending $\nu$ at the start, so that $\L(\nu,\nu-k)=l+1$.

    It follows by induction that if $k\in\Z$ satisfies $1\leq k<\nu$ and $\gcd(\nu,k)=1$, and if $\mu$ is the least residue of $\nu$ mod $k$, then 
    \begin{equation}\label{eqn: inductive HJ-length}
        \L(\nu,\nu-k)=\L(k+\mu,\mu)+\frac{\nu-(k+\mu)}{k}
    \end{equation}

    By direct computation using Definition \ref{dfn: Hirzebruch-Jung length of a pair}, we have $\L(1,0)=0, \L(3,1)=\L(4,1)=\L(7,1)=1, \L(5,2)=2$, and $\L(11,5)=5$. 

    If $\rho\equiv1$ mod $\nu$, then taking $k=1$ and $\mu=0$ in \eqref{eqn: inductive HJ-length}, we obtain \[\L(\nu,-\rho)=\L(\nu,\nu-1)=\L(1,0)+\nu-1=\nu-1\]
    showing (i).

    Now suppose $\rho\equiv 2$ mod $\nu$. Since $\gcd(\nu,\rho)=1$, then $\nu$ is odd, hence $\nu\geq 3$. Taking $k=2$ and $\mu=1$ in \eqref{eqn: inductive HJ-length}, we obtain
    \[\L(\nu,-\rho)=\L(\nu,\nu-2)=\L(3,1)+\frac{\nu-3}{2}=\frac{\nu-1}{2}\]
    showing (ii).

    Now suppose $\rho\equiv3$ mod $\nu$. Since $\gcd(\nu,\rho)=1$, then $\nu=2$ or $\nu\geq 4$. If $\nu=2$, then $\L(\nu,-\rho)=\L(2,1)=1=\frac{\nu+1}{3}$. 
    If $\nu\geq 4$ and $\nu\equiv 1$ mod $3$, then taking $k=3$ and $\mu=1$ in \eqref{eqn: inductive HJ-length}, we obtain
    \[\L(\nu,-\rho)=\L(\nu,\nu-3)=\L(4,1)+\frac{\nu-4}{3}=\frac{\nu-1}{3}\]
    If $\nu\geq 4$ and $\nu\equiv 2$ mod $3$, then taking $k=3$ and $\mu=2$ in \eqref{eqn: inductive HJ-length}, we obtain
    \[\L(\nu,-\rho)=\L(\nu,\nu-3)=\L(5,2)+\frac{\nu-5}{3}=\frac{\nu+1}{3}\]
    Combining these cases, we obtain (iii).
    
    Finally, suppose $\rho\equiv6$ mod $\nu$. Since $\gcd(\nu,\rho)=1$, then $\nu=5$ or $\nu\geq 7$. If $\nu=5$, then $\L(\nu,-\rho)=\L(5,4)=4=\frac{\nu+19}{6}$. If $\nu\geq 7$ and $\nu\equiv 1$ mod $6$, then taking $k=6$ and $\mu=1$ in
    \eqref{eqn: inductive HJ-length}, we obtain
    \[\L(\nu,-\rho)=\L(\nu,\nu-6)=\L(7,1)+\frac{\nu-7}{6}=\frac{\nu-1}{6}\]
    If $\nu\geq 7$ and $\nu\equiv5$ mod $6$, then taking $k=6$ and $\mu=5$ in \eqref{eqn: inductive HJ-length}, we obtain
    \[\L(\nu,-\rho)=\L(\nu,\nu-6)=\L(11,5)+\frac{\nu-11}{6}=\frac{\nu+19}{6}\]
    Combining these cases, we obtain (iv).
\end{proof}

\begin{prop}\label{prop: cb computations}
    Let $n,e\in\N$ with $e\geq 1$, and let $\L$ be as in Definition \ref{dfn: Hirzebruch-Jung length of a pair}. Define 
    \begin{align*}
        c_1(n,e)&=n-\frac{n}{e}\cdot(\L(\frac{e}{\gcd(e,6)},-\frac{6}{\gcd(e,6)})+1)\\
        c_2(n,e)&=n-\frac{n}{e}\cdot\gcd(e,2)\cdot(\L(\frac{e}{\gcd(e,6)},-(\frac{2}{(\gcd(e,2)})^{-1}\frac{6}{\gcd(e,6)})+1)\\
        c_3(n,e)&=n-\frac{n}{e}\cdot\gcd(e,3)\cdot(\L(\frac{e}{\gcd(e,6)},-(\frac{3}{(\gcd(e,3)})^{-1}\frac{6}{\gcd(e,6)})+1)
    \end{align*}
    Then according to the residue of $e$ \textup{mod} $6$, we have
    
       
    \begin{table}[h!]
    \centering
    \renewcommand{\arraystretch}{1.4}
    \begin{tabular}{|c|c|c|c|c|}
        \hline
        & $c_1(n,e)$ & $c_2(n,e)$ & $c_3(n,e)$
        & $c_1(n,e)+c_2(n,e)+c_3(n,e)$ \\
        \hline
        $e\equiv 0\textup{ mod }6$& $\frac{5n}{6}$ & $\frac{2n}{3}$ & $\frac{n}{2}$ 
        &$2n$\\
        $e\equiv 1\textup{ mod }6$& $\frac{5n}{e}(\frac{e-1}{6})$ & $\frac{2n}{e}(\frac{e-1}{3})$ & $\frac{n}{e}(\frac{e-1}{2})$ 
        & $\frac{2n}{e}(e-1)$\\
        $e\equiv 2\textup{ mod }6$& $\frac{n}{2}+\frac{n}{e}(\frac{e-2}{3})$& $\frac{2n}{e}(\frac{e-2}{3})$ & $\frac{n}{2}$ 
        & $\frac{2n}{e}(e-1)$\\
        $e\equiv 3\textup{ mod }6$ & $\frac{2n}{3}+\frac{n}{e}(\frac{e-3}{6})$ & $\frac{2n}{3}$ & $\frac{n}{e}(\frac{e-3}{2})$ 
        & $\frac{2n}{e}(e-1)$ \\
        $e\equiv 4\textup{ mod }6$ & $\frac{n}{2}+\frac{n}{e}(\frac{e-4}{3})$ & $\frac{2n}{e}(\frac{e-4}{3})$ & $\frac{n}{2}$ 
        & $\frac{2n}{e}(e-2)$\\
        $e\equiv 5\textup{ mod }6$ & $\frac{5n}{e}(\frac{e-5}{6})$ & $\frac{2n}{e}(\frac{e-2}{3})$ & $\frac{n}{e}(\frac{e-1}{2})$ 
        & $\frac{2n}{e}(e-3)$\\
        \hline
    \end{tabular}
    \end{table}
\end{prop}
\begin{proof}
    \textbf{Case:} $e\equiv 0$ mod $6$. Then $\gcd(e,2)=2$, $\gcd(e,3)=3$, and $\gcd(e,6)=6$. Using Proposition \ref{prop: explicit hirzebruch-jung lengths} for the second equality of each line, it follows that
    \begin{align*}
    c_1(n,e)&=n-\frac{n}{e}\cdot(\L(\frac{e}{6},-1)+1)=n-\frac{n}{e}(\frac{e}{6}-1+1)=n-\frac{n}{6}=\frac{5n}{6}\\
    c_2(n,e)&=n-\frac{n}{e}\cdot2\cdot(\L(\frac{e}{6},-1)+1)=n-\frac{n}{e}\cdot2\cdot(\frac{e}{6}-1+1)=n-\frac{n}{3}=\frac{2n}{3}\\
    c_3(n,e)&=n-\frac{n}{e}\cdot3\cdot(\L(\frac{e}{6},-1)+1)=n-\frac{n}{e}\cdot3\cdot(\frac{e}{6}-1+1)=n-\frac{n}{2}=\frac{n}{2}
    \end{align*}
    Then the sum $c_1(n,e)+c_2(n,e)+c_3(n,e)$ is equal to
    \[\frac{5n}{6}+\frac{2n}{3}+\frac{n}{2}=2n\]

    \textbf{Case:} $e\equiv 1$ mod $6$. Then $\gcd(e,2)=1$, $\gcd(e,3)=1,\gcd(e,6)=1,$ and $e\equiv 1$ mod $3$. Using Proposition \ref{prop: explicit hirzebruch-jung lengths} for the second equality of each line, it follows that
    \begin{align*}
    c_1(n,e)&=n-\frac{n}{e}(\L(e,-6)+1)=n-\frac{n}{e}(\frac{e-1}{6}+1)=n-\frac{n}{e}(\frac{e+5}{6})=\frac{5n}{e}(\frac{e-1}{6})\\
    c_2(n,e)&=n-\frac{n}{e}(\L(e,-3)+1)=n-\frac{n}{e}(\frac{e-1}{3}+1)=n-\frac{n}{e}(\frac{e+2}{3})=\frac{2n}{e}(\frac{e-1}{3})\\
    c_3(n,e)&=n-\frac{n}{e}(\L(e,-2)+1)=n-\frac{n}{e}(\frac{e-1}{2}+1)=n-\frac{n}{e}(\frac{e+1}{2})=\frac{n}{e}(\frac{e-1}{2})
    \end{align*}
    Then the sum $c_1(n,e)+c_2(n,e)+c_3(n,e)$ is equal to
    \[\frac{5n}{e}(\frac{e-1}{6})+\frac{2n}{e}(\frac{e-1}{3})+\frac{n}{e}(\frac{e-1}{2})=\frac{2n}{e}(e-1)\]

    \textbf{Case:} $e\equiv 2$ mod $6$. Then $\gcd(e,2)=2$, $\gcd(e,3)=1,\gcd(e,6)=2,$ and $\frac{e}{2}\equiv 1$ mod $3$. Using Proposition \ref{prop: explicit hirzebruch-jung lengths} for the second equality of each line, it follows that
    \begin{align*}
        c_1(n,e)&=n-\frac{n}{e}\cdot(\L(\frac{e}{2},-3)+1)=n-\frac{n}{e}\cdot(\frac{\frac{e}{2}-1}{3}+1)=n-\frac{n}{e}(\frac{e}{6}+\frac{2}{3})=\frac{n}{2}+\frac{n}{e}(\frac{e-2}{3})\\
        c_2(n,e)&=n-\frac{n}{e}\cdot2\cdot(\L(\frac{e}{2},-3)+1)=n-\frac{n}{e}\cdot2\cdot(\frac{\frac{e}{2}-1}{3}+1)=n-\frac{n}{e}(\frac{e}{3}+\frac{4}{3})=\frac{2n}{e}(\frac{e-2}{3})\\
        c_3(n,e)&=n-\frac{n}{e}\cdot(\L(\frac{e}{2},-1)+1)=n-\frac{n}{e}\cdot(\frac{e}{2}-1+1)=n-\frac{n}{2}=\frac{n}{2}
    \end{align*}
    Then the sum $c_1(n,e)+c_2(n,e)+c_3(n,e)$ is equal to
    \[\frac{n}{2}+\frac{n}{e}(\frac{e-2}{3})+\frac{2n}{e}(\frac{e-2}{3})+\frac{n}{2}=\frac{2n}{e}(e-1)\]

    \textbf{Case:} $e\equiv 3$ mod $6$. Then $\gcd(e,2)=1$, $\gcd(e,3)=3,$ and $\gcd(e,6)=3$. Using Proposition \ref{prop: explicit hirzebruch-jung lengths} for the second equality of each line, it follows that
    \begin{align*}
        c_1(n,e)&=n-\frac{n}{e}\cdot(\L(\frac{e}{3},-2)+1)=n-\frac{n}{e}\cdot(\frac{\frac{e}{3}-1}{2}+1)=n-\frac{n}{e}(\frac{e}{6}+\frac{1}{2})=\frac{2n}{3}+\frac{n}{e}(\frac{e-3}{6})\\
        c_2(n,e)&=n-\frac{n}{e}\cdot(\L(\frac{e}{3},-1)+1)=n-\frac{n}{e}\cdot(\frac{e}{3}-1+1)=n-\frac{n}{3}=\frac{2n}{3}\\
        c_3(n,e)&=n-\frac{n}{e}\cdot3\cdot(\L(\frac{e}{3},-2)+1)=n-\frac{n}{e}\cdot 3\cdot(\frac{\frac{e}{3}-1}{2}+1)=n-\frac{n}{e}(\frac{e+3}{2})=\frac{n}{e}(\frac{e-3}{2})
    \end{align*}
    Then the sum $c_1(n,e)+c_2(n,e)+c_3(n,e)$ is equal to
    \[\frac{2n}{3}+\frac{n}{e}(\frac{e-3}{6})+\frac{2n}{3}+\frac{n}{e}(\frac{e-3}{2})=\frac{2n}{e}(e-1)\]

    \textbf{Case:} $e\equiv 4$ mod $6$. Then $\gcd(e,2)=2$, $\gcd(e,3)=1,\gcd(e,6)=2,$ and $\frac{e}{2}\equiv 2$ mod $3$. Using Proposition \ref{prop: explicit hirzebruch-jung lengths} for the second equality of each line, it follows that
    \begin{align*}
        c_1(n,e)&=n-\frac{n}{e}\cdot(\L(\frac{e}{2},-3)+1)=n-\frac{n}{e}\cdot(\frac{\frac{e}{2}+1}{3}+1)=n-\frac{n}{e}(\frac{e}{6}+\frac{4}{3})=\frac{n}{2}+\frac{n}{e}(\frac{e-4}{3})\\
        c_2(n,e)&=n-\frac{n}{e}\cdot2\cdot(\L(\frac{e}{2},-3)+1)=n-\frac{n}{e}\cdot2\cdot(\frac{\frac{e}{2}+1}{3}+1)=n-\frac{n}{e}(\frac{e}{3}+\frac{8}{3})=\frac{2n}{e}(\frac{e-4}{3})\\
        c_3(n,e)&=n-\frac{n}{e}\cdot(\L(\frac{e}{2},-1)+1)=n-\frac{n}{e}\cdot(\frac{e}{2}-1+1)=n-\frac{n}{2}=\frac{n}{2}
    \end{align*}
    Then the sum $c_1(n,e)+c_2(n,e)+c_3(n,e)$ is equal to
    \[\frac{n}{2}+\frac{n}{e}(\frac{e-4}{3})+\frac{2n}{e}(\frac{e-4}{3})+\frac{n}{2}=\frac{2n}{e}(e-2)\]

    \textbf{Case:} $e\equiv 5$ mod $6$. Then $\gcd(e,2)=1$, $\gcd(e,3)=1,\gcd(e,6)=1,$ and $e\equiv 2$ mod $3$. Using Proposition \ref{prop: explicit hirzebruch-jung lengths} for the second equality of each line, it follows that
    \begin{align*}
    c_1(n,e)&=n-\frac{n}{e}(\L(e,-6)+1)=n-\frac{n}{e}(\frac{e+19}{6}+1)=n-\frac{n}{e}(\frac{e+25}{6})=\frac{5n}{e}(\frac{e-5}{6})\\
    c_2(n,e)&=n-\frac{n}{e}(\L(e,-3)+1)=n-\frac{n}{e}(\frac{e+1}{3}+1)=n-\frac{n}{e}(\frac{e+4}{3})=\frac{2n}{e}(\frac{e-2}{3})\\
    c_3(n,e)&=n-\frac{n}{e}(\L(e,-2)+1)=n-\frac{n}{e}(\frac{e-1}{2}+1)=n-\frac{n}{e}(\frac{e+1}{2})=\frac{n}{e}(\frac{e-1}{2})
    \end{align*}
    Then the sum $c_1(n,e)+c_2(n,e)+c_3(n,e)$ is equal to
    \[\frac{5n}{e}(\frac{e-5}{6})+\frac{2n}{e}(\frac{e-2}{3})+\frac{n}{e}(\frac{e-1}{2})=\frac{2n}{e}(e-3)\qedhere\]
\end{proof}

\end{document}